\documentclass{amsart}
\usepackage{amssymb,amsmath,amsthm}
\usepackage{mathrsfs}
\usepackage{hyperref}
\usepackage{mathtools}
\usepackage{color}
\usepackage{verbatim}

\mathtoolsset{showonlyrefs}

\renewcommand{\[}{\begin{equation}\begin{aligned}}
\renewcommand{\]}{\end{aligned} \end{equation}}
\renewcommand{\l}{\ell}

\newtheorem{thm}{Theorem}
\newtheorem{prop}[thm]{Proposition}
\newtheorem{lemma}[thm]{Lemma}

\newtheorem{cor}[thm]{Corollary}

\theoremstyle{remark}
\newtheorem{remark}[thm]{Remark}

\theoremstyle{definition}
\newtheorem{definition}[thm]{Definition}

\numberwithin{equation}{section}
\numberwithin{thm}{section}

\author{G\'abor Sz\'ekelyhidi}
\address{Department of Mathematics, University of Notre Dame, Notre Dame, IN 46556}
\email{gszekely@nd.edu}

\title{Minimal hypersurfaces with cylindrical tangent cones}
\date{}

\begin{document}

\begin{abstract}
First we construct minimal hypersurfaces $M\subset\mathbf{R}^{n+1}$
in a neighborhood of the origin, with an isolated singularity
but cylindrical tangent cone $C\times \mathbf{R}$,
for any strictly minimizing strictly stable cone $C$
in $\mathbf{R}^n$. We show that many of these hypersurfaces are
area minimizing. Next, we prove a strong unique continuation result for
minimal hypersurfaces $V$ with such a cylindrical tangent cone, stating that if the
blowups of $V$ centered at the origin
approach $C\times \mathbf{R}$ at infinite order,  
then $V = C\times\mathbf{R}$ in a neighborhood of the origin.
Using this we show that for 
quadratic cones $C = C(S^p \times S^q)$, in dimensions $n > 8$,
all $O(p+1) \times O(q+1)$-invariant minimal hypersurfaces with tangent
cone $C\times \mathbf{R}$ at the 
origin are graphs over one of the surfaces that we constructed.
In particular such an invariant minimal hypersurface is
either equal to $C\times \mathbf{R}$ or has an isolated singularity at
the origin. 
\end{abstract}
\maketitle
\setcounter{tocdepth}{1}
\tableofcontents

\section{Introduction}
Let $M \subset \mathbf{R}^{n+1}$ be a minimal hypersurface with 
$0\in M$. The infinitesimal behavior of $M$ at $0$ is captured by its
tangent cones, obtained as subsequential limits of the sequence of
blowups $2^kM$ as $k\to\infty$, and a basic problem is to relate the behavior of
$M$ near $0$ to that of its tangent cones.
Apart from the case when $M$ is smooth near 0, the most completely
understood situation is when $M$ has a multiplicity one tangent cone
$C$ at 0, which is smooth away from the origin. In this case
Allard-Almgren~\cite{AA81} and Simon~\cite{Simon83} showed that the
tangent cone $C$ is unique, and $M$ can be written as a graph over $C$
near the origin. In particular it follows that $M$ itself has an isolated
singularity at 0 just like its tangent cone. 

In this paper we will be concerned with the next simplest situation,
when $M$ has a multiplicity one cylindrical tangent cone $C\times\mathbf{R}$ at the
origin. We assume that $C$ is smooth away from the origin and in
addition is strictly stable and strictly minimizing in the sense of
Hardt-Simon~\cite{HS85}.
For a large class of
such $C$ Simon~\cite{Simon94} showed that the corresponding
cylindrical tangent cone is unique (see also \cite{Sz20}), and our
goal is to give further information about the behavior of $M$ near
0. Our first result is that it is possible for $M$ to have an isolated
singularity at the origin, even if it is area minimizing. The singular
set of $M$ can therefore be quite different from the singular set of
the tangent cone. 

\begin{thm}\label{thm:T1}
  Let $C\subset\mathbf{R}^n$ be a strictly stable and strictly minimizing cone,
  which is smooth away from 0. Then there exist area
  minimizing hypersurfaces $M\subset \mathbf{R}^{n+1}$ in a
  neighborhood of the origin such that $M$ is smooth away from 0, but
  has tangent cone $C\times\mathbf{R}$ at 0.
\end{thm}

The construction of minimal hypersurfaces with cylindrical tangent
cones is very much analogous to the author's earlier work~\cite{Sz17} on singular
Calabi-Yau metrics with isolated singularities, and we will use quite
similar techniques. The main novelty is showing that we
obtain area minimizing examples.
 This is similar in spirit
to the work of Hardt-Simon~\cite{HS85}, who showed that minimal
hypersurfaces with strictly minimizing tangent cone $C$ are area
minimizing. In our setting a major difficulty is that the geometry of our hypersurfaces
degenerates as we approach the singular line of the tangent cone
$C\times\mathbf{R}$, and we need to employ a delicate
construction of barrier surfaces, using detailed information about the
geometry of the hypersurfaces. 
Simon~\cite{Simon21} has also
recently constructed stable minimal hypersurfaces with cylindrical
tangent cones using the
quadratic cones $C(S^p\times S^q)$, which moreover can have
essentially arbitrary singular sets, however for this one must allow for a perturbation of
the Euclidean metric on $\mathbf{R}^{n+1}$. It is not known whether
these examples are area minimizing.

Our second main result is that if $M$ approaches its tangent cone
$C\times\mathbf{R}$ at infinite order, then necessarily $M = C\times\mathbf{R}$ in a
neighborhood of the origin. More precisely we have the following.
\begin{thm}\label{thm:T2}
  Let $M$ be an $n$-dimensional stationary integral varifold in a
  neighborhood of $0\in \mathbf{R}^{n+1}$, such that
  $C\times\mathbf{R}$ is a multiplicity one tangent cone of $M$ at
  0. Suppose that in the $L^2$-distance $M$ approaches
  $C\times\mathbf{R}$ at infinite order, i.e. for all $k > 0$ we have
  \[ \int_{M\cap B(0,\rho)} d^2\, d\mu_M = O(\rho^k) \]
  as $\rho\to 0$, where $d$ is the distance function to
  $C\times\mathbf{R}$ and $\mu_M$ is the weight measure of $M$.
  Then $M = C\times\mathbf{R}$ in a neighborhood of the origin. 
\end{thm}

This is a strong unique continuation principle for stationary varifolds
with tangent cone $C\times\mathbf{R}$. Our approach is similar in
spirit to other well known unique continuation results, for Q-valued
harmonic functions studied by Almgren~\cite{Alm00}, or general
elliptic equations by Garofalo-Lin~\cite{GL86}, in the sense that we
derive a suitable doubling condition for $M$. A key difficulty is
that in our setting the problem cannot be reduced to a scalar PDE due
to the presence of the singular line in the tangent cone
$C\times\mathbf{R}$. As a result we are not able to define an analog of the
frequency function, which is the basic tool in previous works.
Instead we ``discretize'' the usual argument, relying
on a quantitative three annulus lemma in place of monotonicity of the
frequency. 

It is worth noting that in Theorem~\ref{thm:T2} it is crucial that on
$\mathbf{R}^{n+1}$ we use the Euclidean metric. Indeed, the examples
of Simon~\cite{Simon21} show that for smooth perturbations of the
Euclidean metric it is possible to construct stable minimal
hypersurfaces approaching the tangent cone $C\times\mathbf{R}$
exponentially fast. It is unclear whether this is possible for a real
analytic metric on $\mathbf{R}^{n+1}$. 

A consequence of the unique continuation result is that if $M$ differs
from $C\times\mathbf{R}$, then we can extract a non-zero Jacobi field
on $C\times\mathbf{R}$, which governs the next leading order behavior of
$M$ beyond its tangent cone. In certain cases knowledge of this Jacobi
field can give us further information about the structure of $M$ near
0. When $C = C(S^p\times S^q)$, then we can restrict the possible
Jacobi fields by assuming that $M$ is $O(p+1)\times
O(q+1)$-invariant. Using this leads to our third main result.
\begin{thm}\label{thm:T3}
  Let $M$ be a codimension one, $O(p+1)\times O(q+1)$-invariant
  stationary integral varifold in a
  neighborhood of $0\in
  \mathbf{R}^{p+1}\times\mathbf{R}^{q+1}\times\mathbf{R}$. Suppose that
  $M$ has multiplicity one tangent cone $C\times \mathbf{R}$ at 0, with
  $C = C(S^p\times S^q)$ for $p+q > 6$. Then in a neighborhood of the
  origin either $M=C\times\mathbf{R}$, or $M$ is a graph over a
  scaling of one of
  the hypersurfaces constructed in Theorem~\ref{thm:T1} and in
  particular it has an isolated singularity at 0. 
\end{thm}

Once again this result should be contrasted with the examples
constructed by Simon~\cite{Simon21}, whose singular set can be an
arbitrary compact subset of $\mathbf{R}$, even under a symmetry
assumption. Because of this it is crucial that we use the Euclidean
metric on $\mathbf{R}^{n+1}$. Note also that we excluded the two
lowest dimensional quadratic cones $C(S^3\times S^3)$ and $C(S^2\times
S^4)$. The reason is that $\phi = y^3 r^{-2} - y$ defines a degree one
Jacobi field on both, where $r$ is the distance from the origin in
$\mathbf{R}^n$ while $y$ is the coordinate on the $\mathbf{R}$
factor. To extend the proof of Theorem~\ref{thm:T3} to these cones, we would need a
construction of minimal hypersurfaces near the origin, similar to
those in Theorem~\ref{thm:T1}, modeled on the
Jacobi field $\phi$. While we expect that this is possible, following
the approach of Adams-Simon~\cite{AS88}, it involves additional
difficulties. More or less the same difficulty also appears in proving
the uniqueness of the corresponding cylindrical tangent cones (see
\cite{Simon94} and \cite{Sz20}).

It would be very desirable to remove the symmetry assumption in
Theorem~\ref{thm:T3}, however it is not entirely clear what the
correct expectation is. For instance it is possible to have $M$ with a
curve
of singularities other than a line, see Smale~\cite{Sm91}. The most optimistic
conjecture would be that if $C\times\mathbf{R}$ is a
multiplicity one tangent cone of $M$ at the origin,  then in a
neighborhood of 0 either the singularity is isolated, or the singular
set of $M$ is a real analytic curve through the origin. At the moment,
however, very little seems to be known in this direction beyond
general rectifiability results for the singular set (see
Simon~\cite{Simon93, Simon93_1}, Naber-Valtorta~\cite{NV1}).
As a related result we should also mention Chang's
result~\cite{Ch88}, building on Almgren's work~\cite{Alm00},
stating that the singular set of a two dimensional area minimizing
current in a Riemannian manifold is discrete.
See also Rivi\`ere-Tian~\cite{RT04}, and in
particular De Lellis-Spadaro-Spolaor~\cite{DSSI} in connection with
this result.

We conclude the introduction with an overview of the rest of the
paper. In Section~\ref{sec:preliminaries} we will review some basic
background on strictly stable and strictly minimizing cones and on
Jacobi fields on $C\times\mathbf{R}$. A new result in this section is
Proposition~\ref{prop:QL23annulus}, which is a 
quantitative version of the $L^2$ three annulus lemma for Jacobi
fields. In addition we will construct
barrier surfaces near $C\times\mathbf{R}$, to be used in maximum
principle arguments later. The construction is similar to what we used in
\cite{Sz20}, but here we need effective control of the constants
that appear and we consider general strictly stable strictly
minimizing cones $C$. We prove Theorem~\ref{thm:T1} in
Section~\ref{sec:gluing}, closely following our earlier work
\cite{Sz17} on singular Calabi-Yau
metrics to construct the hypersurfaces, and using the barrier surfaces
to show that they are area minimizing. 
In Section~\ref{sec:uniquecont} we prove the strong unique
continuation principle, Theorem~\ref{thm:T2}. The idea is to pass
the quantitative $L^2$ three annulus lemma for Jacobi fields on $C\times\mathbf{R}$
to a statement about minimal surfaces close to
$C\times\mathbf{R}$. The main difficulty is to control the behavior of
$M$ near the singular ray of the tangent cone, which we achieve by using a
quantitative version of the non-concentration estimate proved in
\cite{Sz20}. Finally in Section~\ref{sec:symmetry} we prove
Theorem~\ref{thm:T3}. The idea is to prove that under the assumptions
the hypersurface $M$ approaches one of the hypersurfaces $T$
constructed in Theorem~\ref{thm:T1} sufficiently fast as we approach
the origin. In particular we show that $M$ approaches $T$ more quickly than the
rate at which $T$ approaches $C\times\mathbf{R}$, and from this it
follows that near the origin $M$ is a graph over $T$. 

\subsection*{Acknowledgements} I am grateful to Nick Edelen and Luca
Spolaor for stimulating discussions. This work was supported in part by NSF grant DMS-1906216.

\section{Preliminary results}\label{sec:preliminaries}
For background on minimal hypersurfaces and varifolds we refer to
Simon~\cite{SimonGMT}.  Throughout the paper, 
on $\mathbf{R}^n\times \mathbf{R}$ we use coordinates
  $x\in\mathbf{R}^n$ and $y\in \mathbf{R}$. We write $r = |x|$ and
  $\rho = (r^2 + y^2)^{1/2}$. 

  We let $C\subset\mathbf{R}^n$ be a strictly minimizing and strictly
stable cone in the sense of Hardt-Simon~\cite{HS85}, which is smooth
away from 0. Recall that
by \cite{HS85} there are minimal hypersurfaces $H_-, H_+$ contained in the
two connected components of $\mathbf{R}^n \setminus C$, such that the
scalings $\lambda H_-, \lambda H_+$ for $\lambda > 0$ together with
$C$ foliate $\mathbf{R}^n$.

For an oriented hypersurface $S$ in a Riemannian
manifold $N$ we denote by $L_S$ the Jacobi operator on $S$, i.e. the
linearization of the mean curvature operator on graphs over $S$. When
$S$ is a minimal hypersurface, then we have
\[ L_S f = \Delta_S f + (|A_S|^2 + \mathrm{Ric}_N(\nu, \nu))f, \]
where $A_S$ is the second fundamental form of $S$, $\nu$ is the unit
normal vector field to $S$ and $\mathrm{Ric}_N$ is the Ricci tensor of
$N$. We will have either $N=\mathbf{R}^k$ with $\mathrm{Ric}_N=0$, or
$N=S^{k-1}$ with $\mathrm{Ric}_N(\nu,\nu) = (k-2)$.

We denote by $\phi_i$ the $i^{th}$ eigenfunction of
$-L_\Sigma$ on the link $\Sigma = C\cap \partial B_1$, with eigenvalue
$\lambda_i$. Corresponding to these there are homogeneous Jacobi
fields $r^{-\gamma_i}\phi_i$ on $C$, where
\[ \label{eq:gl1} \gamma_i^2 - (n-3)\gamma_i - (n-2+\lambda_i) = 0. \]
The strict stability condition implies that we can take $-\gamma_i >
\frac{3-n}{2}$, and that there are
no homogeneous Jacobi fields on $C$ with degrees in the interval
$(3-n + \gamma_1, -\gamma_1)$.
We set $\gamma=\gamma_1$, and assume that $\phi_1 >
0$. 

The assumptions that $C$ is stritly minimizing and strictly stable
imply that outside of a large ball, the surfaces
$H_{\pm}$ are graphs of functions
\[ \Psi_\pm =  \pm r^{-\gamma} \phi_1 + v_{\pm} \]
over $C$.  Here $v_{\pm} =
O(r^{-\gamma - c})$ for some $c > 0$, and we will assume that $c$ is
small. 
We will use the following conventions for
the orientations of $H_-, C, H_+$: on $H_-$ the normal points towards
$C$, on $C$ it points towards $H_+$, and on $H_+$ it points away from
$C$. This naturally extends to orientations of all the scalings
$\lambda H_\pm$. Note that these orientations are consistent with our
convention that outside of a large ball $H_\pm$ are the graphs of
$\Psi_\pm$, where $\Psi_+ > 0$ and $\Psi_- < 0$. 
It will be convenient to combine the foliation
of $\mathbf{R}^n$ into a single family of hypersurfaces $H(t)$ for
$t\in \mathbf{R}$.
\begin{definition}\label{defn:Ht}
  For $t\in \mathbf{R}$ we will write
  \[ H(t) = \begin{cases} |t|^{\frac{1}{\gamma+1}}H_+, &\text{ for }t > 0, \\
      C, &\text{ for }t=0, \\
      |t|^{\frac{1}{\gamma+1}}H_-, &\text{ for }t < 0.
    \end{cases} \]
\end{definition}

  Note that the scaling ensures that for a constant $C_1>0$ depending on
  the cone $C$, and for any $t\in \mathbf{R}$, the
  hypersurface $H(t)$ is the graph of the
  function
  \[ f_t(x)=  |t|^{\frac{1}{\gamma+1}}
    \Psi_\pm(|t|^{-\frac{1}{\gamma+1}}x) \] over $C$ on the region $r >
C_1 |t|^{1 / (\gamma+1)}$ , where the $\pm$ sign depends on the sign of $t$. 
 The function $f_t$ satisfies
  \[ |f_t - tr^{-\gamma} \phi_1| \leq C_1|t|^{1 +
      \frac{c}{\gamma+1}}r^{-\gamma-c}. \]
  Thus, roughly speaking we can think of $H(t)$ as being the graph of $t
  r^{-\gamma}\phi_1$ over $C$, at least on the region where $|t|\ll
  r^{\gamma+1}$. 

We will need the following result, comparing nearby leaves of the
foliation to each other. We state the result for $H_+$ but the
analogous result also holds for $H_-$, as long as we switch the orientation.
\begin{prop}\label{prop:Hestimates}
  The function $v_+$ satisfies the derivative estimates
  \[ \label{eq:nv+}|\nabla^i v_+| \leq C_i r^{-\gamma-c-i}. \]
  For sufficiently small $\epsilon$ the surface $(1+\epsilon)H_+$
  is the graph of a function $\Phi_{+,\epsilon}$ over $H_+$, which we can write as
  \[ \Phi_{+,\epsilon}(x) = \epsilon \Phi_+(x) + \epsilon^2 V_{+,\epsilon}(x). \]
  Here $\Phi_+ > 0$ is a Jacobi field on $H_+$, i.e. $L_{H_+}\Phi_+ = 0$, and
  $V_{+,\epsilon}$ satisfies the estimates
  \[ |\partial_\epsilon^j\nabla^i V_{+,\epsilon}| \leq C_{i,j} r^{-\gamma-i}. \]
\end{prop}
\begin{proof}
 The derivative estimates for $v_+$ follow from a rescaling argument
 and elliptic regularity. More precisely, for large $R > 0$ consider
 the scaled down surface $R^{-1}H_+$. On the annulus $r\in (1/2,4)$ 
this is the graph over $C$ of the
 function
 \[ f(x) = R^{-1-\gamma} r^{-\gamma} \phi_1(x/r) + R^{-1} v_+(Rx), \]
 so in particular we have $|f| \leq C_2
 R^{-1-\gamma}$. Since $C$ and $R^{-1}H_+$ are minimal, we first find
 that the derivatives of $f$ satisfy the same estimate once $R$ is
 sufficiently large, using the regularity theory for minimal graphs. 
By expanding the minimal surface equation as in \eqref{eq:mXQ} we
 obtain an estimate of the form
 \[ |L_C f|_{C^{k,\alpha}} \leq C_3 |f|_{C^{k+2,\alpha}}^2\]
 on the annulus $r\in (1,2)$. Using that $r^{-\gamma}\phi_1$ is a
 Jacobi field, this implies
 \[ |L_C (R^{-1}v_+(Rx))|_{C^{k,\alpha}} \leq C_3 R^{-2-2\gamma}, \]
 for a larger $C_3$. We know that on this annulus $|R^{-1}v_+(Rx)|
 \leq C_4 R^{-1-\gamma-c}$, and at the same time $2+2\gamma >
 1+\gamma+c$. Then using the Schauder estimates we find that on a
 smaller annulus we have 
 \[ |R^{-1}v_+(Rx)|_{C^{k+2,\alpha}} \leq C_5 R^{-1-\gamma-c}. \]
 Scaling back, this implies \eqref{eq:nv+}.

 On any fixed ball, the estimates in the second claim follow from the
 fact that the surfaces $(1+\epsilon)H_+$ vary analytically in
 $\epsilon$, for sufficiently small $\epsilon$.
 The estimates for large $r$ can be seen by writing both $(1+\epsilon)H_+$ and
 $H_+$ as graphs over $C$.
\end{proof}

The following result is analogous to \cite[Lemma 5.1]{Sz20}, with an
essentially identical proof.
\begin{lemma}\label{lem:L51}
  There is a $c_0 > 0$ depending on the cone $C$ with the following
  property. For any $t\in \mathbf{R}$ and $\lambda > 0$, the
  hypersurface $H(t + \lambda)$ is on the positive side of the graph
  of the function $c_0\min\{\lambda r^{-\gamma}, r\}$ over $H(t)$. 
\end{lemma}

The link of $C\times\mathbf{R}$ is singular, with two singularities
modeled on the cone $C$. We will only be
interested in Jacobi fields $u$ for which $r^{\gamma + \kappa}u$ is
locally bounded away from the origin for a small $\kappa > 0$. Since
there are no homogeneous Jacobi fields on $C$ with growth rate in
$(3-n+\gamma, -\gamma)$, such $u$ automatically satisfies that $r^\gamma
u$ is locally bounded away from the origin, if $\kappa$ is sufficiently small. Equivalently
the Jacobi fields that we are interested in can be characterized as
those that are in $W^{1,2}_{loc}$ away from the origin.

We need the following $L^2$ to $L^\infty$ estimate for such Jacobi fields
on $C\times\mathbf{R}$, whose proof is identical to that in
\cite{Sz20}. 
\begin{lemma}\label{lem:L2Linfty}
 Let $u$ be a Jacobi field on $C\times \mathbf{R}$, such that
 $r^\gamma u$ is in $L^\infty$ on $B_1(0)$. Then we have the estimate
  \[ \sup_{B_{1/2}(0)} |r^\gamma u| \leq C \Vert u \Vert_{L^2(B_1)}. \]
\end{lemma}

We will also need the following $L^2$ three annulus lemma, due to
Simon~\cite[Lemma 2]{Simon83} (see also Lemma 3.3 in \cite{Simon85}).
As in \cite{Sz20} this holds on the singular
cone $C\times \mathbf{R}$ as well, since it is a consequence of
spectral decomposition on the link and our assumption that $r^\gamma u$ is
locally bounded ensures that $u$ is in $W^{1,2}$ on the link. For a given
$\rho_0 > 0$ let us use $\Vert u\Vert_{\rho_0, i}$ to denote the following
$L^2$-norm on an annulus: 
  \[ \label{eq:L2ann} \Vert u\Vert_{\rho_0, i}^2 = \int_{(C\times \mathbf{R})\cap
      (B_{\rho_0^{i}}\setminus B_{\rho_0^{i+1}})} |u|^2 \rho^{-n}, \]
  in terms of $n=\dim C\times\mathbf{R}$. Note that for a homogeneous
  degree zero function $u$ the norm $\Vert u\Vert_{\rho_0, i}$ is independent of
  $i$.

  \begin{lemma}\label{lem:L23annulus}
    Given $d\in\mathbf{R}$, there are small $\alpha_0' > \alpha_0 > 0$ and
    $\rho_0 > 0$ satisfying the following. Let $u$ be a Jacobi field
    on the cone $C\times \mathbf{R}$, defined in the annulus
    $B_1\setminus B_{\rho_0^3}$, such that $r^\gamma u \in
    L^\infty$. Then we have:
  \begin{itemize}
  \item[(i)]  If $\Vert u\Vert_{\rho_0,1} \geq \rho_0^{d-\alpha_0} \Vert u\Vert_{\rho_0,0}$, then $\Vert u\Vert_{\rho_0,2} \geq \rho_0^{d-\alpha_0'} \Vert u\Vert_{\rho_0,1}$.
  \item[(ii)] If $\Vert u\Vert_{\rho_0,1} \geq \rho_0^{-d-\alpha_0} \Vert u\Vert_{\rho_0,2}$, then $\Vert u\Vert_{\rho_0,0}\geq \rho_0^{-d-\alpha_0'} \Vert u\Vert_{\rho_0,1}$. 
  \end{itemize}
If in addition $u$ has no degree $d$ component then the conclusion of either (i) or (ii) must hold. 
\end{lemma}
Note that in our notation \cite[Lemma
3.3]{Simon85} deals with the situation where $d=0$ and $\alpha_0 =
\alpha_0' = \alpha$, however the same proof leads to this slightly
more general result. 

Suppose now that $u$ is a Jacobi field for which $r^\gamma u$ is bounded on the ball
$B_1(0)\subset C\times\mathbf{R}$, not just on an annulus.
From Simon~\cite{Simon94} it follows that such a Jacobi field can be expanded as a sum
\[ \label{eq:usum1} u = \sum_{i > 0,\, k,l \geq 0} a_{i,k,l} r^{2k}
  y^l r^{-\gamma_i}\phi_i, \]
in terms of the eigenfunctions of $-L_\Sigma$. 
We will need the following immediate consequence of this for the case of
quadratic cones $C = C(S^p\times S^q)$, using that the only
$O(p+1)\times O(q+1)$-invariant eigenfunction on $S^p\times S^q$ is
the constant function. 
\begin{lemma}\label{lem:CinvJ}
  Let $u$ be an $O(p+1)\times O(q+1)$-invariant Jacobi
  field on $B_1\subset C\times \mathbf{R}$ with $r^\gamma u \in L^\infty$, where $C =
  C(S^p\times S^q)$. Then 
  \[ u = \sum_{k,\l\geq 0} a_{k,\l} r^{2k-\gamma} y^\l. \]
\end{lemma}

In general we can write such a Jacobi field $u$, with $r^\gamma u\in
L^\infty(B_1(0))$, as a sum
\[ \label{eq:uexpand2} u = \sum_{k\geq 1} c_k \rho^{-\mu_k} \Phi_k, \]
where $\Phi_k$ is the $k^{th}$ eigenfunction of
$-L_{(C\times\mathbf{R})\cap \partial B_1}$ on the link of $C\times
\mathbf{R}$, acting on $W^{1,2}$, with eigenvalue $\sigma_k$. 
We have $\mu_1 = \gamma$ and $\Phi_1
= \phi_1$ in our notation above, while $\mu_k \leq \mu_1$ for $k\geq
1$. The eigenvalues $\sigma_k$ and the
growth rates $\mu_k$ are related by an equation like \eqref{eq:gl1},
with $n$ replaced by $n+1$, and so
\[ \label{eq:muksigmak} -\mu_k = \frac{2-n}{2} + \sigma_k^{1/2}\left( 1 +
    \frac{n^2}{4\sigma_k}\right)^{1/2} = \sigma_k^{1/2} + O(1) \]
as $k\to\infty$.

Note that when $r^\gamma u$ is bounded on $B_1(0)$, then it is well
known that the function
\[ t \mapsto \ln \int_{B_{e^{-t}}(0)} |u|^2 \]
is convex in $t$, and linear only if $u$ is homogeneous. This follows
from the fact that the terms in the expansion \eqref{eq:uexpand2} with
different growth rates are $L^2$-orthogonal as used in the proof of the
following Proposition, and it is a variant of the
monotonicity of the frequency for harmonic functions. The
following is a quantitative version of this statement. 
\begin{prop}\label{prop:QL23annulus}
  Let $u$ be a Jacobi field on $B_1\subset C\times \mathbf{R}$,
  satisfying that $|r^\gamma u|$ is bounded. Let $\lambda_0\in
  (0,1/2)$. There is a constant $A > 0$ and a choice of $\lambda\in
  [\lambda_0/2, \lambda_0]$, depending only on the cone $C$ with the
  following property. Suppose that
  \[\label{eq:ua1} \Vert u\Vert_{L^2(B_1)} &\leq 2, \\
      \Vert u\Vert_{L^2(B_{e^{-2\lambda}})} &\leq \frac{1}{2}
      e^{-(n+2)\lambda}. \]
    Then we have
    \[ \label{eq:ua2} \Vert u\Vert_{L^2(B_{e^{-\lambda}})} \leq (1 - \lambda^A)
      e^{-\frac{n+2}{2}\lambda}. \]
  \end{prop}
Note that if we had equality in \eqref{eq:ua1}, it would say that
passing from $B_{e^{-2\lambda}}$ to $B_1$, the function $u$ has growth
rate $\log_{e^{2\lambda}}(4e^{(n+2)\lambda}) \sim \lambda^{-1}$. The
result is saying that by choosing $\lambda$ appropriately, we can
ensure that this growth rate is separated from the possible growth
rates of Jacobi fields sufficiently to deduce the strict convexity
estimate \eqref{eq:ua2}. The particular constants are chosen to work
well in our later application of the result. We will argue somewhat
similarly to the proof of Lemma 3.3 in Simon~\cite{Simon85}.
\begin{proof}
  As above, we can write
  \[ u = \sum_{k\geq 1} c_k \rho^{-\mu_k} \Phi_k. \]
  The Weyl law for the eigenvalues $\sigma_k$ of $-L$ on the link of $C\times
  \mathbf{R}$ implies that there is a constant $C_1$ such that for any $B >
  1$ we have
  \[ \#\{ \sigma_k \,|\, \sigma_k \leq B\} \leq C_1B^{\frac{n-1}{2}}, \]
  and so using \eqref{eq:muksigmak}, for a possibly larger $C_1$ we have
  \[ \#\{ -\mu_k \,|\, -\mu_k \leq B\} \leq C_1B^{{n-1}}. \]
  In particular given $\lambda_0\in (0,1/2)$, we can choose
  $\lambda\in [\lambda_0/2, \lambda_0]$ such that
  \[ \label{eq:minmuk}
    \min_k\left\{ \left|\lambda^{-1}\ln 2 + 1 + \mu_k\right| \right\} \geq
    \frac{C_1^{-1}\lambda^{-1}}{(\lambda^{-1})^{n-1}} =
    C_1^{-1}\lambda^{n-2}. \]

  Note that for any $s > 0$ we have
  \[ \label{eq:uL21} \Vert u\Vert^2_{L^2(B_s)} = V \sum_k c_k^2 \int_0^s
    \rho^{-2\mu_k+n-1}\, d\rho = \sum_k V \frac{c_k^2}{-2\mu_k + n}
    s^{-2\mu_k + n}, \]
  where $V$ is the volume of the link of $C\times \mathbf{R}$. We can
  write this as
  \[ \label{eq:uL30} \Vert u\Vert^2_{L^2(B_s)} = \sum_k \tilde{c}_k^2 s^{-2\mu_k
      +n}. \]
  The assumptions \eqref{eq:ua1} can be written as
  \[ \sum_k \tilde{c}_k^2 &\leq 4, \\
    \sum_k \tilde{c}_k^2 e^{-2\lambda(-2\mu_k+n)} &\leq \frac{1}{4}
      e^{-2(n+2)\lambda}. \]
    It follows that
    \[ \sum_k \tilde{c}_k^2 e^{-\ln 4 - (n+2)\lambda} &\leq
      e^{-(n+2)\lambda}, \\
      \sum_k \tilde{c}_k^2 e^{-2\lambda(-2\mu_k+n) + \ln 4 +
        (n+2)\lambda} &\leq e^{-(n+2)\lambda}, \]
    and so, adding the two inequalities,
    \[ \label{eq:sumbound} \sum_k \tilde{c}_k^2 e^{-\lambda(-2\mu_k + n)}
      \cosh\Big(\lambda(-2\mu_k + n)-\ln 4 - (n+2)\lambda\Big) \leq
      e^{-(n+2)\lambda}. \]
    We have
    \[ |\lambda(-2\mu_k + n) - \ln 4 - (n+2)\lambda| =  2\lambda
      \left| \lambda^{-1}\ln 2 + 1 + \mu_k\right| \geq C_1^{-1}
      \lambda^{n-1}, \]
    using \eqref{eq:minmuk}, which implies that
    \[ \cosh\Big(\lambda(-2\mu_k + n)-\ln 4 - (n+2)\lambda\Big) \geq 1
      + C_1^{-1} \lambda^{2n-2}, \]
    increasing $C_1$ if necessary. From \eqref{eq:sumbound} we then have
    \[ \sum_k \tilde{c}_k^2 e^{-\lambda(-2\mu_k + n)} \leq (1 -
      \lambda^A)^2 e^{-(n+2)\lambda}, \]
    if $A$ is chosen sufficiently large, using that $\lambda <
    1/2$. From \eqref{eq:uL30} we get the required result. 
  \end{proof}

  In the rest of this section we will construct hypersurfaces close to
  $C\times\mathbf{R}$ with negative mean curvature, to be used
  as barriers for minimal hypersurfaces later on. The main result is a quantitative version and
  generalization of
  \cite[Proposition 5.8]{Sz20} to general strictly stable and strictly
  minimizing cones $C$ rather than just the Simons cone.
  We first need the following, which is essentially
  identical to Lemma 5.7 in \cite{Sz20}. The basic input in this
  result is that for $a > -\gamma$ we have $L_C(r^a \phi_1) = c_a
  r^{a-2} \phi_1$ for some constant $c_a > 0$. 
\begin{prop}\label{prop:Fadefn}
  For $a > -\gamma$ there are constants $C_a > 0$ and functions
  $F_{\pm,a}$ on $H_{\pm}$ satisfying the following. 
  $L_{H_\pm} F_{\pm,a} > C_a^{-1} r^{a-2}$, $|\nabla^i F_{\pm,a}| \leq C_a
  r^{a-i}$ for $i\leq 3$, and $F_{\pm,a} = r^a\phi_1$ for
  sufficiently large $r$, where we view $H_{\pm}$ as graphs over $C$ to
  define the function $\phi_1$ on $H_{\pm}$. We extend $F_{\pm,a}$ to
  a function $F_a:\mathbf{R}^n\setminus\{0\} \to \mathbf{R}$
  which is homogeneous with degree $a$.  More precisely we set
  \[ F_a(x) = \begin{cases} \lambda^{a} F_{+,a}(\lambda^{-1} x), &\text{ if }x\in
      \lambda H_+,\\
      r^a\phi_1, &\text{ if }x\in C\setminus\{0\},\\
      \lambda^{a}F_{-,a}(\lambda^{-1} x), &\text{ if }x\in \lambda
      H_-, \end{cases} \]
where $\lambda > 0$. 
  Note that $F_a$ constructed in this way is smooth away from the
  origin and its restriction to any leaf of the foliation satisfies
  $L_{H(t)} F_a > C_a^{-1} r^{a-2}$.  
\end{prop}

The following will be our main tool for constructing barrier surfaces
later on. 
\begin{prop}\label{prop:barrier10}
  There is a large odd integer $p$ and a constant $Q>0$ depending on
  the cone $C$, with the following property.
 Let $f:(a,b)\to\mathbf{R}$ be a $C^3$ function, satisfying
  $|f|_{C^3}\leq K$, for some $K > Q$. Then for any $\epsilon < Q^{-1}$ there is an oriented
  hypersurface $X_\epsilon$ defined in the region where $r < K^{-Q^2}$
  and $y\in (a,b)$, satisfying:
  \begin{itemize}
  \item[(i)] $X_\epsilon$ is $C^2$, with negative mean curvature and
    no boundary in the region $0 < r < K^{-Q}$, $y\in (a,b)$. 
  \item[(ii)] At points of $X_\epsilon$ where $r=0$, the tangent cone
    of $X_\epsilon$ is the graph of $-\epsilon r$ over
    $C\times\mathbf{R}$.
  \item[(iii)] The $X_\epsilon$ vary continuously, and
    in each $y$-slice for $y\in (a,b)$, $X_\epsilon$ lies
    between the hypersurfaces
    \[ H( \epsilon f(y)^p -\epsilon) \text{ and } H(\epsilon f(y)^p +
      \epsilon). \]
  \end{itemize}
  In particular if $V$ is a stationary varifold in the region $r <
  K^{-Q^2}$ and the support of $V$ intersects $X_\epsilon$, then near
  the intersection point $V$ cannot lie on the negative side of
  $X_\epsilon$.
\end{prop}
\begin{proof}
  In each $y$-slice for $y\in (a,b)$
  we define $X_\epsilon\subset\mathbf{R}^{n+1}$ to be the
  graph of the function
  \[ -K^Q\epsilon |f(y)|^{\frac{\gamma p}{\gamma+1}} F_{-\gamma+2}
    -\epsilon F_1 \,\, \text{ over }\,\, H(\epsilon f(y)^p),\]
  in terms of the
  $F_a$ function in Proposition~\ref{prop:Fadefn}. We will show that
  for sufficiently large $Q$, this hypersurface satisfies the required
  properties on the region where $r < K^{-Q^2}$ and $y\in (a,b)$. 
  The calculation is similar to the proof of \cite[Proposition
  5.8]{Sz20}, however we need to keep track of the constants. In the
  proof we will write $a \lesssim b$ if $|a| \leq C_1 b$ for a constant
  $C_1$ depending only on the cone $C$.

  By translating it is enough to study the surface $X_\epsilon$ at a point
  where $y=0$. First we focus on points where $r > 0$,
  and consider the region $|y| < R$, $r\in (R, 2R)$ for some $R >
  0$. We consider the rescaled surface $\tilde{X} = R^{-1}X_\epsilon$, which
  is given by the graph of the function
  \[ B(x,y) = -R^{1-\gamma} K^Q \epsilon |f(Ry)|^{\frac{\gamma p}{\gamma+1}} F_{-\gamma+2} -
    \epsilon F_1, \]
  over the surface with slices $H(\pm E(y)^{\gamma+1})$, for
  \[ E(y) = R^{-1}\epsilon^{\frac{1}{\gamma+1}}
    |f(Ry)|^{\frac{p}{\gamma+1}}, \]
  and the $\pm$ sign agrees with the sign of $f(Ry)$. 
  Note that we are studying this rescaled
  surface in the region $|y| < 1$, $r\in (1,2)$.

  \bigskip
  \noindent{\bf Case I.} Suppose that $|f(0)| \geq K^2R$, and to
  simplify notation, suppose that $f(0) > 0$. The case $f(0) < 0$ is
  completely analogous. Since $|\partial_y
    f(Ry)| \leq RK$, it follows that $|f(Ry) - f(0)| \leq KR \leq 1/2
    |f(0)|$ for $|y|\leq 1$, if $K>2$. Therefore on our region we can
    assume that
    $0 <  f(0) /2 \leq f(Ry) \leq 2f(0)$. Let us define
    \[ G(y) = \frac{E(y)}{E(0)} -1 =
      \frac{f(Ry)^{\frac{p}{\gamma+1}}}{f(0)^{\frac{p}{\gamma+1}}} -
        1. \] 
    We have $G(0) = 0$, and on our region
    \[ G' &\lesssim RK |f(0)|^{-1}, \\
      G'' &\lesssim R^2K^2 |f(0)|^{-2} \leq RK|f(0)|^{-1},\\
      G''' &\lesssim R^3K^3 |f(0)|^{-3} \leq R^2K^2 |f(0)|^{-2}.\]
    It follows that $|G|_{C^3} \lesssim RK|f(0)|^{-1}$, while
    $|G''|_{C^1}\lesssim R^2K^2|f(0)|^{-2}$. Note that
    $RK|f(0)|^{-1} \leq K^{-1}$, so $|G|_{C^3}$ is small if $K$
    is sufficiently large.

    Using Proposition~\ref{prop:Hestimates} we can view the
    surface $H(E(y)^{\gamma+1}) = E(y) H = E(0) (1+G(y)) H$
    as the graph of the function $A$ over $E(0) H$, where
    \[ A(x,y) &= E(0) \Phi_{+,G(y)}(E(0)^{-1}x) \\
      &= E(0) G(y) \Phi_+(E(0)^{-1}x) + E(0) G(y)^2
      V_{+,G(y)}(E(0)^{-1}x). \]
    This function $A(x,y)$ is defined for $x\in E(0)H$. 
    From Proposition~\ref{prop:Hestimates} and the bounds for $G$
    above, we know that the first three $x,y$ derivatives of both
    $\Phi(E(0)^{-1}x)$ and $V_{G(y)}(E(0)^{-1}x)$ are of order
    $E(0)^\gamma$ on our region. It follows from this that
    \[ |A|_{C^3} &\lesssim E(0)^{\gamma+1} RK|f(0)|^{-1} +
      E(0)^{\gamma+1} R^2K^2 |f(0)|^{-2} \\
      &\lesssim R^{-\gamma} \epsilon K |f(0)|^{p-1}, \]
    where we measure derivatives on the surface
    $E(0)H\times\mathbf{R}$. 
    At the same time, working at $y=0$ and using $G(0)=0$, we
    have
    \[ L_{E(0)H\times\mathbf{R}} A &\lesssim E(0)^{\gamma+1} G''(0) +
      E(0)^{\gamma+1} R^2K^2 |f(0)|^{-2} \\
      &\lesssim R^{1-\gamma} \epsilon K^2 |f(0)|^{p-2}. \]
      Note that this estimate is $RK|f(0)|^{-1}$ times the one we had
      for $A$ itself, and by our assumption $RK|f(0)|^{-1} \leq
      K^{-1}$.

      We next estimate $B$. For this we have
      \[ |B|_{C^3} \lesssim R^{1-\gamma} K^Q \epsilon
        |f(0)|^{\frac{\gamma p}{\gamma+1}} + \epsilon. \]
      Recall that $B$ is defined on the whole region that we are
      considering, not just for $x\in E(0)H$, and we are estimating
      the derivatives in the entire region. 
      Note that when we differentiate $B$ and derivatives land on the
      $f(Ry)$ term, then we have an additional factor of $RK |f(0)|^{-1} \leq K^{-1}$, which is
      small. Note also that if the surface $E(0) H$ contains any 
      points with $r\in (1,2)$, then we must have $E(0) \lesssim 1$,
      i.e.
      \[ \label{eq:E0est1}
        R^{-1} \epsilon^{\frac{1}{\gamma+1}} |f(0)|^{\frac{p}{\gamma+1}} \lesssim 1. \]
      We can use this to show that $B$ is small: 
      \eqref{eq:E0est1} implies
      \[\label{eq:Best2} R^{1-\gamma} K^Q \epsilon |f(0)|^{\frac{\gamma p}{\gamma+1}} \lesssim
        RK^Q\epsilon^{\frac{1}{\gamma+1}}\leq
        K^{Q-Q^2},  \]
      assuming $\epsilon \leq 1$. 
      If $Q, K$ are sufficiently
      large, then we find that $|B|_{C^3}$ is small.
      At the same, using Proposition~\ref{prop:Fadefn} we have
      \[ L_{E(0)H\times\mathbf{R}} B \leq -c_1 R^{1-\gamma} K^Q
        \epsilon |f(0)|^{\frac{\gamma p}{\gamma+1}} - c_1\epsilon + C_1R^{3-\gamma}
        K^Q\epsilon K^2 |f(0)|^{\frac{\gamma p}{\gamma+1}-2}, \]
      for some $c_1, C_1 > 0$ depending on the cone $C$.

      By
      Lemma~\ref{lem:graphfg} below
      the surface
      $\tilde{X}$ can be viewed as the graph of a function $u$ over
      $E(0) H\times\mathbf{R}$, where $|u - (A + B)|_{C^2} \leq
      C|A|_{C^3}|B|_{C^3}$. It follows
      that the mean curvature $m_{\tilde{X}}$ satisfies
      \[\label{eq:mtX1} m_{\tilde{X}} &= L_{E(0)H\times\mathbf{R}} (A+B) +
        O(|A|_{C^3}^2 + |B|_{C^3}^2) \\
        &\leq -c_1 R^{1-\gamma} K^Q
        \epsilon |f(0)|^{\frac{\gamma p}{\gamma+1}} - c_1\epsilon + \mathcal{Q}, \]
      where
      \[ \mathcal{Q} &\lesssim R^{3-\gamma}
        K^Q\epsilon K^2 |f(0)|^{\frac{\gamma p}{\gamma+1}-2} +  R^{1-\gamma} \epsilon K^2
        |f(0)|^{p-2} \\ 
        &\qquad + (R^{-\gamma} \epsilon K |f(0)|^{p-1})^2 +
         (R^{1-\gamma} K^Q \epsilon
         |f(0)|^{\frac{\gamma p}{\gamma+1}})^2 + \epsilon^2. \]
       The $\epsilon^2$ term is dominated by the $-c_1\epsilon$ term
       in the expression for $m_{\tilde{X}}$ once $\epsilon$ is small,
       and we will see that the remaining terms in $\mathcal{Q}$ can
       be canceled by the $-c_1 R^{1-\gamma}K^Q\epsilon |f(0)|^{\frac{\gamma
         p}{\gamma+1}}$ term. To see this, we have
       \[ \mathcal{Q} \lesssim R^{1-\gamma}K^Q\epsilon |f(0)|^{\frac{\gamma
           p}{\gamma+1}} &\Big[ R^2K^2 |f(0)|^{-2} + K^{-Q} K^2
       |f(0)|^{\frac{p}{\gamma+1}-2}  \\ 
         &+
       K^{-Q} R^{-1-\gamma} \epsilon K^2 |f(0)|^{p\frac{\gamma+2}{\gamma+1}-2} +
       R^{1-\gamma}K^Q\epsilon |f(0)|^{\frac{\gamma p}{\gamma+1}}\Big] + \epsilon^2\]
     Using that $|f(0)|\geq K^2R$ and \eqref{eq:E0est1},
     \eqref{eq:Best2}, and assuming $p/(\gamma+1) > 2$, we have
     \[  \mathcal{Q} \lesssim R^{1-\gamma}K^Q\epsilon |f(0)|^{\gamma
         p}\Big[ K^{-2} + K^{-Q} K^{\frac{p}{\gamma+1}} +
       K^{Q-Q^2}\Big] + \epsilon^2.\]
     Since $\epsilon < Q^{-1}$,
     if $K$ and $Q$ are sufficiently
     large (depending on the cone $C$), then the resulting bound for
     $\mathcal{Q}$ together with the estimate \eqref{eq:mtX1} for
     $m_{\tilde{X}}$ implies that $m_{\tilde{X}} < 0$.

     \bigskip
     \noindent{\bf Case II.} Suppose that $|f(0)| \leq K^2R$. On our
     region $|y|\leq 1$ this implies that $|f(Ry)| \leq K^2R + KR \leq
     2|f(0)|$. We have
     \[ E(0) &\lesssim R^{-1} \epsilon^{\frac{1}{\gamma+1}} K^{\frac{2p}{\gamma+1}} R^{\frac{p}{\gamma+1}}
       \leq \epsilon^{\frac{1}{\gamma+1}} K^{\frac{2p}{\gamma+1} - (\frac{p}{\gamma+1}-1)Q^2} \\
       &\lesssim \epsilon^{\frac{1}{\gamma+1}} K^{-Q^2}, \]
     using $R < K^{-Q^2}$, if $p > 3(\gamma+1)$ and $Q$ is large.
     Differentiating $E(y)$ we find that
     \[ |E(y)|_{C^3} \lesssim \epsilon^{\frac{1}{\gamma+1}}
       K^{-Q^2} \]
     as well. If $K$ is chosen large (depending on the cone $C$),
     we can view $H(\pm E(y)^{\gamma+1})$ as the graph of the function $A$
     over $C\times \mathbf{R}$, where
     \[ A(x,y) = E(y)\Psi_\pm( E(y)^{-1}x), \]
     the $\pm$ sign depending on the sign of $f(Ry)$. 
     Since the $x$-derivatives of $\Psi_\pm(E(y)^{-1}x)$ are of order
     $E(y)^{\gamma}$, it follows that
     \[ |A|_{C^3} \lesssim \epsilon K^{-(\gamma+1)Q^{2}} \leq \epsilon
       K^{-Q^2}. \]
     To estimate the derivatives of $B$, note that since
     \[ B = -R^{1-\gamma} K^Q \epsilon |f(Ry)|^{\frac{\gamma p}{\gamma+1}}
       F_{-\gamma+2} - \epsilon F_0, \]
     we have
     \[  \nabla B \lesssim R^{1-\gamma} K^Q \epsilon |f(0)|^{\frac{\gamma p}{\gamma+1}}
       + R^{2-\gamma} K^Q\epsilon K |f(0)|^{\frac{\gamma p}{\gamma+1}-1} +
       \epsilon. \]
     Using  $|f(0)| \leq K^2 R, R < K^{-Q^2}$
     this implies
     \[ \nabla B &\lesssim R^{2-\gamma} K^Q\epsilon K^2 |f(0)|^{\frac{\gamma p}{\gamma+1}-1} +
       \epsilon \\
     &\leq R^{1 + \gamma(\frac{p}{\gamma+1}-1)} K^Q \epsilon
     K^{\frac{2\gamma p}{\gamma+1}} + \epsilon 
     \\
     &\leq K^{-Q^2/2}\epsilon + \epsilon\]
     as long as $K$ and $Q$ are large, as well as $\frac{\gamma
       p}{\gamma+1} >1$
     and $\frac{p}{\gamma+1} > 1$. 
     Similarly we obtain $|B|_{C^3} \lesssim \epsilon$ as long as
     $\gamma p > 3$, while the first term in the expression for $B$
     satisfies the better bound
     \[ \Big|R^{1-\gamma} K^Q \epsilon |f(Ry)|^{\frac{\gamma p}{\gamma+1}}
       F_{-\gamma+2}\Big|_{C^3} \lesssim K^{-Q^2/2}\epsilon. \]
     Recall that $L_{C\times\mathbf{R}} r < -c_1$ on the
     region $r\in (1,2)$ for some $c_1 > 0$. We then have
     \[ L_{C\times \mathbf{R}} B \leq C_1K^{-Q^2/2}\epsilon -
       c_1\epsilon \leq -\frac{c_1}{2}\epsilon\]
     if $K, Q$ are large. Arguing as in Case I above, it follows that
     \[ m_{\tilde{X}} \leq -\frac{c_1}{2}\epsilon + \mathcal{R}, \]
     where
     \[  \mathcal{R} \lesssim K^{-Q^2}\epsilon + \epsilon^2. \]
     if $\epsilon < Q^{-1}$ for sufficiently large $Q$, then we get $m_{\tilde{X}} <
     0$.

    \smallskip
      
     Finally let us suppose that $f(0)=0$. For any (small) $R > 0$ the rescaled surface
     $\tilde{X} = R^{-1}X$ is still given by the graph of $B$ over the
     surface with slices $H(\pm E(y)^{\gamma+1})$ in $\mathbf{R}^n\times \{y\}$ as
     above. Since $f(0)=0$ we have $|f(Ry)|\leq RK$ for $|y|\leq 1$, and so
     \[ E(y) \lesssim R^{\frac{p}{\gamma+1}-1}\epsilon^{\frac{1}{\gamma+1}} K^{\frac{p}{\gamma+1}}, \]
     which converges to 0 as $R\to 0$. Similarly the first term in $B$ satisfies
     \[ -R^{1-\gamma} K^Q \epsilon |f(Ry)|^{\frac{\gamma p}{\gamma+1}} F_{-\gamma+2}
       &\lesssim R^{1+ \gamma(\frac{p}{\gamma+1}-1)}K^Q\epsilon
       K^{\frac{\gamma p}{\gamma+1}}, \] 
     which also converges to zero as $R\to 0$. It follows that the
     tangent cone of $X$ at the point $(0,0)\in
     \mathbf{R}^n\times\mathbf{R}$ is the graph of $-\epsilon r$ over
     $C\times\mathbf{R}$.

     It follows that our surface $X_\epsilon$ satisfies the conditions
     (i) and (ii). It is also clear from the construction that
     $X_\epsilon$ varies continuously with $\epsilon$.
     We now show property (iii), increasing the value of
     $Q$ if necessary. To see that $X_\epsilon$ is on the negative
     side of the surface with slices $H(\epsilon f(y)^p + \epsilon)$,
     note that by Lemma~\ref{lem:L51} $H(\epsilon f(y)^p + \epsilon)$ is on the positive side
     of the graph of the function
     \[ c_0 \min\{ \epsilon r^{-\gamma}, r\} \]
     over $H(\epsilon f(y)^p)$ for a small $c_0 > 0$ depending on the
     cone $C$. Therefore it is enough
     to check that on the region $r < K^{-Q^2}$ we have
     \[ - K^Q \epsilon |f(y)|^{\frac{\gamma p}{\gamma+1}}F_{-\gamma+2}
       - \epsilon F_1 < c_0 \min\{ \epsilon r^{-\gamma}, r\}. \]
     Since $F_1 \lesssim r$, it is clear that $|\epsilon F_1| \leq
     \epsilon c_0 r^{-\gamma}$ is $r$ is sufficiently small, and also
     $|\epsilon F_1| \leq c_0 r$ if $\epsilon$ is sufficiently
     small. For the other term we have
     \[ \Big| K^Q \epsilon |f(y)|^{\frac{\gamma p}{\gamma+1}}
       F_{-\gamma+2}\Big| \leq  K^{Q+\frac{\gamma
           p}{\gamma+1}}r^2  \epsilon r^{-\gamma} \leq K^{Q+\frac{\gamma
           p}{\gamma+1} - 2Q^2} \epsilon r^{-\gamma}, \]
     which is smaller than $c_0 \epsilon r^{-\gamma}$ if $Q, K$ are
     large. Finally, to bound the same term by $c_0 r$ note that,
     just as in \eqref{eq:E0est1}, on $H(\epsilon f(y)^p)$
     we have
     \[ \epsilon^{\frac{1}{\gamma+1}} |f(y)|^{\frac{p}{\gamma+1}}
       \lesssim r. \]
     This implies
     \[ \Big| K^Q \epsilon |f(y)|^{\frac{\gamma p}{\gamma+1}}
       F_{-\gamma+2}\Big| \leq C_1  K^Q \epsilon^{\frac{1}{\gamma+1}}
       r^2 \leq C_1 K^{Q-Q^2} \epsilon^{\frac{1}{\gamma+1}} r, \]
     which is smaller than $c_0r$ if $K,Q$ are large. The argument to
     see that $X_\epsilon$ lies on the positive side of $H(\epsilon
     f(y)^p -\epsilon)$ is entirely analogous, using that $H(\epsilon
     f(y)^p - \epsilon)$ lies on the negative side of the graph of
     $-c_0\min\{\epsilon r^{-\gamma}, r\}$ over $H(\epsilon f(y)^p)$. 

     The final claim follows from the maximum principle due to
     White~\cite[Theorem 4]{White10}. Indeed, suppose that the support
     of a stationary varifold $V$ intersects $X_\epsilon$. If at the
     intersection point $X_\epsilon$ is smooth (i.e. $r > 0$), then by
     White's maximum principle $V$ cannot lie on the negative side of
     $X_\epsilon$ near the intersection point, since $X_\epsilon$ has
     strictly negative mean curvature. At the same time if at the
     intersection point $z$ we have $r=0$, then any tangent cone $V_z$
     of  $V$ at $z$ must lie on the negative side of the graph of
     $-\epsilon r$ over $C\times \mathbf{R}$. In particular the link
     of $V_z$ lies on one side of the link of
     $C\times\mathbf{R}$. This again can be seen to contradict the
     maximum principle. 
\end{proof}

We used the following lemma, which is essentially the same as \cite[Lemma
5.9]{Sz20}.
\begin{lemma}\label{lem:graphfg}
  Let $S$ be a hypersurface in $\mathbf{R}^n$ with second fundamental
  form $A_S$ satisfying the bounds $|\nabla^i A_S|\leq 1$,
  for $i\leq k$. There exists a constant $C$ depending on the
  dimension and on $k$, with the following property. Suppose that $f$
  is a function on $S$ with $|f|_{C^k} < C^{-1}$ and $g$ is a function
  on a neighborhood of $S$ with $|g|_{C^k} < C^{-1}$. Let $S'$
  denote the graph of $g$ over the graph of $f$ on $S$. Then $S'$ can
  be written as the graph of a function $u$ over $S$, such that
  \[ |u - (f+g)|_{C^{k-1}} \leq C |f|_{C^k} |g|_{C^k}. \]
  Here the derivatives of $u-(f+g)$ and $f$ are measured along $S$,
  while the derivatives of $g$ are measured in $\mathbf{R}^n$.  
\end{lemma}

\section{Minimal hypersurfaces with cylindrical tangent cones}\label{sec:gluing}
In this section we construct minimal hypersurfaces in a neighborhood of
$0\in \mathbf{R}^n \times \mathbf{R}$ with an isolated singularity at
the origin, and tangent cone $C\times \mathbf{R}$. The construction is
analogous to the construction of singular Calabi-Yau metrics with
isolated singularities in \cite{Sz17}
(see also Hein-Naber~\cite{HN08}).

Suppose that $\l$ is an integer such that $\l-\gamma > 1$. Then
$C\times\mathbf{R}$ admits a homogeneous Jacobi field of degree
$\l-\gamma$ of the form
\[ \label{eq:ul} u_\l = (y^\l r^{-\gamma} + a_1 y^{\l-2} r^{2-\gamma}  +
  \ldots a_{\lfloor \l/2\rfloor} y^{\l - 2\lfloor \l/2\rfloor}
  r^{2\lfloor \l/2\rfloor - \gamma} )\phi_1, \]
where, as above, $\phi_1$ is the first eigenfunction of $-L_\Sigma$ on
the link $\Sigma$ of $C$, and $r^{-\gamma}\phi_1$ is the corresponding
Jacobi field on $C$. The $a_i$ are suitable constants uniquely determined by
the condition that $L_{C\times\mathbf{R}} u_\l = 0$.

We can consider the graph of $u_\l$ over $C\times \mathbf{R}$ on a region where $|u_\l| \ll
r$, i.e. where $|y|^\l \ll r^{\gamma+1}$. At the same time, on the region
where $r$ is much smaller, we can glue in suitable scaled
copies of the Hardt-Simon smoothings $H_{\pm}$ in the slices of the
form $\mathbf{R}^n \times \{y\}$. Interpolating between the two
regions using cutoff functions we will obtain a hypersurface $X$ whose
mean curvature will be almost zero in a suitable weighted space. We
will then show that in a possibly smaller neighborhood of $0$ we can
find a minimal graph over $X$. Using barrier arguments based on
Proposition~\ref{prop:barrier10} we will show that many of these
minimal hypersurfaces are area minimizing in a neighborhood of the
origin. In particular we will show the
following.
\begin{thm}\label{thm:Texist}
  There exist minimal hypersurfaces in a neighborhood of $0\in
  \mathbf{R}^n\times \mathbf{R}$ that are smooth away from the origin,
  and have tangent cone $C\times \mathbf{R}$ as their unique
  (multiplicity one) tangent cone at the origin. If the integer $\l$ in
  the construction is sufficiently large, then the minimal
  hypersurface that we construct is area minimizing in a neighborhood
  of the origin. 
\end{thm}

The proof of this result will take up the rest of this section. The
existence of the required minimal hypersurface $T$ is proven in
Proposition~\ref{prop:Texist}, while the fact that $T$ is area
minimizing for large $\l$ is shown in Proposition~\ref{prop:Tminimizing}. In the
process we will also give a more detailed picture of what these
surfaces look like, see Proposition~\ref{prop:Tgraphbound} below. It
is reasonable to expect all of the hypersurfaces that we construct to be area
minimizing near the origin, but showing this may need a more careful
construction of suitable barrier surfaces.

\subsection{The approximate solutions}\label{sec:approxsoln1}
In this subsection we will construct approximate solutions to our
problem. Recall that $\l$ is an integer such that $\l-\gamma > 1$. Let
us define the number $a = \frac{\l}{1+\gamma}$, and let $\beta\in
(1,a)$. We define $X$ in the ball $\{\rho \leq A^{-1}\}$ for a
sufficiently large $A$, in the following three pieces:
\begin{itemize}
  \setlength\itemsep{1em}
  \item On the region where $r \geq 2|y|^\beta$ we let $X$ be the
    graph of $u_\l$ over $C\times \mathbf{R}$. It is convenient to
    deal separately with the region where $r \geq |y|$, where we have
    $|r^{-1} u_\l| = O(r^{\l-\gamma-1})$ as $r\to 0$. Since $\l-\gamma > 1$, it
    makes sense to consider the graph of $u_\l$ on a sufficiently small
    neighborhood of 0. At the same time, on the region $2|y|^\beta
    \leq r \leq |y|$ we have $|r^{-1}u_\l| = O(|y|^\l r^{-\gamma-1})$ as $r\to
    0$. Since $\l > \beta(\gamma+1)$, it makes sense to consider the
    graph of $u_\l$ on this region too, once $r$ is sufficiently
    small. 
 \item On the region where $r \leq |y|^\beta$, we define $X$ to be
   the surface $H(y^\l)$ in the slice $\mathbf{R}^n \times
   \{y\}$. Note that by Definition~\ref{defn:Ht} we have $H(y^\l) =
   |y|^a H_\pm$. 
\item On the intermediate region $|y|^\beta \leq r \leq 2|y|^\beta$ we
  interpolate between the two definitions above, using a cutoff
  function. More precisely, let $\chi:\mathbf{R}\to [0,1]$ be a
  standard cutoff function, with $\chi(t) = 1$ for $t < 1$ and
  $\chi(t) = 0$ for $t > 2$. In addition recall that $H_{\pm}$ is the
  graph of $\pm r^{-\gamma}\phi_1 + v_{\pm}$ over $C$, outside of a
  large ball. On the region $|y|^\beta \leq r \leq 2|y|^\beta$ we let
  $X$ be the graph of
  \[ \label{eq:gluing1} u_\l + \chi\left(\frac{|x|}{|y|^\beta}\right) \left[ y^\l
      r^{-\gamma}\phi_1 - u_\l + |y|^a v_{\pm}(|y|^{-a} x)\right] \]
  over $C\times \mathbf{R}$,
  where the choice of $v_{\pm}$ depends on the sign of $y^\l$. Note
  that this definition matches up with the definitions of $X$ in the
  two regions above. 
\end{itemize}

We need to estimate the mean curvature of $X$. 
\begin{prop}\label{prop:mXest}
  Let $\beta\in (1,a)$. Suppose that $\delta > \l-\gamma$ is sufficiently close to
  $\l-\gamma$, and $\tau \leq -\gamma$.
  Then there exists a $\kappa > 0$ such that on the
  punctured ball
  $\{0 < \rho < A^{-1}\}$ for sufficiently large $A$, we have the
  estimates
  \[ \label{eq:mXest1} |m_X| + r |\nabla m_X| < A^{-\kappa} \rho^{\delta - \tau}
    r^{\tau-2}. \]
  Here $m_X$ denotes the mean curvature of $X$, and $\nabla m_X$ is
  the derivative of $m_X$ on $X$.
\end{prop}
\begin{proof}
 To simplify notation
we will restrict ourselves to the region where $y \geq 0$, and we will
simply write $H, v$ for $H_+, v_+$. We will analyze the surface $X$ in
different regions, by rescaling it and comparing it to various
scalings of the surfaces $H\times\mathbf{R}$ and $C\times\mathbf{R}$.
In addition we can assume without loss of generality that $\tau =
-\gamma$, since if \eqref{eq:mXest1} holds for $\tau=-\gamma$ then it
also holds for all $\tau < -\gamma$. This is because $\rho\geq r$, and
so $(r / \rho)^{\tau}$ increases as $\tau$ becomes more negative.  

\begin{itemize}
  \setlength\itemsep{1em}
\item {\bf Region I}, where $r > |y|$ and $r\in (R, 2R)$. Here
  $r, \rho\sim R$. Consider the rescaled surface $\tilde{X} = R^{-1}X$,
  given by the graph of $\tilde{u}_\l(x,y) = R^{-1}u_\l(Rx, Ry)$ over $C\times
  \mathbf{R}$. After scaling our region corresponds to $r > |y|$ with
  $r\in (1,2)$. From the definition of $u_l$ we see that in this region
  \[ \label{eq:tildeuest1}|\nabla^i \tilde{u}_\l| \leq C_i R^{\l-\gamma-1}. \]
  Since $u_l$ is a Jacobi field, once $R$ is sufficiently small
  (i.e. $R < A^{-1}$ for large $A$), the mean curvature of $\tilde{X}$
  satisfies
  \[ |m_{\tilde{X}}| + |\nabla m_{\tilde{X}}| < C R^{2\l-2\gamma-2}. \]
  Scaling back, this implies
  \[ |m_X| + R|\nabla m_X| < C R^{2\l-2\gamma-3}. \]
  Since in our region $r, \rho \sim R < A^{-1}$, in order for \eqref{eq:mXest1}
  to hold for sufficiently large $A$ and small $\kappa > 0$, it is enough to ensure that
  \[ 2\l-2\gamma -3 > (\delta-\tau) + (\tau-2) = \delta-2, \]
  i.e. $\delta < 2\l-2\gamma -1$. This is satisfied if $\delta$ is
  sufficiently close to $\l-\gamma$ since by assumption $\l-\gamma >
  1$.
\item {\bf Region II}, where $2|y|^\beta \leq |x| \leq |y|$. Here the
  surface $X$ is still the graph of $u_l$ over $C\times
  \mathbf{R}$. Again assume that $r\in (R, 2R)$, so we have $r\sim R$
  and $\rho\sim |y|$, i.e. $\rho^\beta\lesssim R \lesssim \rho$. As
  above consider
  $\tilde{X} = R^{-1}X$, viewed as the graph of $\tilde{u}_\l =
  R^{-1}u_\l(R\cdot)$ over $C\times \mathbf{R}$. We will use rescaled
  coordinates $\tilde{x} = R^{-1}x, \tilde{y} = R^{-1}y$. From the definition of
  $u_\l$ we now have the estimate
  \[ \label{eq:tildeuest2}|\nabla^i \tilde{u}_\l| \leq C_i
    |\tilde y|^\l R^{\l-\gamma-1}, \]
  and in the rescaled region $2|R\tilde y|^\beta \leq |R\tilde x|$,
  i.e. $|\tilde y|\lesssim R^{\beta^{-1} -1}$. Note that then
  \[ \label{eq:i2} |\tilde y|^\l R^{\l-\gamma-1} \lesssim R^{\l\beta^{-1} -\gamma-1}, \]
  and by our assumption on $\beta$ we have $l\beta^{-1} -\gamma-1 >
  0$. If $R < A^{-1}$ for sufficiently large $A$, and since $u_\l$ is a
  Jacobi field, we then have
  \[  |m_{\tilde{X}}| + |\nabla m_{\tilde{X}}| < C |\tilde y|^{2\l}
    R^{2\l-2\gamma-2}, \]
  and after scaling
  \[ |m_X| + R|\nabla m_X| < C |y|^{2\l} R^{-2\gamma-3} \leq C\rho^{2\l}
    R^{-2\gamma-3}. \]
  In order for \eqref{eq:mXest1} to hold for sufficiently small
  $\kappa > 0$ (with $\tau=-\gamma$), we need
  \[\label{eq:i1} r^{-3-2\gamma} \rho^{2\l} \leq \rho^{\kappa+\delta +\gamma}
    r^{-\gamma-2}. \]
  Using that in our region $\rho^{\beta}
  \lesssim r$, it is enough to ensure that
  $2\l-\delta-\gamma-(1+\gamma)\beta > 0$, or equivalently
  \[ \beta < \frac{\l}{1+\gamma} - \frac{\delta -
      (\l-\gamma)}{1+\gamma}. \]
  Since by assumption $\beta < a$, this inequality can be satisfied by
  choosing $\delta>\l-\gamma$ sufficiently close to $\l-\gamma$. Then in
  turn \eqref{eq:i1} will be satisfied for sufficiently small $\kappa
  > 0$.
\item {\bf Region III}, where $|y|^\beta \leq |x| \leq
  2|y|^\beta$. This is the gluing region, where $X$ is the graph of
  the function in \eqref{eq:gluing1} over $C\times \mathbf{R}$. We
  write this as the graph of $u_\l + E$, where
  \[ E = \chi\left(\frac{|x|}{|y|^\beta}\right) \left[ y^\l
      r^{-\gamma}\phi_1 - u_\l + |y|^a v(|y|^{-a} x)\right]. \]
  As before, we assume that $r\in (R, 2R)$ and we consider the
  rescaled surface $\tilde{X} = R^{-1}X$, which is the graph of
  $R^{-1}u_\l(R\cdot) + R^{-1}E(R\cdot)$. We let $\tilde{x} = R^{-1}x,
  \tilde{y} = R^{-1}y$ as above. The leading contributions to
  the mean curvature $m_{\tilde{X}}$ come from the Jacobi operator
  $L_{C\times\mathbf{R}}$ applied to $E$, and the non-linear part of
  the mean curvature operator applied to $u_\l$. The latter
  contribution is bounded in exactly the same way as in Region II, so
  we focus on the former. For this we need to estimate the derivatives
  of the function $\tilde{E}(\tilde x,\tilde y) = R^{-1}E(R\tilde x,
  R\tilde y)$. The contribution
  of the cutoff function is uniformly bounded so we will ignore
  it.

  The leading term in $y^\l r^{-\gamma}\phi_1 - u_\l$ is $a_1
  y^{\l-2}r^{2-\gamma}\phi_1$, so after rescaling the corresponding
  term and its derivatives are $O( |\tilde y|^{\l-2} R^{\l-\gamma-1})$ (compare
  to \eqref{eq:i2}). Scaling back down, we find that 
the contribution of this term to $|m_X| +
  R|\nabla m_X|$ is $O( \rho^{\l-2} R^{-\gamma})$. For the estimate
  \eqref{eq:mXest1} to hold we then need
\[ \rho^{\l-2} R^{-\gamma} \leq \rho^{\kappa+\delta+\gamma}
  R^{-\gamma-2}. \]
Using that $R \sim \rho^\beta$, it is enough to ensure that
\[ \kappa+\delta+\gamma-\l+2 \leq 2\beta. \] 
This holds for sufficiently small $\kappa$ since $\beta > 1$, as long
as $\delta - (\l-\gamma)$ is chosen very small. 

  It remains to estimate the derivatives of the
  function
  \[ \tilde{v}(\tilde x,\tilde y) = R^{a-1}|\tilde y|^a v( R^{1-a}
    |\tilde y|^{-a}\tilde x). \]
  By Proposition~\ref{prop:Hestimates} we have
  $\nabla^i v = O(|x|^{-\gamma-c-i})$ for some $c > 0$. 
Using this we can estimate the derivatives of
  $\tilde{v}$:
  \[ |\nabla^i \tilde{v}| \leq C_i (R^{1-a}
    |\tilde y|^{-a})^{-\gamma-c-1}. \]
  The contribution of this to $|m_X| + R|\nabla m_X|$ is of order $(R
  \rho^{-a})^{-\gamma-c-1} R^{-1}$, and so to obtain \eqref{eq:mXest1}
  we need to ensure that
\[ \rho^{a(\gamma+1+c)} R^{-\gamma-2-c} \leq
  \rho^{\kappa+\delta+\gamma} R^{-\gamma-2}. \]
Using again that $R \sim \rho^{\beta}$, it is enough to show that
\[  \kappa + \delta - (\l-\gamma) - ac < -\beta c.\]
This holds for small $\kappa$ as long as $\delta - (\l-\gamma)$ is
sufficiently small, since $\beta < a$.

  \item {\bf Region IV}, where $|x| \leq |y|^\beta$. In this region
    the surface $X$ is defined to be $H(y^\l) = y^aH$ in the slice
    $\mathbf{R}^n \times \{y\}$ (recall that we assume $y \geq
    0$). Note that as a result we also have $|x| \geq c_0 y^a$ for some
    $c_0 > 0$. Suppose
    that $r\in (R,2R)$, and let $\tilde{X} = R^{-1}X$ as before. We
    suppose in addition that $y\in (y_0-R, y_0+R)$. It is
    convenient to introduce new coordinates
    \[ \tilde{x} &= R^{-1}x, \quad \tilde{y} = R^{-1}(y-y_0), \]
    so $|\tilde{x}| \in (1,2)$ and $|\tilde{y}|\leq 1$. In terms of
    the new coordinates the scaled up surface $\tilde{X}$ is given by
    $E(\tilde{y})\cdot H$ in the slice $\mathbf{R}^n\times\{\tilde{y}\}$, where
    \[ E(\tilde{y}) = R^{-1}(R\tilde{y} + y_0)^a. \]
    To estimate the mean curvature of this, we view it as a graph over
    the surface $E(0)\cdot H\times \mathbf{R}$. This is somewhat
    similar to the calculation in
    the proof of Proposition~\ref{prop:barrier10}

    Using that $|x|\sim R \ll y_0$, we find that $E(\tilde{y})$ and
    all of its derivatives are of order $R^{-1}y_0^a$ (in fact we have
    more precise bounds $|\partial_{\tilde{y}}^i E(\tilde{y})|
    \lesssim (Ry_0)^{-i} R^{-1}y_0^a$). In addition we
    have
    \[ E(0)^{-1} E(\tilde{y}) = y_0^{-a}(R\tilde{y} + y_0)^a = (1 +
      Ry_0^{-1}\tilde{y})^a = 1 + O(Ry_0^{-1}). \]
    We write $E(0)^{-1}E(\tilde{y}) - 1 = G(R y_0^{-1}\tilde{y})$
    where $G$ has uniformly bounded derivatives. 
    Using Proposition~\ref{prop:Hestimates} we view $E(\tilde{y}) H$
    as the graph of the function
    \[ A &= E(0) \Phi_{G(Ry_0^{-1}\tilde{y})} (E(0)^{-1}\cdot) \\
      &=  E(0) G(Ry_0^{-1}\tilde{y}) \Phi( E(0)^{-1}\cdot) +
      E(0) G(Ry_0^{-1}\tilde{y})^2 V_{G(Ry_0^{-1}\tilde{y})} (E(0)^{-1}\cdot) \]
    over $E(0)\cdot H \times \mathbf{R}$. For $A$ and its derivatives
    we have the estimate
    \[ |\nabla^i A| \leq C_i E(0) Ry_0^{-1} E(0)^\gamma = C_i
      R^{-\gamma} y_0^{a + a\gamma - 1}.\]
    At the same time calculating at $\tilde{y}=0$ we
    have
    \[ L_{E(0)H\times\mathbf{R}} A 
      &= E(0) R^2 y_0^{-2} G''(0) \Phi(E(0)^{-1}\cdot)  \\
      &\quad + 
    E(0) L_{E(0)H\times \mathbf{R}}
    (G(Ry_0^{-1}\tilde{y})^2 V_{G(Ry_0^{-1}\tilde{y})}(E(0)^{-1}\cdot)). \]
  To estimate the last term we calculate that
  \[ \Big|\nabla^i E(0)\,G(Ry_0^{-1}\tilde{y})^2 V_{G(Ry_0^{-1}\tilde{y})}
    (E(0)^{-1}\cdot)\Big| &\leq C_i R^2y_0^{-2} E(0)^{1+\gamma} \\
    &= C_i R^{1-\gamma} y_0^{a+a\gamma-2}. \]
  In sum, the mean curvature satisfies
  \[ m_{\tilde{X}} &= L_{E(0)H\times\mathbf{R}} A + O(A^2) \\
    &\lesssim R^{1-\gamma}y_0^{a+a\gamma-2} +
    R^{-2\gamma}y_0^{2a+2a\gamma-2} \lesssim
    R^{1-\gamma}y_0^{a+a\gamma-2}, \]
  where we used that $R\geq c_0 y_0^a$. The derivatives of
  $m_{\tilde{X}}$ satisfy the same estimate. Scaled down we have
  \[ |m_{X}| + R|\nabla m_{X}| \lesssim R^{-\gamma} y_0^{a+a\gamma
      -2}. \]
  To satisfy \eqref{eq:mXest1} (with $\tau = -\gamma$) we need
  \[ R^{-\gamma} y_0^{a+a\gamma-2} < y_0^{\kappa + \delta+\gamma}
    R^{-\gamma-2}, \]
  since in our region $\rho\sim y_0$.
  Using that $R\leq y_0^\beta$, we can find a suitable $\kappa > 0$ as
  long as
  \[ 2\beta > \delta + \gamma + 2 -a-a\gamma = \delta - (\l-\gamma) +
    2. \]
  Since by assumption $\beta > 1$, this is satisfied as long as
  $\delta$ is sufficiently close to $\l-\gamma$. Then \eqref{eq:mXest1}
  holds for some $\kappa > 0$. 
\end{itemize}
\end{proof}

\subsection{Linear analysis}
In this section we study the mapping properties of the linearized operator
$L_X$ on the surface $X$ that we constructed. The analysis here is
very close to the discussion in \cite{Sz17} so we will be somewhat
brief. As in \cite{Sz17} we  invert the Jacobi operator on
different model spaces, and then glue together the local inverses
using cutoff functions in order to construct an appoximate inverse on
a sufficiently small neighborhood of the origin in $X$.

\subsubsection{Weighted spaces on $X$}
We consider
locally $C^{k,\alpha}$-functions on $X\cap \rho^{-1}(0,A_0^{-1})$, and
define their weighted $C^{k,\alpha}_{\delta,\tau}$-norm by
\[ \label{eq:weightedn}\Vert f\Vert_{C^{k,\alpha}_{\delta,\tau}} = \sup_{R,S > 0} R^{-\tau}
    S^{\tau-\delta} \Vert
    f\Vert_{C^{k,\alpha}_{R^{-2}g_X}(\Omega_{R,S})}. \]
  Here $\Omega_{R,S}\subset X\cap \rho^{-1}(0,A_0^{-1})$
  is the region where $\rho\in (S,2S)$ and
$r \in (R,2R)$. The metric $g_X$ denotes the induced metric on $X$,
and the subscript $R^{-2}g_X$ indicates that we measure the usual
H\"older norm using this rescaled metric.

Note that from the analysis of $X$ in different regions in the proof of
Proposition~\ref{prop:mXest} it follows that the metric $R^{-2}g_X$
has bounded geometry on $\Omega_{R,S}$, with bounds independent of $R,
S$. Because of this we have the
following.
\begin{lemma}
  The Jacobi
operator of $X$ defines a bounded linear map
\[ L_X : C^{2,\alpha}_{\delta, \tau} \to
  C^{0,\alpha}_{\delta-2,\tau-2}. \]
\end{lemma}
\begin{proof}
  Suppose that $\Vert u\Vert_{C^{2,\alpha}_{\delta, \tau}} \leq 1$. On
    a given region $\Omega_{R,S}$ this means that
    \[ \Vert u\Vert_{C^{2,\alpha}_{R^{-2}g_X}(\Omega_{R,S})} \leq
      R^\tau S^{\delta-\tau}. \]
    Let us denote by $L_{R^{-1}X}$ the Jacobi operator with respect
    to the scaled up surface. The bounded geometry of $R^{-2}g_X$
    implies that
    \[ \Vert L_{R^{-1}X}
      u\Vert_{C^{0,\alpha}_{R^{-2}g_X}(\Omega_{R,S})} \leq C \Vert
      u\Vert_{C^{2,\alpha}_{R^{-2}g_X}(\Omega_{R,S})} \leq CR^\tau S^{\delta-\tau} \]
    for a constant $C$ (independent of $R,S$). We have $L_{R^{-1}X} =
    R^2 L_X$, and so
    \[ \Vert L_X u\Vert_{C^{0,\alpha}_{R^{-2}g_X}(\Omega_{R,S})} \leq
      C R^{\tau-2} S^{\delta-\tau}.\]
    Taking the supremum over all $R,S$ we obtain
    \[ \Vert L_X u\Vert_{C^{0,\alpha}_{\delta-2,\tau-2}} \leq C \Vert
      u\Vert_{C^{2,\alpha}_{\delta, \tau}}\]
    as required. 
  \end{proof}

  We will also need an estimate for the nonlinear terms in the mean
  curvature operator. For a function $u\in C^{2,\alpha}_{\delta,\tau}$
  let us denote by $m_X(u)$ the mean curvature of the graph of $u$
  over $X$. We define the nonlinear operator $Q_X$ by
  \[ \label{eq:mXQ} m_X(u) = m_X + L_X(u) + Q_X(u). \]
  We then have the following, expressing that $Q_X$ contains only
  quadratic and higher order terms.
  \begin{lemma}\label{lem:Qquadratic}
    There are constants $c_1, C > 0$ with the following
    property. Suppose that $\Vert u_i \Vert_{C^{2,\alpha}_{1,1}} \leq
    c_1$ for $i=1,2$. Then
    \[ \Vert Q_X(u_1) - Q_X(u_2)\Vert_{C^{0,\alpha}_{\delta-2,\tau-2}}
      \leq C (\Vert u_1 \Vert_{C^{2,\alpha}_{1,1}} + \Vert u_1
      \Vert_{C^{2,\alpha}_{1,1}}) \Vert u_1 -
      u_2\Vert_{C^{2,\alpha}_{\delta,\tau}}. \]
  \end{lemma}
  \begin{proof}
    Consider one of the regions $\Omega_{R,S}$. By assumption
    \[ \Vert R^{-1} u_i\Vert_{C^{2,\alpha}_{R^{-2}g_X}(\Omega_{R,S})} \leq
      c_1. \]
    If we denote by $X_u$ the graph of $u$ over $X$, then the rescaled
    surface $R^{-1}X_u$ is the same as the graph of $R^{-1}u$ over
    $R^{-1}X$. Since $R^{-1}X$ has bounded geometry, it follows that
    if $\Vert R^{-1}u\Vert_{C^{2,\alpha}_{R^{-2}g_X}}$ is sufficiently
    small, then 
    \[ m_{R^{-1}X}(R^{-1}u) = m_{R^{-1}X} + L_{R^{-1}X}(R^{-1}u) +
      Q_{R^{-1}X}(R^{-1}u), \]
    where $Q_{R^{-1}X}(R^{-1}u)$ is a convergent power series in
    $R^{-1}u$ and its first two derivatives, with at least degree two
    terms. From this it follows that if $c_1$ is sufficiently small
    above, then 
    \[ Q_{R^{-1}X}(R^{-1}u_1) - Q_{R^{-1}X}(R^{-1}u_2) \]
    is a convergent power series in $R^{-1}u_i$ and its first two
    derivatives, in which every term has a factor of $R^{-1}(u_1-u_2)$
    or one of its derivatives. It follows that
    \[ \Vert Q_{R^{-1}X}(R^{-1}u_1) & - Q_{R^{-1}X}(R^{-1}u_2)
      \Vert_{C^{0,\alpha}_{R^{-2}g_X}(\Omega_{R,S})} \leq \\
      &C \Vert
      R^{-1}(u_1-u_2)\Vert_{C^{2,\alpha}_{R^{-2}g_X}}
      \Big( \Vert
      R^{-1}u_1\Vert_{C^{2,\alpha}_{R^{-2}g_X}} + 
      \Vert
      R^{-1}u_2\Vert_{C^{2,\alpha}_{R^{-2}g_X}} \Big). \]
    The claimed result follows using that $m_{R^{-1}X}(R^{-1}u) = R
    m_X(u)$. 
  \end{proof}

  For technical reasons it will be useful to work with H\"older spaces
  on $X\cap \rho^{-1}(0,A^{-1}]$, with $A > 2A_0$, for some fixed $A_0 >
  0$. We define the
  corresponding space $C^{k,\alpha}_{\delta,\tau}(X\cap
  \rho^{-1}(0,A^{-1}])$ to be the space of functions obtained by
  restricting elements of
  $C^{k,\alpha}_{\delta,\tau}(X\cap\rho^{-1}(0,A_0^{-1}))$. The
  norm of a function $f\in C^{k,\alpha}_{\delta,\tau}(X\cap
  \rho^{-1}(0,A^{-1}])$ is defined to be the infimum of the corresponding
  norms of the extensions of $f$ to
  $\rho^{-1}(0,A_0^{-1})$. 
  
  We will need the following comparison between the weighted
  spaces for different weights.
  \begin{lemma}\label{lem:normcomp1}
    Suppose that $\delta,\tau$ are as in
    Proposition~\ref{prop:mXest}, with $\tau$ sufficiently close to
    $-\gamma$.  Then for constants $\kappa, C > 0$ we
    have
    \[ \Vert u\Vert_{C^{2,\alpha}_{1,1}(\rho^{-1}(0, A^{-1}])} \leq C
      A^{-\kappa} \Vert
      u\Vert_{C^{2,\alpha}_{\delta,\tau}(\rho^{-1}(0,A^{-1}])}. \]
  \end{lemma}
  \begin{proof}
    From the construction of $X$ we know that for a constant $c_0 > 0$
    we have $r \geq c_0\rho^a$ on $X$. This implies that on the region
    $\Omega_{R,S}$, i.e. where $r\in (R,2R)$ and $\rho\in (S,2S)$, we
    have
    \[ R^\tau S^{\delta-\tau} \leq C S^{a(\tau-1)+\delta-\tau} R. \]
    If $\delta > \l-\gamma$ and $\tau$ is sufficiently close to
    $-\gamma$, then $a(\tau-1) + \delta-\tau > 0$. The result follows
    from this. 
  \end{proof}
  
We will also need to study the geometry of $X$ in different regions that are
larger than those considered in Proposition~\ref{prop:mXest}. This is
analogous to Propositions 6 and 7 in \cite{Sz17}. We will use similar
notation.

For large $A, \Lambda > 0$, consider the region
\[ \mathcal{U} = \{0 < \rho < A^{-1}, r > \Lambda\rho^a\}\cap X \]
inside $X$. Let $F : \mathcal{U} \to C\times \mathbf{R}$ be the
nearest point projection. We then have the following.
\begin{prop}\label{prop:projest1}
  Given $k, \epsilon > 0$, if $\Lambda > \Lambda(\epsilon)$ and $A >
  A(\epsilon)$ are sufficiently large, then on the set $\mathcal{U}$
  we have $|\nabla^i(F(x,y) - (x,y))| \leq \epsilon |x|^{1-i}$, and in addition
  \[ |\nabla^i(F^*g_{C\times\mathbf{R}} - g_X)|_{g_X} < \epsilon
    r^{-i}, \]
  for $i\leq k$, where $\nabla$ denotes the covariant derivative with
  respect to $g_X$.  
\end{prop}
\begin{proof}
  As in \cite[Proposition 6]{Sz17} the result follows from the analysis of $X$
  in different regions in the proof of
  Propostion~\ref{prop:mXest}. We will only consider Region IV since
  the others are simpler. In Region IV $\rho$ is comparable to
  $y$, so we are essentially considering the region where $\Lambda y^a <
  r < y^\beta$ (and we are assuming that $y\geq 0$ as in
  Proposition~\ref{prop:mXest}). As in the proof of
  Proposition~\ref{prop:mXest}, we suppose that $r\in (R, 2R)$ and
  $y\in (y_0-R, y_0+R)$,  and we
  let $\tilde{X} = R^{-1}X$. In the proof of the proposition we showed
  that on our region $\tilde{X}$ is a small perturbation of $R^{-1}y_0^a\cdot H
  \times\mathbf{R}$ if $R$ (or equivalently $\rho$)
  is sufficiently small. By assumption $R^{-1}y_0^a < 2\Lambda^{-1}$,
  which we can make arbitrarily small by choosing $\Lambda$ large. It
  follows then that $\tilde{X}$ is a small perturbation of the cone
  $C\times\mathbf{R}$, and the nearest point projection $\tilde{F}$
  from $\tilde{X}$ to $C\times\mathbf{R}$ satisfies
  \[ |\nabla^i(\tilde{F}^*g_{C\times\mathbf{R}} - g_{\tilde{X}})|_{g_{\tilde{X}}} < \epsilon
    r^{-i}, \]
  once $\Lambda$ is sufficiently large. The required estimate follows
  by scaling back to $X$. 
\end{proof}

Next we consider a region $\mathcal{V}$ in $X$ of the form
\[ \mathcal{V} = \{ r < 2\Lambda |y_0|^a, |y-y_0|< B|y_0|^a\}\cap
  X, \]
where $\Lambda, B$ are large, and $y_0$ is such that $|y_0| <
A^{-1}$. On this region we define the nearest point projection
\[ G : \mathcal{V} \to H(y_0^\l)\times\mathbf{R}, \]
which satisfies the following.
\begin{prop}\label{prop:Gproj1}
  Given $k, \epsilon, \Lambda, B > 0$, if $A > A(k,\epsilon,\Lambda, B)$,
  then $|\nabla^i(G(x,y) - (x,y))| \leq \epsilon |x|^{1-i}$ and
  \[ |\nabla^i(G^*g_{H(y_0^\l)\times\mathbf{R}} - g_X)|_{g_X} < \epsilon
    r^{-i}, \]
  for $i\leq k$. 
\end{prop}
\begin{proof}
  This follows from the analysis of Region IV in the proof of
  Proposition~\ref{prop:mXest}. As in that proof, we assume that we
  are in a neighborhood where $r\in (R,2R)$, and we let $\tilde{X}=
  R^{-1}X$. The main difference is that $y$ can range in a larger interval: $|y-y_0| <
  By_0^a$. Since we have $R \geq c_0 y_0^a$, we still have $|y-y_0| <
  Bc_0^{-1}R$. As for the upper bound for $R$, if $A$ is sufficiently
  large, given $\Lambda$, then the bound $R < 2\Lambda |y_0|^a$
  implies $R < |y_0|^\beta$ so that we are in Region IV. Given these
  observations, the argument in Proposition~\ref{prop:mXest}
  still shows that for a given $B$, by taking $A$ sufficiently large
  (and therefore $R$ small), the rescaled surface $R^{-1}X$ can be
  seen as an arbitrarily small perturbation of $R^{-1}
  H(y_0^\l)\times\mathbf{R}$. The required result follows from this by scaling
  back down by a factor of $R$. 
\end{proof}

\subsubsection{The model space $C\times\mathbf{R}$} In this section
we consider the Jacobi operator on $C\times\mathbf{R}$, using weighted
spaces in terms of the weight function $r$. For simplicity let us write
$C^*$ for the complement of the vertex in $C$. For a function 
$f\in C^{k,\alpha}_{loc}$ on $C^*\times\mathbf{R}$ we define the weighted
$C^{k,\alpha}_\tau$ norm as follows. Let $g$ denote the induced metric
on $C^*\times\mathbf{R}$, and define
\[ \label{eq:wdefn1} \Vert f\Vert_{C^{k,\alpha}_\tau} = \sup_{R > 0} R^{-\tau} \Vert
  f\Vert_{C^{k,\alpha}_{R^{-2}g}(\{R < r < 2R\})}. \]
As before, the subscript $R^{-2}g$ indicates that we are measuring the
$C^{k,\alpha}$ norm of $f$ on the region where $r\in (R, 2R)$ using
the rescaled metric $R^{-2}g$. Note that on this region $R^{-2}g$ has
uniformly bounded geometry.

We have the following result, analogous to Corollary 12 in
\cite{Sz17}.
\begin{prop}\label{prop:CR1}
Suppose that $L_{C\times\mathbf{R}} f =0$ for $f\in
C^{k,\alpha}_{\tau}(C\times\mathbf{R})$, where $\tau\in
(3-n+\gamma, -\gamma)$. Then $f=0$.
\end{prop}
The proof is identical to that in \cite{Sz17}, using the strict
stability of $C$ which implies that 
there are no homogeneous Jacobi fields on $C$ with growth rates
in the interval $(3-n+\gamma, -\gamma)$. 

\subsubsection{The cone $C\times \mathbf{R}$} Here we again consider the
space $C\times \mathbf{R}$, however this time we use more general
doubly weighted H\"older spaces, with the norm given exactly as in
\eqref{eq:weightedn}. More precisely, if $f\in
C^{k,\alpha}_{loc}$ on $C^*\times\mathbf{R}$, then we define the norm
\[ \Vert f\Vert_{C^{k,\alpha}_{\delta, \tau}} = \sup_{R, S > 0}
  R^{-\tau} S^{\tau - \delta} \Vert f\Vert_{C^{k,\alpha}_{R^{-2}g}(\{
    R < r < 2R, S < \rho < 2S\})}. \]
Using this definition we have that the $C^{k,\alpha}_{\tau,\tau}$-norm
coincides with the $C^{k,\alpha}_\tau$ norm from the previous
section. In addition, since the scaled metric $R^{-2}g$ has bounded
geometry on the region where $R < r < 2R$, it follows that the Jacobi
operator defines a bounded linear map
\[ L_{C\times\mathbf{R}} : C^{k,\alpha}_{\delta, \tau} \to
  C^{k-2,\alpha}_{\delta-2, \tau-2}. \]
Analogously to Proposition 14 in \cite{Sz17} we have the following,
with almost exactly the same proof. 
\begin{prop} \label{prop:LCR10} For $\delta$ avoiding a discrete set of indicial roots,
  and $\tau\in (3-n+\gamma, -\gamma)$, the Jacobi operator
  \[ L_{C\times\mathbf{R}} : C^{k,\alpha}_{\delta, \tau}(C\times\mathbf{R}) \to
    C^{k-2,\alpha}_{\delta-2, \tau-2}(C\times\mathbf{R}) \]
  is invertible.
\end{prop}

This result has the following
consequence, using that our doubly weighted norm
$C^{k,\alpha}_{\tau,\tau}$ coincides with the $C^{k,\alpha}_\tau$
norm from the previous section.
\begin{cor}\label{cor:CR2}
  For generic $\tau\in(3-n+\gamma,-\gamma)$ the Jacobi operator
  \[ L_{C\times\mathbf{R}} : C^{2,\alpha}_\tau(C\times\mathbf{R}) \to
    C^{0,\alpha}_{\tau-2}(C\times\mathbf{R}) \]
  is an isomorphism.
\end{cor}

\subsubsection{The space $H\times\mathbf{R}$}  We now consider the
product space $H\times\mathbf{R}$, where $H$ denotes one of the Hardt-Simon
smoothings, $H_\pm$,  of $C$. We can define weighted H\"older norms
$C^{k,\alpha}_\tau$  for functions $f$ on $H\times\mathbf{R}$ using the formula
\eqref{eq:wdefn1} that we used on $C\times\mathbf{R}$. Note that for
the induced metric $g$ on $H\times\mathbf{R}$, the rescaled metrics
$R^{-2}g$ still have bounded geometry on the regions defined by $r\in
(R, 2R)$. Analogously to Proposition 21 in \cite{Sz17} we have the
following.
\begin{prop}\label{prop:LHR10}
  The Jacobi operator
\[ L_{H\times\mathbf{R}} : C^{k,\alpha}_\tau(H\times\mathbf{R}) \to
  C^{k-2,\alpha}_{\tau-2}(H\times\mathbf{R}) \]
is invertible for $\tau\in(3-n+\gamma, -\gamma)$. 
\end{prop}
It is worth pointing out that in this result, as in
Proposition~\ref{prop:CR1},
it is crucial that $C$ is
strictly minimizing and strictly stable, since this implies that on
$H$ the Jacobi operator
\[ \label{eq:Lh1} L_H : C^{k,\alpha}_{\tau}(H) \to C^{k-2,\alpha}_{\tau-2}(H) \]
is invertible for $(3-n+\gamma, -\gamma)$. The weighted spaces here
can be defined as above, or as the subspaces of translation invariant
functions in the corresponding spaces over $H\times\mathbf{R}$. To
see that \eqref{eq:Lh1} is invertible, note that the standard Fredholm
theory in weighted spaces due to Lockhart-McOwen~\cite{LM85} implies
that the cokernel of $L_H$ is given by the kernel of
\[ \label{eq:Lh2} L_H : C^{2,\alpha}_{3-n-\tau}(H) \to C^{0,\alpha}_{1-n-\tau}(H). \]
At the same time, since $C$ is strictly minimizing and strictly
stable, $H$ admits a positive Jacobi field asymptotic to
$r^{-\gamma}\phi_1$, corresponding to homothetic scalings of $H$
 (see Hardt-Simon~\cite{HS85}). Therefore, by the maximum principle,
 there cannot be any Jacobi field on $H$ decaying faster than
 $r^{-\gamma}$. In particular if $\tau\in (3-n+\gamma, -\gamma)$, then
 the operators \eqref{eq:Lh1} and \eqref{eq:Lh2} have trivial kernel,
 and so \eqref{eq:Lh1} is invertible. The rest of the proof of the
 proposition follows the proof of Proposition 21 in \cite{Sz17}
 closely. 

 \subsubsection{The Jacobi operator on $X$} We can now invert the linearized
operator $L_X$ in suitable weighted spaces, at least close to the
origin. The main result is the following.
\begin{prop}\label{prop:Linvert}
  Let $\tau\in(3-n+\gamma, -\gamma)$, and suppose that $\delta$ avoids
  a discrete set of indicial roots. Then for sufficiently large $A >
  0$, the Jacobi operator
  \[ L_X : C^{2,\alpha}_{\delta, \tau}(X \cap \rho^{-1}(0, A^{-1}])
    \to C^{0,\alpha}_{\delta-2, \tau-2}(X\cap \rho^{-1}(0,A^{-1}]) \]
  is surjective, with a right inverse $P$ bounded independently of $A$. 
\end{prop}
The proof is based on patching together the inverses given by
Propositions~\ref{prop:LCR10} and \ref{prop:LHR10} using cutoff
functions to construct an approximate right inverse of $L_X$. This is
possible because of Propositions~\ref{prop:projest1} and
\ref{prop:Gproj1} showing that $X$ can be well approximated by
$C\times \mathbf{R}$ and $H\times\mathbf{R}$ on suitable
regions. The details of the proof are essentially identical to the proof
of Proposition 22 in \cite{Sz17}. 

\subsection{Solution of the nonlinear problem}
In this section we construct
a minimal surface $T$ as a graph over $X$ of a function in our weighted spaces. More
precisely we have the following. 
\begin{prop}\label{prop:Texist}
  Let us choose $\beta\in (1,a),\delta > \l-\gamma$  as in
  Proposition~\ref{prop:mXest} and let $\tau < -\gamma$ be very close
  to $-\gamma$. Then 
  for sufficiently large $A > 0$ there is a function $u\in
  C^{2,\alpha}_{\delta,\tau}$ satisfying
  \[ \Vert u\Vert_{C^{2,\alpha}_{\delta, \tau}(X\cap \rho^{-1}(0,A^{-1}])}
    \leq A^{-\kappa/2}, \]
  such that the graph $T$ of $u$ over $X$ is minimal. $T$ has an
  isolated singularity at the origin, and tangent cone
  $C\times\mathbf{R}$ there. 
\end{prop}
\begin{proof}
  Given Proposition~\ref{prop:Linvert}, the proof of the existence of
  $u$ is an 
  application of the contraction mapping theorem (see also for instance
  Proposition 4.5 in \cite{Sz20}).

  Let us define the set
  \[ \mathcal{B} = \left\{ u\in C^{2,\alpha}_{\delta,\tau}(X\cap
        \rho^{-1}(0,A^{-1}])\,:\, \Vert
        u\Vert_{C^{2,\alpha}_{\delta,\tau}} \leq
        A^{-\kappa/2}\right\}. \]
  Our goal is to solve the equation $m_X(u) = 0$ for $u\in
  \mathcal{B}$. Writing $m_X(u) = m_X + L_X(u) + Q_X(u)$ as before,
  this is equivalent to the equation
  \[ L_X(u) = -m_X - Q_X(u), \]
  and so in terms of the right inverse $P$ given by
  Proposition~\ref{prop:Linvert} it is sufficient to solve $u =
  \mathcal{N}(u)$, where
  \[ \mathcal{N}(u) = - P(m_X + Q_X(u)). \]
It is enough to show that once $A$ is sufficiently large, the map
$\mathcal{N}$ defines a contraction from $\mathcal{B}$ to itself. 

Note first that by Proposition~\ref{prop:mXest} we have
\[ \label{eq:N0est} \Vert \mathcal{N}(0)
  \Vert_{C^{2,\alpha}_{\delta,\tau}(\rho^{-1}(0,A^{-1}])} \leq C
  A^{-\kappa}. \]
At the same time, using Lemma~\ref{lem:normcomp1} we have that if
$u\in \mathcal{B}$, then $\Vert
u\Vert_{C^{2,\alpha}_{1,1}} < CA^{-\kappa/2}$, and so for $A$
sufficiently large we can apply
Lemma~\ref{lem:Qquadratic}. It follows that for $u_i\in \mathcal{B}$
we have
\[ \Vert \mathcal{N}(u_1) -
  \mathcal{N}(u_2)\Vert_{C^{2,\alpha}_{\delta,\tau}} &\leq C\Vert
  Q_X(u_1) - Q_X(u_2)\Vert_{C^{0,\alpha}_{\delta-2,\tau-2}} \\
  &\leq C A^{-\kappa/2} \Vert u_1 -
  u_2\Vert_{C^{2,\alpha}_{\delta,\tau}} \\
  &\leq \frac{1}{2}\Vert u_1 -
  u_2\Vert_{C^{0,\alpha}_{\delta,\tau}} \]
if $A$ is sufficiently large. It remains to check that $\mathcal{N}$
maps $\mathcal{B}$ to itself. For this we can combine the two
estimates above to see that if $u\in \mathcal{B}$, then
\[ \Vert \mathcal{N}(u)\Vert_{C^{2,\alpha}_{\delta,\tau}} &\leq \Vert
  \mathcal{N}(u) - \mathcal{N}(0) \Vert_{C^{2,\alpha}_{\delta,\tau}} +
  \Vert \mathcal{N}(0) \Vert_{C^{2,\alpha}_{\delta,\tau}}  \\
  &\leq \frac{1}{2}\Vert u \Vert_{C^{2,\alpha}_{\delta,\tau}}  +
  CA^{-\kappa} \leq A^{-\kappa/2} \]
if $A$ is sufficiently large. It follows that $\mathcal{N}$ has a
fixed point as required.

Since $T$ is a graph over $X$, and $X$ is smooth away from the
  origin, $T$ is necessarily also smooth away from the origin. To see
  that the tangent cone of $T$ at the origin is $C\times\mathbf{R}$,
  it is enough to show that the sequence of rescalings $2^kT$ converges
  in the sense of currents to $C\times\mathbf{R}$ on the ball
  $B_1(0)$, as $k\to \infty$. This follows from
  Proposition~\ref{prop:projest1}. In particular choose $\epsilon >
  0$, and let $\Lambda, A$ be sufficiently large as in the
  proposition. According to the proposition, once $2^k  > A$, the
  rescaled surface $2^kX$ is the graph of a function $f_k$ over
  $C\times\mathbf{R}$ on the set $\{r > 2^{(1-a)k}\Lambda\rho\}\cap
  B_1(0)$, where $|f_k| \leq \epsilon r$. At the same time, by
  Proposition~\ref{prop:Texist} and Lemma~\ref{lem:normcomp1} we have
  that if $k$ is chosen even larger if necessary, then on the ball
  $B_1(0)$ the surface $2^kT$ is the graph of a function $u_k$ over $2^kX$
  with $|u_k| < \epsilon$.

  In sum we obtain that as $k\to\infty$, we have an increasing sequence of open subsets
  $\Omega_k\subset B_1(0)\cap \{r > 0\}$, such that $\Omega_k\to
  B_1(0)\cap\{r > 0\}$ and $2^kT$ is the graph of a function $U_k$
  over $(C\times\mathbf{R})\cap \Omega_k$, with $\sup_{\Omega_k}|U_k|
  \to 0$.  This is enough to conclude that the tangent cone at the
  origin is $C\times\mathbf{R}$. 
\end{proof}

\begin{remark}\label{rem:deltatau}
  In the rest of the paper we will often use constants $\delta, \tau$
  as in Proposition~\ref{prop:Texist}. We always first choose
  $\delta > \l - \gamma$ sufficiently close to $\l-\gamma$,
  and then $\tau < -\gamma$ is chosen close to $\gamma$, possibly
  depending on $\delta$.
\end{remark}

In later arguments we will need the following result about
the surface $T$.

\begin{prop}\label{prop:Tgraphbound}
  Let $A > 0$ be sufficiently large so that $T$ is
  defined in the ball $B(0, 4A^{-1})$. There are constants $C_1, \kappa$
  depending on $T$ (i.e. depending on the cone $C$ and the value of
  $\l$ in the construction of $T$) with the following property. For
  $|y_0| < A^{-1}$, the slice $T\cap\{y=y_0, r < A^{-1}\}$
  is the graph of a function $f$
  over $H(y_0^\l)$ (both viewed as subsets of $\mathbf{R}^n$), where
  \[\label{eq:fest2} |f| \leq C_1\rho^{\l-\kappa} r^{\kappa - \gamma}. \]
\end{prop}
\begin{proof}
  We consider two regions separately:
  \begin{itemize}
    \item On the region where $r > |y_0|$, we know from the proof of
      Proposition~\ref{prop:mXest} that the approximate solution $X$
      is the graph of a function $f_0$ over $C\times\mathbf{R}$, with
      $|f_0|\leq Cr^{\l-\gamma}$. We also know that on this region the
      surface $H(y_0^\l)$ is the graph of a function $f_1$ over
      $C\times\mathbf{R}$ with $|f_1| \leq C
      |y_0|^{\l}r^{-\gamma}$. Since $r > |y_0|$, we have $|f_1|\leq
      Cr^{\l-\gamma}$. Finally, by Proposition~\ref{prop:Texist}
      the surface $T$ is the graph of a function $u$ over $X$ with
      \[ |u|\leq \rho^{\delta-\tau}r^\tau. \]
      On the region $r > |y_0|$ we have that $\rho$ is comparable to
      $r$, and in addition $\delta > \l-\gamma$, so $|u| \leq C
      r^{\l-\gamma}$. Combining the three estimates we find that $T$
      can be written as the graph of a function $f$ over $H(y_0^\l)$,
      with $|f| \leq Cr^{\l-\gamma}$ as required.
    \item Consider the region $r < 2|y_0|$. We first show that the
      approximate solution $X$ is the graph of a function over
      $H(y_0^l)$ satisfying the estimate \eqref{eq:fest2}.  We use the analysis of Regions
      II, III, IV in the proof of Proposition~\ref{prop:mXest}. On
      Region II $X$ is the graph of $u_\l$ over $C\times X$, and also
      by \eqref{eq:ul} we have
      \[ |u_\l - y^\l r^{-\gamma} \phi_1| \leq C
        |y|^{\l-2}r^{2-\gamma}. \]
      At the same time in this region $H(y^\l)$ is the graph of a
      function $f_1$ over $C\times X$, where
      \[ |f_1 - y^\l r^{-\gamma}\phi_1| \leq C
        |y|^{a(\gamma+1+\kappa')}r^{-\gamma-\kappa'}, \]
      for some $\kappa' > 0$. Note that
      \[ |y|^{a(\gamma+1+\kappa')}r^{-\gamma-\kappa'} \leq C
        |y|^{\l-\kappa} r^{\kappa-\gamma}\]
      for suitable $\kappa > 0$. Indeed, this is equivalent to
      \[\label{eq:20} |y|^{a\kappa' +\kappa} \leq C r^{\kappa'+\kappa}, \]
        however on this region $r \geq 2|y|^\beta$ for some $\beta\in
        (1,a)$. It follows that \eqref{eq:20} holds if $\kappa$ is
        sufficiently small. In sum we find that on Region II the
        surface $X$ can be viewed as the graph of a function $f_2$
        over $H(y^\l)$, with
        \[ |f_2| \leq C |y|^{\l-\kappa}r^{\kappa-\gamma}. \]
        A very similar calculation shows that the same is also true on
        Region III, while on Region IV by definition $X$ is just the
        surface $H(y^\l)$.

        At the same time $T$ is the graph of a function $u$ over $X$
        with $|u| \leq \rho^{\delta-\tau}r^\tau$, and on our region
        $\rho$ is comparable to $|y|$. It is therefore enough to check
        that
        \[ |y_0|^{\delta-\tau}r^\tau \leq C
          |y_0|^{\l-\kappa}r^{\kappa-\gamma}, \]
      noting that on our region $|y_0|\sim \rho$. 
      Recall that we have $\delta > \l-\gamma$, and $\tau < -\gamma$
      can be taken very close to $-\gamma$. In particular it is enough
      to check the inequality above for $\tau=-\gamma$, in which case
      it is equivalent to
      \[  |y_0|^{\delta - (\l-\gamma) +\kappa} \leq C r^{\kappa}. \]
      Moreover, since we also have $r > C^{-1}|y|^a$, it is enough to
      ensure that
      \[ \delta - (\l-\gamma) + \kappa > a\kappa, \]
      and since $\delta > \l-\gamma$, this will be satisfied for
      sufficiently small $\kappa$. 
  \end{itemize}
\end{proof}

We also have the following variant of Proposition~\ref{prop:Fadefn},
based on the fact that at suitable scales $T$ looks either like
$C\times \mathbf{R}$ or $H\times\mathbf{R}$.
\begin{lemma}\label{lem:Fa2}
  Let $a\in (3-n+\gamma, -\gamma)$. There is a constant $C_a$,
  depending on $a$ and the cone $C$, and functions $F_{\pm,a} > C_a^{-1}r^a$ on $H_{\pm}$
  satisfying the following. We have $|\nabla^i F_{\pm,a}|\leq C_ar^{a-i}$ for
  $i\leq 3$ and $L_{H_\pm} F_{\pm,a} < - C_a^{-1} r^{a-2}$. In addition $F_{\pm,a} = r^a\phi_1$
  outside a large ball, where we view $H_\pm$ as graphs over
  $C$ in order to define $\phi_1$ on $H_\pm$. 

  As in Proposition~\ref{prop:Fadefn}, we define $F_a$ on $\mathbf{R}^n\setminus\{0\}$ to be
  homogeneous of degree $a$, restricting to $F_{\pm,a}$ on $H_\pm$.
  Then in a sufficiently small neighborhood
  of the origin the function $F_a$ satisfies the same estimates on the
  surface $T$. That is, $|\nabla^i F_a| \leq C_a r^{a-i}$ and $L_T
  F_a < - C_a^{-1} r^{a-2}$, for a larger constant $C_a$.  
\end{lemma}
\begin{proof}
  In the same way as in Proposition~\ref{prop:Fadefn} we can construct
  $F_{\pm, a}$ on $H_\pm$ such that $|\nabla^i F_{\pm a}|\leq C_a
  r^{a-i}$, $L_{H_\pm}F_{\pm,a} < -C_a^{-1} r^{a-2}$ and $F_{\pm,a} =
  r^a\phi_1$ outside of a large ball. For this note that $L_C r^a =
  -c_a r^{a-2}$ for $a\in(3-n+\gamma, -\gamma)$, where $c_a > 0$.
  To see that $F_{\pm,
    a} > 0$ we argue as in Step 1 of \cite[Proposition 5.11]{Sz20}: We
  use the strict minimizing property of $C$, which
  implies that $H_{\pm}$ admit positive Jacobi fields $\Phi_{\pm}$
  satisfying $|\Phi_\pm| < C r^{-\gamma}$, arising from homothetic
  rescalings. If $F_{\pm,a}$ were non-positive somewhere, then we could
  find some $\lambda \geq 0$ such that $F_{\pm,a}\geq
  -\lambda\Phi_{\pm}$ with equality at some point (note that
  $F_{\pm,a}$ decays faster than $\Phi_{\pm}$ at infinity). At the
  contact point we would have $L_{H_\pm}F_{\pm,a} \geq
  L_{H_\pm}(-\lambda\Phi_{\pm})=0$, which is a contradiction.

  It remains to estimate $\nabla^i F_a$ and $L_T F_a$ on our
  hypersurface $T$. For $\nabla^i F_a$ note that $F_a$ is homogeneous
  of degree $a$ and the regularity scale of $T$ at each point is
  bounded below by $c_1 r$ for some $c_1 > 0$. Therefore the rescaled
  surface $\lambda^{-1}T$ has uniformly bounded geometry in the region
  $r\in (1,2)$. Here $F_a$ has uniformly bounded derivatives, and so
  scaling back we obtain the estimates $|\nabla^i F_a|\leq C_a
  r^{a-i}$ on $T$.

  To estimate $L_T F_a$ we need to use more information about $T$ in
  different regions. Instead of $T$ we first consider $X$, and use the
  description of $X$ in different regions given in
  Proposition~\ref{prop:mXest}, in the ball $B(0,A^{-1})$ for large
  $A$. The argument there shows that in each
  region, after rescaling to the ``unit scale'' $r\in(1,2)$, the
  rescaled surface $\tilde{X}$ can be approximated arbitrarily well by
  $H(t)\times\mathbf{R}$ for some $t\in \mathbf{R}$ ($t=0$ in Regions
  I-III, and $t=R^{-\gamma-1}y_0^\l$ in Region IV), as long as $A$ is
  sufficiently large. On the annulus $r\in (1,2)$ we have
  $L_{H(t)}F_a < -C_a^{-1}$, and so we will also have $L_{\tilde{X}}
  F_a < -\frac{1}{2}C_a^{-1}$ if $A$ is sufficiently large. At the
  same time the estimate for $u$ in Proposition~\ref{prop:Texist}
  implies that after rescaling, $R^{-1}T$ is also an arbitrarily small
  perturbation of $\tilde{X}$ in the region $r\in (1,2)$, as long as
  $A$ is sufficiently large. We can therefore ensure that $L_{R^{-1}T}
  F_a < -\frac{1}{4}C_a^{-1}$. Scaling back and using the homogeneity
  of $F_a$ we get the required estimate $L_T F_a < -\frac{1}{4}
  C_a^{-1} r^{a-2}$. 
\end{proof}

\subsection{$T$ is area minimizing}\label{sec:minimizing}
In this section we complete the proof of
Theorem~\ref{thm:Texist} by showing that
$T$ is area minimizing near the
origin when $\l$ is sufficiently large.
\begin{prop}\label{prop:Tminimizing}
  For sufficiently large $\l$, depending on the cone $C$, the
  hypersurface $T$ constructed in Proposition~\ref{prop:Texist} is
  area minimizing in a neighborhood of the origin. 
\end{prop}
\begin{proof}
We construct a foliation near $T$ whose leaves have negative mean
curvature for the normal pointing away from $T$. The leaves of the
foliation will be built out of graphs away from the origin,
and barrier surfaces given by Proposition~\ref{prop:barrier10} near the
origin. These two pieces will be combined by taking their
``minimum''.

Let $\gamma < \gamma_1 < \gamma_2 < \gamma_3$, where $\gamma_i$ are
all very close to $\gamma$. Let $\xi > 0$ be small, and $\epsilon =
\xi^\lambda$, where $\lambda$ satisfies
\[ \label{eq:l11} 1 - \frac{\gamma_1 -\gamma}{\gamma_2+1} < \lambda < 1 -
  \frac{\gamma_1-\gamma}{\gamma_3+1}. \]
We will build a
hypersurface $S_\xi$ on the positive side of $T$ (the negative side being
entirely analogous) out of two pieces: $S_{\xi,1}$ is 
the graph of $\xi F_{-\gamma_1}$ over $T$, where $F_{-\gamma_1}$ is
provided by Lemma~\ref{lem:Fa2}, while $S_{\xi,2}$ will be a
hypersurface, which in each $y$-slice is contained between the
surfaces $H(y^\l + 3\epsilon \pm 2\epsilon)$. We work on a ball
$B(0,A^{-1})$ for sufficiently large $A$, so that Lemma~\ref{lem:Fa2}
applies. We will show the
following properties:
\begin{enumerate}
  \item  $S_{\xi,1}$ is defined on the region where $r >
\xi^{1/(\gamma_2+1)}$, while $S_{\xi,2}$ is defined where $r <
\xi^{1/(\gamma_3+1)}$, once $\xi$ is sufficiently small. On these
regions both hypersurfaces are strictly on the positive side of $T$.
  \item On the annular region 
\[ \label{eq:ann1} \xi^{\frac{1}{\gamma_2+1}} < r <
  \xi^{\frac{1}{\gamma_3+1}}, \]
  both $S_{\xi,1}, S_{\xi,2}$ are graphs over the hypersurface with $y$-slices
  $H(y^\l)$.
 \item Near the outer boundary $r = \xi^{1/(\gamma_3+1)}$
the graphical surface $S_{\xi,1}$ lies on the negative side of
$S_{\xi,2}$, while on the inner boundary the opposite is true, and so
we can define the hypersurface $S_\xi$ to be the minimum of
$S_{\xi, 1}$ and $S_{\xi,2}$, such that every point of $S_\xi$ in the ball
$B(0,A^{-1})$ is an interior point of either $S_{\xi,1}$
or $S_{\xi,2}$.
\end{enumerate}

Let us consider $S_{\xi,1}$ first, given by the graph of
$\xi F_{-\gamma_1}$ over $T$. Recall that the regularity scale of $T$ is
bounded below by $c_0 r$ for some $c_0 > 0$, so that the graph is well
defined as long as $\xi F_{-\gamma_1} < c_1 r$ for a small $c_1 >
0$. Since $F_{-\gamma_1}$ is comparable to $r^{-\gamma_1}$, this is
satisfied as long as $\xi \ll r^{\gamma_1 + 1}$, which in particular
holds if
\[ r \geq \xi^{\frac{1}{\gamma_2+1}} \]
as long as $\gamma_2 > \gamma_1$ and $\xi$ is sufficiently
small. Since $|\nabla^i F_{-\gamma_1}|\leq C_{\gamma_1} r^{-\gamma_1 -
  i}$ and $L_T F_{-\gamma_1} < -C_{\gamma_1}^{-1} r^{-\gamma_1-2}$, it
also follows (scaling up annular regions $r\in (R, 2R)$ by $R^{-1}$ as
in the proof of Proposition~\ref{prop:mXest})
that the mean curvature of $S_{\xi,1}$ is negative
on the region $r \geq \xi^{1/(\gamma_2+1)}$ as long as $\xi$ is
sufficiently small. Since $F_{-\gamma_1}$ is strictly positive, it is
clear that $S_{\xi,1}$ is strictly on the positive side of $T$. 

In order to construct $S_{\xi,2}$ we apply
Proposition~\ref{prop:barrier10} for a suitable function
$f_\epsilon(y)$ on the interval $|y|\leq y_\epsilon$, where
\[ \label{eq:yed} y_\epsilon = C_1 \xi^{\frac{1}{a(\gamma_3+1)}}, \]
for a large $C_1 > 0$ to be chosen below and $a= \l / (1+\gamma)$ as
before.  Since we want the resulting
hypersurface to be contained between $H(y^\l+3\epsilon\pm 2\epsilon)$,
we want to use a function $f_\epsilon$ such that
\[ \epsilon^{-1}y^\l + 2 \leq
  f_\epsilon(y)^p \leq \epsilon^{-1}y^\l + 3, \]
with suitable control on the $C^3$-norm of $f_\epsilon$. We define
$f_\epsilon$ by
\[ f_\epsilon(y) = \sigma(\epsilon^{-1}y^\l + 2.5), \]
where $\sigma : \mathbf{R}\to\mathbf{R}$ is a fixed smooth function
satisfying $\sigma(t) = t^{1/p}$ for $|t| > 1$, and
  $|\sigma(t)^p - t| < \frac{1}{10}$ for all
  $t$. We then have estimates
  \[ |\sigma(t)| &\leq |t|^{1/p} + 1, \\
      |\partial_t^i \sigma(t)| &\leq C
    \]
  for $i=1,2,3$. Using this we can estimate the derivatives of
  $f_\epsilon$. We are working on $|y| < y_\epsilon$, where by
  \eqref{eq:yed} we can assume that
  $\epsilon^{-1} y_\epsilon^\l \geq 10$ if $\gamma_3$ is sufficiently
  close to $\gamma$. On this interval we have
  \[ |f_\epsilon(y)| &\leq
    C(\epsilon^{-1}y_\epsilon^\l)^{\frac{1}{p}} \\
    |\partial_y f_\epsilon(y)| &\leq C\l
    \epsilon^{-1}y_\epsilon^{\l-1}\\
    |\partial_y^2 f_\epsilon(y)| &\leq C\l^2 (
    \epsilon^{-1}y_\epsilon^{\l-2} + (\epsilon^{-1}y_\epsilon^{\l-1})^2)
    \leq C\l^2 (\epsilon^{-1}y_\epsilon^{\l-1})^2
    \\
    |\partial_y^3 f_\epsilon(y)| &\leq C\l^3
    (\epsilon^{-1}y_\epsilon^{\l-1})^3. 
  \]
  In sum we get $|f_\epsilon|_{C^3} \leq C
  (\l\epsilon^{-1}y_\epsilon^{\l-1})^3$. According to
  Proposition~\ref{prop:barrier10} we will obtain a hypersurface
  $X_\epsilon$ with negative mean curvature, defined on the region $|y|
  < y_\epsilon$ and $r <
  (\l \epsilon^{-1}y_\epsilon^{\l-1})^{-Q}$ for some fixed large
  $Q$, once $\epsilon$ is small enough.
  Since in each $y$-slice $X_\epsilon$ lies between $H(y^\l +3\epsilon
  \pm 2\epsilon)$, it follows 
  that in the $y$-slices where $|y| = y_\epsilon$, we have
  \[ \label{eq:rl2} r > C^{-1} y_\epsilon^{\frac{\l}{\gamma+1}} = C^{-1}
    y_\epsilon^a. \]

  We need to choose the constant $C_1$ in \eqref{eq:yed}
  and $\gamma_3$ in a suitable
  way, so that the resulting surface $X_\epsilon$ is defined and has
  no boundary in the region $r < \xi^{1/(\gamma_3+1)}$.
  Using \eqref{eq:yed}, \eqref{eq:rl2} together with the
  constraint $r < (\l\epsilon^{-1}y_\epsilon^{\l-1})^{-Q}$, it follows
  that we need to
  choose the constants so that
  \[ \label{eq:c10} \xi^{1/(\gamma_3+1)} &< (\l\xi^{-\lambda}y_\epsilon^{\l-1})^{-Q}, \\
    \xi^{1/(\gamma_3+1)} &< C^{-1}y_\epsilon^a, \]
  for sufficiently small $\xi$. 
  The second condition follows from \eqref{eq:yed} if $C_1$ is
  large. As for the first condition, using
  \eqref{eq:yed}, it is equivalent to
  \[ \xi^{\frac{1}{\gamma_3+1}} <
    \l^{-Q}\xi^{Q\lambda}C_1^{-Q(\l-1)}\xi^{\frac{-Q(\l-1)}{a(\gamma_3+1)}}. \]
  Once $Q, l$ are fixed, this inequality will hold for sufficiently
  small $\xi$, as long as
  \[ Q\lambda - \frac{Q(\l-1)}{a(\gamma_3+1)} - \frac{1}{\gamma_3+1} <
    0, \]
  i.e.
  \[ \label{eq:l21} \lambda < \frac{\l-1}{a(\gamma_3+1)} +
    \frac{1}{Q(\gamma_3+1)}. \]
  If we choose $\gamma_3$ sufficiently close to $\gamma$, and $\l,a$
  are sufficiently large (recall $\l = a(\gamma+1)$), then any $\lambda
  < 1$ satisfies \eqref{eq:l21}. We then define $S_{\xi, 2}$ to be the
  hypersurface $X_\epsilon$ on the region $r <
  \xi^{1/(\gamma_3+1)}$. To see that $S_{\xi,2}$ is graphical over $T$
  on the region where $r > \xi^{\frac{1}{\gamma_2+1}}$, it is enough
  to check that $H(y^\l + 4\epsilon)$ is graphical over $H(y^\l)$ on
  this region. This in turn follows if on this region we have
  $\epsilon r^{-\gamma} \ll r$, i.e. $\xi^\lambda \ll
  \xi^{\frac{\gamma+1}{\gamma_2+1}}$. This is implied by
  \eqref{eq:l11} and that $\gamma_2 > \gamma_1$. 

  We next claim that $S_{\xi,2}$ lies
  strictly on the positive side of $T$. To see this recall from
  Proposition~\ref{prop:Tgraphbound} that in each $y$-slice $T$ is the
  graph of a function $f$ over $H(y^\l)$, with
  \[ |f| \leq C(r^{\l-\gamma} + |y|^{\l-\kappa}r^{\kappa-\gamma}). \]
  Using that $|y| < C r^{a^{-1}}$, it follows that
  \[ |f| \leq Cr^{1 + \kappa(1-a^{-1})} \leq C r^{1 + \kappa/2}, \]
  where we used that $\l, a$ are large. At the same time $S_{\xi,2}$
  lies on the positive side of $H(y^\l + \epsilon)$,
  which in turn is on the positive side of
  the graph of
  \[ c_0 \min \{ \epsilon r^{-\gamma}, r\} \]
  over $H(y^\l)$ by Lemma~\ref{lem:L51}. It remains to check that as long as $r <
  \xi^{1/(\gamma_3+1)}$ and $\xi$ is sufficiently small, we have
  \[ C r^{1 + \kappa/2} < c_0 \min\{\xi^\lambda r^{-\gamma}, r\}. \]
  It is clear that $Cr^{1+\kappa/2} < c_0 r$ if $r$ is small, and the
  remaining inequality is
  \[ r^{\gamma+1+\kappa/2} < c_0C^{-1} \xi^\lambda. \]
  For this to hold on the region $r < \xi^{1/(\gamma_3+1)}$ we need
  \[ \frac{\gamma+1+\kappa/2}{\gamma_3+1} > \lambda. \]
  If $\gamma_3-\gamma$ is sufficiently small, then this holds as soon
  as $\lambda < 1$. 

  We now have our two hypersurfaces $S_{\xi,1}, S_{\xi,2}$ with no
  boundary on the annular region $\xi^{1/(\gamma_2+1)} < r <
  \xi^{1/(\gamma_3+1)}$. Let us
  examine the relative position of the surfaces near the inner
  boundary $r = \xi^{1/(\gamma_2+1)}$. We claim that here the
  graphical surface $S_{\xi,1}$ is on the negative side of
  $S_{\xi,2}$. Using Proposition~\ref{prop:Tgraphbound} as above we know that
  in each $y$-slice $T$ is the graph of a function $f$ over $H(y^\l)$,
  with $|f| \leq Cr^{1 + \kappa/2}$. At
  the same time $F_{-\gamma_1} > C_{\gamma_1}^{-1}r^{-\gamma_1}$. It
  follows that near the inner boundary $S_{\xi,1}$ can be written as
  the graph of a function $f_1$ over $H(y^l)$, where
  \[ f_1 > C_{\gamma_1}^{-1}\xi^{1 - \frac{\gamma_1}{\gamma_2+1}} - C
    \xi^{\frac{1+\kappa/2}{\gamma_2+1}}. \]
  We have
   \[ \label{eq:i12} \frac{1+\kappa/2}{\gamma_2+1} >
      \frac{1+\gamma_2-\gamma_1}{\gamma_2+1} \]
    as long as $\gamma_2-\gamma_1$ is sufficiently small ($\kappa$ is
    fixed).  It follows that for
    sufficiently small $\xi$ we have
    \[ f_1 > \frac{1}{2}
      C_{\gamma_1}^{-1}\xi^{1-\gamma_1/(\gamma_2+1)}. \]
    At the same time near the surface $S_{\xi,2}$ lies on the negative
    side of $H(y^\l + 5\epsilon)$. From
    \eqref{eq:l11} it follows that as long as $\gamma_1-\gamma$ is
    sufficiently small, then $\lambda > 1/(\gamma_1+1)$, and so 
    \[ \epsilon = \xi^\lambda \ll \xi^{\frac{1}{\gamma_2+1}}. \]
    Because of this, near the inner boundary $r =
    \xi^{1/(\gamma_2+1)}$ the surface $H(y^\l + 5\epsilon)$ lies on the negative side
    of the graph of the function $C \epsilon r^{-\gamma}$
    over $H(y^\l)$. We just need to check that this is smaller than the
    function $f_1$ above, i.e. that
    \[ C\xi^\lambda \xi^{-\frac{\gamma}{\gamma_2+1}} <
      \frac{1}{2}C_{\gamma_1}^{-1}\xi^{1-\frac{\gamma_1}{\gamma_2+1}}. \]
    This in turn follows for sufficiently small $\xi$ if
    \[ \lambda > 1 - \frac{\gamma-\gamma_1}{\gamma_2+1}, \]
    which is what we assumed in \eqref{eq:l11}.

    Let us now consider the outer boundary $r = \xi^{1/(\gamma_3+1)}$,
    where we claim that the graphical surface $S_{\xi,1}$ lies on the
    negative side of $S_{\xi,2}$. Arguing as above, we find that here
    $S_{\xi,1}$ is the graph of a function $f_2$ over $H(y^l)$
    satisfying
    \[ f_2 < C_{\gamma_1} \xi^{1 - \frac{\gamma_1}{\gamma_3+1}} +
      C\xi^{\frac{1+\kappa/2}{\gamma_3+1}}. \]
    The inequality \eqref{eq:i12} holds with $\gamma_2$ replaced by
    $\gamma_3$, as long as $\gamma_3-\gamma_1$ is also sufficiently
    small, and so it follows that for sufficiently small $\xi$ we have
    \[ f_2 < 2C_{\gamma_1} \xi^{1-\frac{\gamma_1}{\gamma_3+1}}.\]
    At the same time the surface $S_{\xi,2}$ is on the positive side
    of $H(y^\l + \epsilon)$. Arguing as above,
    we see that near the outer boundary this is on the positive side
    of the graph of $C\epsilon r^{-\gamma}$ over $H(y^\l)$. We need to
    check that this is greater than $f_2$, i.e. 
    \[ C^{-1} \xi^\lambda \xi^{-\frac{\gamma}{\gamma_3+1}} >
      2C_{\gamma_1} \xi^{1-\frac{\gamma_1}{\gamma_3+1}}. \]
    This holds for small $\xi$ as long as
    \[ \lambda < 1 - \frac{\gamma-\gamma_1}{\gamma_3+1}, \]
    which is satisfied by \eqref{eq:l11}.

    We can now build a comparison surface $S_\xi$ in the ball
    $B(0,A^{-1})$ for sufficiently large $A$ as follows ($A$ is
    determined by the neighborhood of the origin on which the function
    $F_{-\gamma_1}$ satisfies the properties in
    Lemma~\ref{lem:Fa2}). On the region
    $r > \xi^{1/(\gamma_3+1)}/4$ we let $S_\xi = S_{1,\xi}$, while on $r
    < 4\xi^{1/(\gamma_2+1)}$ we let $S_\xi = S_{2,\xi}$. To define
    $S_\xi$ on the annular region
    \[ \label{eq:ann2} \xi^{1/(\gamma_3+1)}/2 < r < 2\xi^{1/(\gamma_2+1)} \]
    note that here both $S_{\xi,1}$ and $S_{\xi,2}$ are graphs over
    $H(y^\l)$ in the $y$-slices. In particular it makes sense to define
    $S_\xi$ as their minimum. From the discussion above we know that near the outer boundary of
    \eqref{eq:ann2} the surface $S_\xi$ coincides with $S_{\xi,1}$,
    while near the inner boundary it coincides with $S_{\xi,2}$. It
    follows that every point of $S_\xi$ in the ball $B(0,A^{-1})$ is an interior point of either
    $S_{\xi,1}$ or of $S_{\xi,2}$.

    Using the comparison surfaces we
    complete the proof that $T$ is area minimizing in a
    neighborhood of the origin. We first choose a $\xi_0$ sufficiently
    small so that $S_{\xi_0}$ is defined. Note that $S_{\xi_0}$ is
    strictly on the positive side of $T$. If $T$ were not area
    minimizing in any neighborhood of the origin, then we could find an
    area minimizing hypersurface $T'$ in a very small ball
    $B(0,\rho_0)$, with boundary $T'\cap \partial B(0,\rho_0) = T\cap
    \partial B(0,\rho_0)$, but $T'$ not equal to $T$. We can assume
    that $T'$ contains points lying strictly on the positive side of
    $T$ (otherwise we can repeat the whole construction on the
    negative side of $T$, reversing orientations), while if $\rho_0$ is 
    sufficiently small, $T'$ will lie on the negative side of
    $S_{\xi_0}$. The hypersurfaces $S_\xi$ vary continuously with
    $\xi$, and so we can find a largest $\xi_1 < \xi_0$, such that
    $S_{\xi_1}$ has non-empty intersection with $T'$. This
    intersection must happen at an interior point $z$ of $T'$, and $z$
    is also an interior point of either $S_{\xi_1, 1}$ or $S_{\xi_1,
      2}$. Both of these possibilities are a contradiction
    since the $S_{\xi, i}$ have negative mean curvature at
    smooth points, and $T'$ also cannot have a contact point with
    $S_{\xi_1, 2}$ at its singular points by
    Proposition~\ref{prop:barrier10}. 
\end{proof}

\section{Strong unique continuation}\label{sec:uniquecont}
In this section we prove a strong unique continuation result for
minimal hypersurfaces $M$ with cylindrical tangent cone
$C\times\mathbf{R}$ at the origin, where $C$ is a strictly minimizing
strictly stable cone in $\mathbf{R}^n$. To state the precise result, let us
denote by $d:\mathbf{R}^n\times\mathbf{R} \to \mathbf{R}$ the distance
function to the cone $C\times \mathbf{R}$.

\begin{thm}\label{thm:uniquecont1}
  Suppose that $M$ is an $n$-dimensional stationary integral varifold in a neighborhood of the
  origin $0\in\mathbf{R}^n\times\mathbf{R}$, which admits
  $C\times\mathbf{R}$ as a (multiplicity one) tangent cone at
  the origin. Suppose that for all $k > 0$ there is a constant $C_k$
  such that for all $\rho < 1$ we have
  \[ \label{eq:L2vanish} \int_{M\cap B_\rho(0)} d^2 < C_k \rho^k, \]
  i.e. the $L^2$-distance from $M$ to $C\times\mathbf{R}$ on the ball
  $B_\rho(0)$ vanishes to infinite order as $\rho\to 0$. Then $M = C\times\mathbf{R}$. 
\end{thm}

The proof relies on a similar idea as the monotonicity of frequency
used by Almgren~\cite{Alm00} and Garofalo-Lin~\cite{GL86}, although the
details are quite different, since we are not able to define a
suitable frequency function in our setting. To explain the basic idea, let us denote
by $d(M,\rho)$ some measure of the distance between $M$ and
$C\times\mathbf{R}$ on the ball $B_\rho(0)$. Suppose that $d(M,\rho)$ is
defined in a scale invariant way, so that $d(M,\rho)= d(\rho^{-1}M, 1)$.
In practice $d(M,\rho)$ will be
a scaled $L^2$-distance, ``regularized'' by adding a small multiple of
an $L^\infty$-type distance (see \eqref{eq:ddefn} below). It is possible that one
could also use the $L^2$-distance itself by relying on the
non-concentration result due to Simon~\cite[Corollary 2.3]{Simon94} in
the arguments below. 

Suppose that $\lambda > 0$.
Let us say that the three-annulus property holds for the pair
$(M,\lambda)$, if $d(M,e^{-\lambda}) \geq \frac{1}{2}d(M,1)$ implies
$d(M,e^{-2\lambda}) \geq \frac{1}{2} d(M,e^{-\lambda})$. Note that if for some
$\lambda > 0$ the three-annulus property holds for $(e^{k\lambda}M,
\lambda)$ for all $k\geq 0$, and in addition $d(M,e^{-\lambda}) \geq
\frac{1}{2} d(M, 1)$, then iterating the three-annulus property we
find that $d(M, e^{-k\lambda}) \geq 2^{-k} d(M, 1)$. This should imply
that $M$ cannot approach its tangent cone at infinite order for any
reasonable definition of the distance $d$.

From the three-annulus lemma (see Proposition~\ref{prop:QL23annulus})
for Jacobi fields on $C\times\mathbf{R}$, together with a
contradiction argument, one expects that perturbing $\lambda$ slightly
if necessary, the three-annulus property holds for $(M,\lambda)$,
whenever $M$ is sufficiently close to $C\times\mathbf{R}$, say
whenever $d(M,1) < E(\lambda)$, for a function $E$ converging to
zero as $\lambda\to 0$. A precise version of this statement will be
the core of our argument. See Proposition~\ref{prop:3anndecay} below,
where we will show that for a suitable definition of $d$, we can
choose $E(\lambda) = \lambda^Q$ for some $Q > 0$.

Given this, we can conclude as follows (see Corollary~\ref{cor:doubling} for
details).  Assuming that $\rho^{-1} M$ is
sufficiently close to $C$ for all $\rho < 1$, the three-annulus
property will hold for the pairs
$(\rho^{-1}M, \lambda_0)$, for some $\lambda_0 > 0$.  
If for a given $\rho_0 \in (0,1)$ we had $d(M,
e^{-\lambda_0}\rho_0) \geq \frac{1}{2} d(M, \rho_0)$, then $M$ would not
approach its tangent cone at infinite order. Therefore for all $\rho <1$ we must have
\[ d(M, e^{-\lambda_0}\rho) < \frac{1}{2} d(M, \rho). \]
Proposition~\ref{prop:3anndecay} will then imply that for all $\rho <
e^{-\lambda_0}$ the three-annulus property holds for $(\rho^{-1}M,
\lambda_1)$ for some $\lambda_1 < s \lambda_0$, with $s < 1$
depending only on the number $Q$ in
Proposition~\ref{prop:3anndecay}. Iterating this, it follows that if $M$
approaches its tangent cone at infinite order, then we have $d(M,
\rho_0) = 0$ for $\rho_0 = e^{-\lambda_0(1+s+s^2+\ldots)}$, leading to
Theorem~\ref{thm:uniquecont1}.

Most of this section will be taken up by the proof of
Proposition~\ref{prop:3anndecay}, which is essentially a quantitative version
of the three annulus lemma (see Proposition 7.5 in
\cite{Sz20}). In \cite{Sz20} we used a contradiction argument similarly
to related earlier results 
in the literature such as in Simon~\cite{Simon83}, together
with a non-concentration estimate. 
To obtain a quantitative version of the result we cannot argue by
contradiction, and a first step is to obtain a quantitative
version of the non-concentration estimate, Proposition 7.4 in
\cite{Sz20}. We prove this in Section~\ref{sec:qnonconc}. We will
then give the proof of Proposition~\ref{prop:3anndecay} in
Section~\ref{sec:q3ann}. 

\subsection{Quantitative non-concentration}\label{sec:qnonconc}
As in \cite{Sz20}, we define a kind of $L^\infty$-distance between
hypersurfaces $M$ and the cone $C\times \mathbf{R}$ using comparisons
to rescalings of the Hardt-Simon smoothings $H_\pm$ of the cone $C$.
\begin{definition}\label{defn:DCRdist}
  For $\epsilon > 0$, let us define the neighborhood
  $N_\epsilon(C\times\mathbf{R})$ to be the (closed) region contained
  between the two hypersurfaces $H(\pm \epsilon)$ using the notation
  in Definition~\ref{defn:Ht}. 

  For any subsets $M, U\subset\mathbf{R}^{n+1}$ (in practice $M$ will
  be the support of a stationary integral varifold and $U$ an open set), we define
  \[ D_{C\times\mathbf{R}}(M; U) = \inf\{ \epsilon > 0\,|\, M\cap U
    \subset N_\epsilon(C\times\mathbf{R})\}. \]
  We think of this as the distance from $M$ to $C\times\mathbf{R}$ over
  the set $U$. 
\end{definition}

The following is a quantitative version of the non-concentration
result, Proposition 7.4 in \cite{Sz20}. 
\begin{prop}\label{prop:nonconc}
  There is a constant $A>0$ depending only on the cone $C$ with the
  following property. Let $b > 0$ and $s\in (0,1/2)$, and suppose that
  $M\subset\mathbf{R}^n\times\mathbf{R}$ is a codimension one
   stationary integral varifold defined in the region $\{|y| <
   b\}$. We then have
   \[ \label{eq:nonconc1}
     D_{C\times\mathbf{R}}(M; \{|y| < b/2\}) &\leq A
     D_{C\times\mathbf{R}}(M; \{r \geq b s^A\}\cap \{|y| < b\})
     \\ &\qquad + s D_{C\times\mathbf{R}}(M; \{|y| < b\}). \]
\end{prop}

\begin{proof}
  The proof is similar to that in \cite{Sz20}, however we need to keep
track of the constants. 
  By scaling we can reduce the result to the case when $b=1$. Given $s
  > 0$, we apply
  Proposition~\ref{prop:barrier10} to the function
  \[ f(y) = 4(1+y)^{-1} + 4(1-y)^{-1}, \]
  on the interval $y\in (-1+s, 1-s)$. Note that on this interval we
  have $|f|_{C^3} \leq C_1s^{-3}$ for a fixed constant
  $C_1$, so we can take $K = s^{-a_3}$ in the Proposition for some large
  $a_3$, since $s < 1/2$. Using that $f > 4$, we have
  \[ \label{eq:fc1} \frac{1}{2}f^p + 2 < f^p < 2f^p - 2, \]
  for the $p$ from the Proposition. 
  Let us denote by $X_\epsilon$ the surface constructed in
  Proposition~\ref{prop:barrier10}, so that in the $y$-slices $X_\epsilon$ lies between
  the hypersurfaces $H(\epsilon f(y)^p \pm \epsilon)$, on the region
  $r < K^{-Q^2} = s^{a_3Q^2}$.

    \bigskip
    Let us denote by $U_s$ the region $\{r < K^{-Q^2}\}\times \{|y| <
    1-s\}$. In $U_s$ the surface $X_\epsilon$ has negative mean curvature
    and, using \eqref{eq:fc1}, lies between
  the hypersurfaces with $y$-slices
  $H(\epsilon f(y)^p/2)$ and $H(2\epsilon f(y)^p)$. 
  In particular we have
    \begin{enumerate}
      \item On the boundary piece $\{r < K^{-Q^2}\} \times \{ |y| =
        1-s\}$, the surface $X_\epsilon$ lies on the positive side of
        $H(2\epsilon s^{-1})$.
      \item On all of $U_s$ the surface $X_\epsilon$ lies on the
        positive side of
        $H(2\epsilon)$, using that
        $f > 4$. 
      \item On the smaller region $\{r < K^{-Q^2}\}\times \{|y| <
        \frac{1}{2}\}$ the surface $X_\epsilon$ lies on the negative
        side of $H(C_1\epsilon)$ for $C_1$ depending
        on $p$ (i.e. on the cone).
      \end{enumerate}

      Let us write $r_0 = K^{-Q^2} = s^{a_3Q^2}$ and define
      \[ D &= D_{C\times \mathbf{R}}(M; \{|y| < 1\}), \\
          D_{r_0} &= D_{C\times \mathbf{R}}(M; \{ r \geq r_0\}\times
          \{|y| < 1\}). \]
      By definition in the region where $|y| < 1$, the support of $M$
      lies between $H(\pm 2D)$, while on the
      boundary piece $\{r = r_0\} \times \{|y| < 1\}$ of $U_s$, $M$ lies
      between $H(\pm 2D_{r_0})$. It follows from
      this that:
      \begin{itemize}
      \item[(i)] Using property (2) above, in $U_s$, the support of
        $M$ lies on the negative side of
          $X_\epsilon$, as long as $\epsilon > D$.
        \item[(ii)] Using properties (1), (2), on the boundary of
          $U_s$ the support of $M$ lies on the negative side of
          $X_\epsilon$ as long as
          \[ \epsilon > sD, \text{ and } \epsilon > D_{r_0}. \]
          This holds in particular if $\epsilon > D_{r_0} + sD$. 
        \end{itemize}
        It follows
        that in $U_s$ the support of $M$ must be on the negative side of
        $X_\epsilon$ for all $\epsilon > D_{r_0} + sD$. Indeed, if
        this is not the case, then we let $\epsilon_0$ be the supremum of all
        $\epsilon$ for which the support of $M$ intersects $X_\epsilon$ (in the
        region $r < r_0$), and we find
        that $M$ lies on the negative side of $X_{\epsilon_0}$, but
        the two surfaces meet at an interior point of $U_s$ (property (ii)
        ensures that the intersection cannot be on $\partial
        U_s$). This is a contradiction since $X_{\epsilon_0}$ has
        negative mean curvature.

        $M$ therefore lies on the negative side of $X_{D_{r_0} +
          sD}$, and so by property (3) above, on the region $\{r <
        r_0\} \times \{|y| < 1/2\}$, $M$ lies on the negative side of
        $H(C_1D_{r_0} + C_1sD)$. Arguing in the same way from the negative
        side of $C\times\mathbf{R}$ as well, it follows that
        \[ D_{C\times\mathbf{R}}(M; \{|y| < 1/2\}) \leq C_1D_{r_0} +
          C_1sD,\]
        i.e.
        \[ D_{C\times\mathbf{R}}(M; \{|y| < 1/2\}) &\leq
          C_1D_{C\times\mathbf{R}}(M; \{r \geq s^{a_3 Q^2}\}\times \{|y|
          < 1\}) \\
          &\qquad + C_1s D_{C\times\mathbf{R}}(M; \{|y| < 1\}). \]
       Replacing $s$ by $C_1^{-1}s$, we can choose $A$ sufficiently
       large (depending on $Q, a_3$ and $C_1$) such that the required
       inequality \eqref{eq:nonconc1} holds. 
\end{proof}

\subsection{The quantitative three annulus lemma}\label{sec:q3ann}
This subsection contains the main new technical ingredient in the
proof of Theorem~\ref{thm:uniquecont1}. In order to state the result,
we introduce some notation. We fix a small number $q\in (0, 1/2)$, to be
chosen below, and we let
\[ \label{eq:ddefn} d(M, 1) = \left( \int_{M\cap B_1(0)} d^2 \right)^{1/2} +
  D_{C\times\mathbf{R}}(M; B_1(0))^{1 + q^2}, \]
where $d$ denotes the distance to $C\times\mathbf{R}$. For any $\rho
\leq 1$ we let $d(M, \rho) = d(\rho^{-1}M, 1)$. A basic property of
this is that if $a\in [\frac{1}{2}, 1]$ and $\rho \leq 1$, then
  \[ d(M, a\rho) \leq C_1 d(M, \rho), \]
where $C_1$ depends on the cone. This follows from the scaling
properties
\[ \int_{a^{-1}\rho^{-1}M \cap B_1} d^2 =
  a^{-2-n} \int_{\rho^{-1}M \cap B_a} d^2, \]
and
\[ \label{eq:DCRscale} D_{C\times\mathbf{R}}(a^{-1}\rho^{-1}M; B_1) = a^{-\gamma-1}
  D_{C\times \mathbf{R}}(\rho^{-1}M; B_a). \]
Our main technical result, used in place of the monotonicity of the
frequency function in the classical proofs of strong unique
continuation, is the following. 
\begin{prop}\label{prop:3anndecay}
  If $q>0$ in \eqref{eq:ddefn} is sufficiently small (depending on the cone $C$), and
  $Q=q^{-4}$, then for any $\lambda_0\in (0, 1/2)$ there exists a
  $\lambda\in [\lambda_0/2, \lambda_0]$ with the following
  property. Let $M$ be a stationary integral varifold in $B_2(0)$, with
  \[ \label{eq:Mareabound} \Vert
  M\Vert (B_2(0)) \leq \frac{11}{10} \Vert C\times\mathbf{R}\Vert
  (B_2(0)),\] such that 
  \[ \frac{1}{2}d(M,1) \leq d(M, e^{-\lambda}) \leq \lambda_0^Q. \]
  Then
  \[ \frac{1}{2}d(M, e^{-\lambda}) \leq d(M, e^{-2\lambda}). \]
\end{prop}
\begin{proof}
  Let us write $\epsilon = d(M, e^{-\lambda})$, and suppose that
  \[\label{eq:da1} d(M,1)\leq 2\epsilon, \text{ and } d(M, e^{-2\lambda})\leq
    \frac{\epsilon}{2}. \]
  Our goal is to show that if $q$ is sufficiently small and $\epsilon
  \leq \lambda_0^Q$, then for an appropriate choice of $\lambda\in
  [\lambda_0/2, \lambda_0]$ (depending only on the cone $C$), the
  inequalities \eqref{eq:da1} imply that we also have $d(M,
  e^{-\lambda}) < \epsilon$, contradicting the definition of
  $\epsilon$. It will then follow that $d(M,1)\leq 2 d(M,
  e^{-\lambda})$ implies $d(M, e^{-2\lambda}) > d(M,e^{-\lambda})/2$
  under the assumption that $d(M,e^{-\lambda})$ is sufficiently small,
  as required. 

  \bigskip
  \noindent{\em Step 1}. We will first show that if $q$ is small
  enough, then $M$ is graphical over
  $C\times \mathbf{R}$ on the region where $r > \epsilon^q$ and $\rho
  < e^{-\epsilon^q}$. To see this, note first that since $d(M,1)\leq
  2\epsilon$, by the definition of $d$ we have
  \[ \label{eq:DC2} D_{C\times\mathbf{R}}(M;B_1)\leq
    (2\epsilon)^{\frac{1}{1+q^2}}. \]
  This means that $M$ is contained between the surfaces
  $H(\pm (2\epsilon)^{\frac{1}{1+q^2}}) $. It follows
  that there is a
  constant $C_0 > 0$ depending on the cone $C$ only such that on the
  region where $r > C_0 \epsilon^{\frac{1}{\gamma+1}\frac{1}{1+q^2}}$,
  $M$ is contained between the graphs of
  \[ \label{eq:pmgraphs} \pm C_0 \epsilon^{\frac{1}{1+q^2}} r^{-\gamma} \]
  over $C\times \mathbf{R}$. Note that if $q, \epsilon$ are small,
  then $\epsilon^q >  C_0
  \epsilon^{\frac{1}{\gamma+1}\frac{1}{1+q^2}}$, and 
  so on our region $M$
  is contained between the graphs of the functions in
  \eqref{eq:pmgraphs}. Moreover if $r > \epsilon^q$ then
  \[ \label{eq:ge1} C_0 \epsilon^{\frac{1}{1+q^2}} r^{-\gamma} < C_0
    \epsilon^{\frac{1}{1+q^2}-\gamma q} < \epsilon^{1/2} \]
  if $\epsilon, q$ are small. 

  Let us now consider a
  point $p\in C\times\mathbf{R}$ in the region where $r > \epsilon^q$
  and $\rho < e^{-\epsilon^q}$. Then for any $c_0 < 1/2$ the ball $B(p,
  c_0\epsilon^q)$ is contained in $B(0,1)$, and the second fundamental
  form of
  $C\times\mathbf{R}$ on this ball is bounded by $C_1 \epsilon^{-q}$
  for some $C_1 > 0$. Using the bound \eqref{eq:ge1} we have that in
  this ball $M$ is contained in the $\epsilon^{1/2}$-neighborhood of
  $C\times\mathbf{R}$. Let us rescale: the surface
  $\tilde{M} = c_0^{-1}\epsilon^{-q}M$ will then be contained in the
  $c_0^{-1}\epsilon^{1/2-q}$-neighborhood of a surface $S$ in a unit
  ball $B$ whose
  second fundamental form satisfies a bound $|A|\leq C_1c_0$. Using
  the monotonicity formula, and the bound \eqref{eq:Mareabound} we
  have a uniform bound for the area of $\tilde{M}$ in $B$. We fix a
  small $\delta > 0$, and let 
  $c_0$ sufficiently small so that we can apply
  Lemma~\ref{lem:mult} below. We apply this to a cover of the region
  $\{r > \epsilon^q, \rho < e^{-\epsilon^q}\}$ in $C\times\mathbf{R}$
  with balls of radius $c_0\epsilon^q$. It follows using the Lemma
  that if in any of these balls the density of $M$ is greater than
  $(1+\delta)$, then in fact the density is greater than $(2-\delta)$
  for all of the balls. This contradicts \eqref{eq:Mareabound} once
  $\delta$ and $\epsilon$ are sufficiently small. It follows that the
  density of $M$ in any of the balls in our cover is less than
  $1+\delta$, which implies that $M$ is graphical using the Allard
  regularity theorem~\cite{All72}. 

  It follows that on the region $\{r > \epsilon^q, \rho <
  e^{-\epsilon^q}\}$, $M$ is given as the graph of a function $u$ that
  satisfies (see \eqref{eq:pmgraphs})
  \[\label{eq:ub1} |u| \leq C_0 \epsilon^{\frac{1}{1+q^2}}
    r^{-\gamma}. \]

  We can now get higher order estimates by rescaling and using
  interior derivative estimates for the minimal graph equation: given
  a point $p$ in the region $\{r > 2\epsilon^q, \rho <
  e^{-2\epsilon^q}\}$ with $r\in (R,2R)$, we rescale by
  $\epsilon^{-q}$. The scaled up surface $\tilde{M} = \epsilon^{-q}$,
  in the ball $B_1(\epsilon^{-q}p)$,
  is then the graph of a function $\tilde{u}$ over a surface with
  uniformly bounded geometry (a ball in $C\times\mathbf{R}$ in the
  region where $r > 1$). Using
  the bound \eqref{eq:ub1} we then have
  \[\label{eq:tu1} |\nabla^i \tilde{u}| \leq C_0\epsilon^{\frac{1}{1+q^2}-q}
    R^{-\gamma}, \text{ for } i=0,1,2,3, \]
  increasing $C_0$ if necessary. In addition since $M$ is minimal, we
  have
  \[ L_{C\times\mathbf{R}}(\tilde{u}) +
    Q_{C\times\mathbf{R}}(\tilde{u}) = 0, \]
  using the notation in \eqref{eq:mXQ}, and so
  \[ \label{eq:Ltu2} | L_{C\times\mathbf{R}}\tilde{u}|, |{\nabla} L_{C\times\mathbf{R}}\tilde{u}| \leq C_0
    \epsilon^{\frac{2}{1+q^2}-2q} R^{-2\gamma} \leq
    \epsilon^{\frac{3}{2}} R^{-2\gamma}, \]
  if $q,\epsilon$ are small. Let $\chi$ be a cutoff function such that
  $\chi(t)=0$ for $t < 1$ and $\chi=1$ for $t> 2$. The function
  $\chi(r)\tilde{u}$ satisfies the same derivative bounds as
  $\tilde{u}$ in \eqref{eq:tu1}, so after scaling back down, we have
  \[ |\nabla^i (\chi(\epsilon^{-q}r) u)| \leq C_0
    \epsilon^{\frac{1}{1+q^2}-2q-iq} R^{-\gamma}. \]
  In particular, if we fix a small $\beta > 0$, then
  on the region where $r\in (\epsilon^q, 2\epsilon^q)$
  we have
  \[ \label{eq:niL1} |\nabla^i L_{C\times\mathbf{R}}(\chi(\epsilon^{-q}r) u)| \leq C_0
    \epsilon^{\frac{1}{1+q^2} + \beta q} r^{-\gamma-\beta-2-i}. \]
  The same estimate also holds on the region $r \geq
  2\epsilon^q$, using \eqref{eq:Ltu2} and that for such $r$ we have
  $\chi(\epsilon^{-q}r)=1$. Indeed, \eqref{eq:Ltu2} implies after
  rescaling that for $r \geq 2\epsilon^q$ we have
  \[ |\nabla^i L_{C\times\mathbf{R}} \chi(\epsilon^{-q}r) u| \leq
    \epsilon^{\frac{3}{2}-q-iq} r^{-2\gamma} \leq
    \epsilon^{\frac{3}{2}-(i+1)q} r^{-\gamma+\beta+2+i} r^{-\gamma-\beta-2-i}. \]
  This implies the estimate \eqref{eq:niL1}, as long as $q$ is small
  enough. 

  \bigskip
  \noindent{\em Step 2}. Next we use the invertibility of
  $L_{C\times\mathbf{R}}$ in weighted spaces, given by
  Corollary~\ref{cor:CR2} to perturb $u$ to a Jacobi field $U$, and
  we apply the quantitative 3-annulus lemma,
  Proposition~\ref{prop:QL23annulus}. For
  this, we note that the estimate \eqref{eq:niL1} says that
  \[ \Big\Vert L_{C\times\mathbf{R}}(\chi(\epsilon^{-q}r)
    u)\Big\Vert_{C^{0,\alpha}_{-\gamma-\beta-2}} \leq C_0
    \epsilon^{\frac{1}{1+q^2} + \beta q}, \]
  and so we can solve the equation $Lv = L(\chi(\epsilon^{-q}r)u)$ on
  the ball $\{\rho < e^{-2\epsilon^q}\}$,  with
  \[ \label{eq:vweighted} \Vert v\Vert_{C^{2,\alpha}_{-\gamma-\beta}} \leq C_0
    \epsilon^{\frac{1}{1+q^2} + \beta q}, \]
  for a larger $C_0$.
  Note that this also implies the estimate
  \[ \label{eq:vL2} \Vert v\Vert_{L^2(B(0, e^{-2\epsilon^{q}}))} \leq \epsilon^{1 +
      \frac{\beta}{2} q}, \]
  choosing $q$ (and therefore $\epsilon$) smaller if necessary. 

  We now define $U = \chi(\epsilon^{-q}r)u - v$, so
  that by construction $U$ is a Jacobi field on $C\times\mathbf{R}$,
  defined in the ball $\{\rho < e^{-2\epsilon^q}\}$. The assumption
  \eqref{eq:da1} implies that
  \[ \Vert \chi(\epsilon^{-q}r) u\Vert_{L^2(B(0,1))} &\leq
    2\epsilon, \\
    \Vert \chi(\epsilon^{-q}r) u\Vert_{L^2(B(0,e^{-2\lambda}))} &\leq
    e^{-\lambda\left(n + 2\right)}\, \frac{\epsilon}{2}. \]
  Together with \eqref{eq:vL2} this implies
  \[ \Vert U\Vert_{L^2(B(0,e^{-2\epsilon^q}))} &\leq 2\epsilon +
    \epsilon^{1 + \frac{\beta}{2}q}, \\
      \Vert U\Vert_{L^2(B(0,e^{-2\lambda}))} &\leq e^{-\lambda\left(n
          + 2\right)}\, \frac{\epsilon}{2} + \epsilon^{1 +
        \frac{\beta}{2}q}. \]
    If we let $C_\epsilon = \epsilon + 2e^{n+2} \epsilon^{1 +
      \frac{\beta}{2} q}$, and $\tilde{U} = C_\epsilon^{-1}U$, then
    \[ \Vert \tilde{U}\Vert_{L^2(B(0, e^{-2\epsilon^q}))} &\leq 2, \\
        \Vert \tilde{U}\Vert_{L^2(B(0, e^{-2\lambda - 2\epsilon^q}))}
        &\leq \frac{1}{2} e^{-(n+2)\lambda}. \]
      This is the setting to which Proposition~\ref{prop:QL23annulus}
      applies (after a rescaling of the balls by $e^{2\epsilon^q}$),
      and so if $\lambda\in [\lambda_0/2, \lambda_0]$ is chosen
      appropriately, we have
      \[ \Vert \tilde{U}\Vert_{L^2(B(0, e^{-\lambda - 2\epsilon^q}))}
        \leq (1 - \lambda^A) e^{-\frac{n+2}{2}\lambda}. \]
      Using the convexity of $t \mapsto \Vert
      \tilde{U}\Vert_{L^2(B(0,e^t))}$, and that $\epsilon^{q^3} \geq
      2\epsilon^q$ for small $q$, we obtain
      \[ \Vert \tilde{U}\Vert_{L^2(B(e^{-\lambda + q^3}))} &\leq
        \left(1- \frac{2\epsilon^{q^3}}{\lambda}\right) (1-\lambda^A)
        e^{-\frac{n+2}{2}\lambda} + \frac{2\epsilon^{q^3}}{\lambda}
        \cdot 2 \\
        &\leq \left(1 - \lambda^A/2\right)
        e^{-\frac{n+2}{2}\lambda}, \]
      as long as $\frac{1}{2}\lambda^A e^{-(n+2)\lambda / 2} > 4
      \epsilon^{q^3}\lambda^{-1}$. Since $\lambda \geq \frac{1}{2}
      \epsilon^{1/Q}$, this will be satisfied for sufficiently small
      $\epsilon$ as long as $(A+1)Q^{-1} < q^3$. This holds for $Q
      = q^{-4}$ if $q$ is sufficiently small.

      Using the definition of $\tilde{U}$, we then have
      \[ \Vert U\Vert_{L^2(B(e^{-\lambda + \epsilon^{q^3}}))} &\leq
        \epsilon e^{-\frac{n+2}{2}\lambda} (1 - \lambda^A/2) (1 +
        2e^{n+2}\epsilon^{\beta q/2}) \\
        &\leq \epsilon e^{-\frac{n+2}{2}\lambda} (1 + \epsilon^{\beta
          q/4} - \lambda^A/2), 
      \]
      once $q$ is small enough. Using again that $\lambda \geq
      \epsilon^{q^{4}}/2$, we find that $\epsilon^{\beta q/4} <
      \lambda^A / 4$ if $q, \epsilon$ are small enough, and so
      we can assume that
    \[ \label{eq:UL24}\Vert U\Vert_{L^2(B(e^{-\lambda + \epsilon^{q^3}}))} &\leq
      \epsilon e^{-\frac{n+2}{2}\lambda} - \epsilon \lambda^{A}/4. \]
    We apply the $L^\infty$-estimate, Lemma~\ref{lem:L2Linfty},
    for Jacobi fields (on balls of
    radius comparable to $\epsilon^{q^3}/2$, scaled to unit size), to obtain
    \[ \label{eq:Usup1} |r^\gamma U| \leq C_3 \epsilon^{1 - C_3 q^3}, \text{ on }
      B(e^{-\lambda + \epsilon^{q^3}/2}). \]

    \bigskip
    \noindent{\em Step 3.} We now use our estimates for $U$ and $v$ to
    deduce corresponding estimates for $u$ on the region where $r >
    \epsilon^{2q^3}$ and $\rho < e^{-\lambda + \epsilon^{q^3}/2}$, and
    then apply the non-concentration result, 
    Proposition~\ref{prop:nonconc}, to find estimates for smaller $r$
    in a slightly smaller ball. Recall that on this region $u = U
    + v$. On the one hand, \eqref{eq:UL24} combined with \eqref{eq:vL2} implies
\[ \label{eq:uL22} \Vert \chi(\epsilon^{-q}r) u\Vert_{L^2(B(0,
    e^{-\lambda}))} \leq \epsilon e^{-\frac{n+2}{2}\lambda} - \epsilon
  \lambda^A / 8, \]
if $q$ is sufficiently small.

    At the same time by \eqref{eq:vweighted}, $v$ satisfies 
      \[ |r^\gamma v| \leq C_0 \epsilon^{\frac{1}{1+q^2} + \beta q}
        r^{-\beta} \leq \epsilon^{1 + \beta q / 2}, \]
      once $q$ is small enough (using $r > \epsilon^{2q^3}$). Combining
      this with \eqref{eq:Usup1} we have that on our region
      \[ |r^\gamma u| \leq \epsilon^{1 - C_3 q^3}, \]
      increasing $C_3$ if necessary, once $q$ is small enough. In
      particular we can use this to bound the distance from
      $C\times\mathbf{R}$ on this region:
      \[ \label{eq:DC1} D_{C\times\mathbf{R}}\left(M; \{ r > \epsilon^{2q^3}, \rho <
        e^{-\lambda + \epsilon^{q^3}/2}\}\right) \leq C_4 \epsilon^{1-C_3
        q^3}. \]
      We would like to apply Proposition~\ref{prop:nonconc} on
      suitable subsets of the ball $\{\rho < e^{-\lambda + \epsilon^{q^3}/2}\}$ with
      width $b = \frac{1}{4}\epsilon^{q^3}$, i.e. satisfying $|y-y_0|
      < b$. In order to use the estimate \eqref{eq:DC1} in
      Inequality~\eqref{eq:nonconc1},  we then need $\epsilon^{2q^3} < bs^A$,
      i.e. $s^A > 4\epsilon^{q^3}$. This is satisfied if we choose $s
      = \epsilon^{q^3/C_5}$ for a sufficiently large $C_5$, and then
      from   Proposition~\ref{prop:nonconc} we can deduce 
      \[ \label{eq:DC3} D_{C\times\mathbf{R}}(M; B_{e^{-\lambda}}) &\leq A
        D_{C\times\mathbf{R}} (M;
        B_{e^{-\lambda+\epsilon^{q^3}/2}}\cap \{ r \geq e^{2q^3}\}) \\
        &\qquad\qquad + \epsilon^{q^3/C_5} D_{C\times\mathbf{R}}(M;
        B_{e^{-\lambda+\epsilon^{q^3}/2}}) \\
      &\leq AC_4 \epsilon^{1-C_3q^3} + \epsilon^{q^3/C_5}
      (2\epsilon)^{\frac{1}{1+q^2}} \\
      &\leq C_6 (\epsilon^{1-C_3 q^3} + \epsilon^{q^3/C_5 + \frac{1}{1+q^2}}),\]
    where we also used \eqref{eq:DC2} again.
  
    \bigskip
    \noindent{\em Step 4}. We will now estimate $d(M, e^{-\lambda})$
    using \eqref{eq:uL22} and \eqref{eq:DC3}. First, \eqref{eq:DC3}
    implies (scaling up by $e^\lambda$ and increasing $C_6$), that
    \[ \label{eq:DC5} D_{C\times\mathbf{R}}(e^\lambda M; B_1)^{1+q^2} &\leq
      C_6(\epsilon^{1+q^2/2} + \epsilon^{1 + q^3/C_5}) \\
      &\leq \epsilon^{1+q^3/(2C_5)},\]
    and also
    \[ \label{eq:DC6} D_{C\times\mathbf{R}}(M; B_{e^{-\lambda}}) \leq
      \epsilon^{1-q^2} \]
  for sufficiently small $q$. 

  To estimate the $L^2$-distance of $M$ to $C\times\mathbf{R}$ on the
  ball $B(0,e^{-\lambda})$ note that $M$ is the graph of $u$ over
  $C\times\mathbf{R}$ on the region where $r > 2\epsilon^q$, and
  \eqref{eq:uL22} gives an $L^2$ bound for $u$ on this region. It
  remains to estimate the integral of $d^2$ over $M$ in the set $\{r\leq
  2\epsilon^q\} \cap B(0, e^{-\lambda})$. 
  
  The bound \eqref{eq:DC6} says that on the ball $B_{e^{-\lambda}}$
  the surface $M$ is contained between $H(\pm
  \epsilon^{1-q^2})\times\mathbf{R}$.
Let us write $\zeta =
  \epsilon^{\frac{1-q^2}{1+\gamma}}$, and denote by $\Omega$ the
  region between $H(\pm \zeta^{1+\gamma})\times\mathbf{R}$, satisfying also
  $\{|y| < 1, r < 2\epsilon^q\}$. We are interested
  in the part of $M$ that lies in $\Omega\cap B_{e^{-\lambda}}$. To
  estimate the integral of $d^2$ on this part of $M$, we
  cover $\Omega$ with a collection of
  balls as follows. First note that there is a constant $C_7 > 0$
  such that on the set $\{ r > C_7\zeta\}$, the surfaces $H(\pm \zeta^{1+\gamma})$
  are graphs of functions $u_{\pm}$ over $C$, with
  \[\label{eq:upm1} |u_{\pm}| < C_7\zeta^{1+\gamma}r^{-\gamma}.\]
  We can cover $\Omega\cap \{r <
  4C_7\zeta\}$ with $N_0$ balls of radius $\zeta$, where $N_0 \leq
  C_8\zeta^{-1}$. Using \eqref{eq:Mareabound} together with the
  monotonicity formula, this implies that the area of $M$ in this
  region is bounded by $C_9 N_0 \zeta^n$, while the distance $d$ to
  the cone is bounded by $4C_7\zeta$, so the contribution to the
  integral of $d^2$ is at most $C\zeta^{n+1}$. 

  To cover the rest of $\Omega$, note that we can
  cover a region of the form $\{R < r < 2R\} \cap \Omega$ using $N_R$
  balls of radius $\zeta^{1+\gamma} R^{-\gamma}$, where
  \[ N_R \leq C_{10} \frac{R^{n-1} \zeta^{1+\gamma}
      R^{-\gamma}}{\zeta^{(n+1)(1+\gamma)}R^{-(n+1)\gamma}} = C_{10}
    R^{n(1+\gamma)-1} \zeta^{-n(1+\gamma)}. \]
  For this note that the area of the region $\{R < r < 2R, |y| <
  1\}\cap (C\times\mathbf{R})$ is comparable to $R^{n-1}$, and by
  \eqref{eq:upm1} the region $\{R < r < 2R\}\cap \Omega$ is in the
  $C_7\zeta^{1+\gamma}R^{-\gamma}$-neighborhood of this. Using the
  monotonicity formula as above, we find that the area of $M$ in this
  region is of order $N_R \zeta^{n(1+\gamma)} R^{-n\gamma}$. As the
  distance $d$ to the cone is bounded by
  $C_7\zeta^{1+\gamma}R^{-\gamma}$, it follows that the integral of
  $d^2$ on this region of $\Omega$ is bounded by
  \[ C_{11} R^{n(1+\gamma)-1}\zeta^{-n(1+\gamma)} R^{-n\gamma}
    \zeta^{n(1+\gamma)}\zeta^{2+2\gamma} R^{-2\gamma} =  C_{11}
    R^{n-1-2\gamma} \zeta^{2+2\gamma}. \]

  We add up these contributions for $R= 2^iC_7\zeta$ with $i=1, 2,
  \ldots, K$, where $K$ is the smallest integer with $2^kC_7\zeta >
  2\epsilon^q$. Since $\gamma < \frac{n-1}{2}$, in the end we obtain
  the bound
  \[ \int_{M\cap\Omega\cap B_{e^{-\lambda}}} d^2 &\leq C_{12} (\epsilon^q)^{n-1-2\gamma}
    \zeta^{2+2\gamma}  = C_{12} \epsilon^{q(n-1-2\gamma) +2(1-q^2)} \\
    &\leq \epsilon^{2(1+q^2)}, \]
  for sufficiently small $q$.

  Combining this with \eqref{eq:uL22} we find that
  \[ \label{eq:DL2} \left(\int_{M\cap B(0,e^{-\lambda})} d^2\right)^{1/2} \leq
    \epsilon e^{-\frac{n+2}{2}\lambda} - \epsilon \lambda^A / 8 +
    \epsilon^{1+q^2}. \]
  Using that $\lambda > \epsilon^{q^4}/2$, for sufficiently small $q$
  we get
  \[  \left(\int_{M\cap B(0,e^{-\lambda})} d^2\right)^{1/2} \leq
    \epsilon e^{-\frac{n+2}{2}\lambda} - \epsilon \lambda^A / 16. \]
  After rescaling this, and combining it with \eqref{eq:DC5}, we get 
  \[ \left(\int_{e^\lambda M\cap B(0,1)} d^2 \right)^{1/2} +
    D_{C\times\mathbf{R}}(e^\lambda M; B_1)^{1+q^2} &\leq
    \epsilon - \epsilon e^{\frac{n+2}{2}\lambda} \lambda^A / 16 +
    \epsilon^{1+q^3/(2C_5)} \\
  < \epsilon,\]
  if $q$ is sufficiently small. This means $d(M,e^{-\lambda}) <
  \epsilon$, contradicting our choice of $\epsilon$ at the
  beginning of the proof. 
\end{proof}

We used the following lemma in the proof above. 
\begin{lemma}\label{lem:mult}
  Given any $\delta, V > 0$, there is a constant $c(n,\delta, V) > 0$ depending on the
  dimension $n$ and on $\delta, V$ with the
  following property.
  Let $S\subset\mathbf{R}^{n+1}$ be a smooth connected minimal hypersurface
  without boundary in the  ball $B_1(0)$, with $0\in S$ and
  with second fundamental form
  $|A|\leq c$. Suppose also that $M$ is an $n$-dimensional 
  stationary integral varifold in $B_1(0)$, supported in the $c$-neighborhood of
  $S$ with $\Vert M\Vert (B_1(0)) < V$. Then for all $p\in S\cap
  B_{1/2}(0)$ the following three inequalities are equivalent:
  \[ \label{eq:MB3} \Vert M \Vert (B_{1}(0)) &\geq (1+\delta)\omega_n, \\
  \Vert M\Vert (B_1(0)) &\geq (2-\delta)\omega_n, \\
  \Vert M\Vert (B_{1/2}(p)) &\geq (1+\delta) 2^{-n} \omega_n. \]
  where $\omega_n$ is the volume of the unit ball in $\mathbf{R}^n$. 
\end{lemma}
\begin{proof}
  We can argue by contradiction, assuming that for  given $V, \delta$
  no such $c$ exists, 
  i.e. there are sequences $S_i, M_i$ satisfying the hypotheses for $c
  = i^{-1}$, but we can find points $p_i$ for which only one of the
  inequalities in \eqref{eq:MB3} holds. Up to choosing a subsequence
  we can assume that the $S_i$ converge to a hyperplane $S$, and
  the $M_i$ converge as measures to a multiple $kS$, for $k\in
  \mathbf{N}$. The required claim is clear in the limit. 
\end{proof}

Proposition~\ref{prop:3anndecay} implies the following doubling condition, which
leads to the unique continuation result.
\begin{cor}\label{cor:doubling}
  Let $M$ be a stationary integral varifold in a neighborhood of the origin with unique tangent cone
  $C\times\mathbf{R}$ at the origin (with multiplicity one). Then
  either $M = C\times\mathbf{R}$ in a neighborhood of the origin, or
  there exists constants $A, \rho_0 > 0$ such that
  \[ \label{eq:ddouble} d(M, \rho) \leq A\, d(M, \rho/2), \]
  for all $\rho < \rho_0$. 
\end{cor}
\begin{proof}
  Let $\lambda_0 < 1/2$, and suppose that \eqref{eq:Mareabound} holds
  for $\rho^{-1}M$ for
  all $\rho \leq 1$, and in addition $d(M, \rho) \leq \lambda_0^Q$ for all
  $\rho \leq 1$ for the $Q$ in Proposition~\ref{prop:3anndecay}. This
  can be arranged by first scaling $M$ up. 

  Suppose that there is some $\rho_0 \leq 1$ such that for all $\lambda\in
  [\lambda_0/2, \lambda_0]$ we have $d(M, \rho_0) \leq 2d(M,
  e^{-\lambda}\rho_0)$. Then under our assumptions,
  Proposition~\ref{prop:3anndecay} implies that for all $k > 0$ we
  have $d(M, e^{-k\lambda}\rho_0) \leq 2 d(M,
  e^{-(k+1)\lambda})$. This implies that for a suitable constant $A$
  we have $d(M, \rho) \leq A d(M, \rho/2)$ for all $\rho < \rho_0$.

  We can therefore assume that for all $\rho \leq 1$ there exists some
  $\lambda\in [\lambda_0/2, \lambda_0]$ for which
  \[ \label{eq:dM5} d(M, \rho) > 2d(M, e^{-\lambda}\rho). \]
  If $\lambda_0$ is sufficiently small, then by the scaling behavior
  of $d(M, \rho)$, we always have
  \[ d(M, e^{-\lambda}\rho) \geq \frac{3}{4} d(M, e^{-\lambda_0}\rho), \]
  and so \eqref{eq:dM5} implies
  \[ d(M, \rho) > \frac{3}{2} d(M, e^{-\lambda_0}\rho). \]
  In particular if we let $M_1 = e^{\lambda_0}M$, and $\lambda_1 =
  (2/3)^{Q^{-1}}\lambda_0$, then our original assumptions apply to the
  pair $M_1$ and $\lambda_1$ replacing $M, \lambda_0$: we have
  \[ d(M_1, \rho) \leq \frac{2}{3} d(M, \rho) \leq \lambda_1^Q, \]
  while the area bound \eqref{eq:Mareabound} still holds.

  We can therefore iterate this construction, and we obtain a sequence
  of constants $\lambda_k = (2/3)^{Q^{-1}}\lambda_{k-1}$ and scalings
  $M_k = e^{\lambda_{k-1}}M_{k-1}$, satisfying $d(M_k, 1) \leq
  \lambda_k^Q$. Since $\lambda_k^Q\to 0$, and at the same time
  \[ \lambda_0 + \lambda_1 + \ldots =
    \frac{\lambda_0}{1-(2/3)^{Q^{-1}}} < \infty, \]
  it follows that $d(M, \rho_0) = 0$ for some $\rho_0 > 0$. This
  implies that $M=C\times\mathbf{R}$ in a neighborhood of the origin
  as required. 
\end{proof}

Based on this result we can prove Theorem~\ref{thm:uniquecont1}.
\begin{proof}[Proof of Theorem~\ref{thm:uniquecont1}]
  Suppose that $M$ is a stationary integral varifold in a neighborhood of the
  origin, satisfying the bounds \eqref{eq:L2vanish}. It is enough to
  show that \eqref{eq:ddouble} cannot hold. To see this, we show that
  for a stationary integral varifold $M$ in the ball $B(0,2)$ the $L^2$
  distance $\int_{M\cap B(0,2)} d^2$ to the cone $C\times\mathbf{R}$
  can be used to control $d(M, 1)$.

  Suppose that
  \[ \label{eq:intd3} \int_{M\cap B(0,2)} d^2 < \delta, \]
  and suppose that $p\in M\cap B(0,1)$ satisfies $d(p) > d_0$. Then
  for any $q\in  B(p, d_0/2)$ we have $d(q) > d_0/2$, and by the
  monotonicity formula we have
  \[ \Vert M\Vert B(p, d_0/2) \geq \omega_n (d_0/2)^n. \]
  It follows that
  \[ \int_{M\cap B(0,2)} d^2 > \omega_n (d_0/2)^{n+2}, \]
  and so using \eqref{eq:intd3} we have
  \[ d_0 < C_n \delta^{\frac{1}{n+2}} \]
  for a dimensional constant $C_n$. In particular the Hausdorff
  distance from $C\times\mathbf{R}$ to the support of $M\cap B(0,1)$
  is at most $C_n \delta^{1/(n+2)}$.

  At the same time the region between the surfaces
  $H(\pm \epsilon)$ contains the
  $c_1\epsilon$-neighborhood of $(C\times\mathbf{R})\cap B(0,2)$,
  for some $c_1$ depending on the cone $C$. Since $M\cap B(0,1)$ is
  in the $C_n \delta^{1/(n+2)}$-neighborhood of $C\times\mathbf{R}$, it
  follows that
  \[ D_{C\times \mathbf{R}}(M; B(0,1)) \leq C_n c_1^{-1}
    \delta^{1/(n+1)}. \]
  Therefore in general we have the estimate
  \[  D_{C\times\mathbf{R}}(M, B(0,1)) \leq C \left(\int_{M\cap
        B(0,2)} d^2 \right)^{1/(n+1)}.\]

  Suppose now that $M$ satisfies \eqref{eq:L2vanish}. For each $\rho >
  0$ this implies that
  \[ \int_{\rho^{-1}M\cap B(0,2)} d^2 < C_k(2\rho)^{k+n+2}, \]
  and so
  \[ D_{C\times\mathbf{R}}(\rho^{-1}M; B(0,1)) \leq C_k
    \rho^{\frac{k+n+2}{n+1}}, \]
  for new constants $C_k$. Using the definition of $d(M, \rho)$ it
  follows from this that up to further changing the constants $C_k$,
  we have
  \[ d(M, \rho) \leq C_k \rho^{\frac{k+n+2}{n+1}}. \]
  In particular for all $k, i$ we have
  \[ d(M, 2^{-i}) \leq C_k 2^{-i\frac{k+n+2}{n+1}} \]
  For sufficiently large $k$ this contradicts \eqref{eq:ddouble},
  since by iterating the estimate \eqref{eq:ddouble} we obtain
  $d(M, 2^{-i}) \geq A^{-i} d(M,1)$  for all $i$. 
\end{proof}

In the presence of an area bound \eqref{eq:Mareabound} the distance
$D_{C\times\mathbf{R}}(M; B(0,1))$ controls the $L^2$ distance from $M$ to $C\times\mathbf{R}
$ on $B(0,1/2)$ (see the argument above leading to \eqref{eq:DL2}),
and so we
have an analogous result in terms of $D_{C\times\mathbf{R}}$:
\begin{cor}
  Suppose that $M$ is a stationary integral varifold in $B(0,1)$, with unique
  tangent cone $C\times\mathbf{R}$ at the origin (with multiplicity
  one). If for any $k > 0$ we have a constant $C_k > 0$ satisfying
  \[ \label{eq:DCvanish} D_{C\times\mathbf{R}}(M, B(0,\rho)) \leq C_k \rho^k \]
  for all $\rho < 1$, then $M = C\times\mathbf{R}$ in a neighborhood
  of the origin.
\end{cor}

It follows from this result that for any such $M$ that is not equal to
$C\times\mathbf{R}$ near the origin, we can extract a non-zero Jacobi
field by a blowup argument, representing the leading order deviation
of $M$ from its tangent cone. 

\begin{cor}\label{cor:nonzeroU}
  Suppose that $M$ is a stationary integral
  varifold in a neighborhood of 0 with unique multiplicity one tangent cone
  $C\times\mathbf{R}$ at the origin, with $M$ not equal to $C\times\mathbf{R}$
  in any ball. For $i > 0$ let $a_i =
  D_{C\times\mathbf{R}}(2^iM; B_1)$. Then for a sequence $r_i \to 0$ we
  can write $2^iM$ as the graph of a function $u_i$ over $C\times\mathbf{R}$ on the region
  $B_1\cap \{r > r_i\}$. In addition up to choosing a subsequence we
  have $a_i^{-1}u_i \to U$ locally uniformly on
  $B_1\cap (C\times\mathbf{R})\setminus \{0\}\times\mathbf{R}$, where
  $U\not\equiv 0$ is a
  non-zero Jacobi field on $(C\times\mathbf{R}) \cap B_1$, satisfying
  $|U| \leq C_1 r^{-\gamma}$ for a constant $C_1$ depending on the
  cone $C$. 
\end{cor}
\begin{proof}
  The fact that $2^iM$ is the graph of $u_i$ on a set of the form $\{r
  > r_i\}$ follows in the same way as Step 1 in the proof of
  Proposition~\ref{prop:3anndecay}. In addition the estimate $|u_i|
  \leq a_i C_1 r^{-\gamma}$ on the set $\{r > r_i\}$ follows from the
  definition of the distance $D_{C\times\mathbf{R}}$. Choosing a subsequence we can extract a limit
  $a_i^{-1}u_i \to U$, to a Jacobi field $U$ on $B_1\cap 
  (C\times\mathbf{R})$. We only need to rule out the possibility that
  for all choices of subsequence we get the limit $U=0$.

  Suppose therefore that $a_i^{-1}u_i \to 0$ (along the entire
  sequence) locally uniformly on
  $B_1\cap (C\times\mathbf{R})$, away from the singular ray. We will
  show that then $D_{C\times\mathbf{R}}(2^iM; B_1)$ vanishes to
  infinite order in $2^{-i}$.

  Let $s < 1/2$. The non-concentration result
  Proposition~\ref{prop:nonconc} implies (for a larger constant $A$) that
  \[ \label{eq:nc1} D_{C\times\mathbf{R}}(2^iM; B(0,1/2)) &\leq A
    D_{C\times\mathbf{R}}(2^iM; \{r \geq s^A\}\cap B(0,3/4)) \\
    &\qquad + sD_{C\times\mathbf{R}}(2^iM; B(0,3/4)). \]
  Once $i$ is sufficiently large (depending on $s$), we have $r_i <
  s^A$, and so on the set $\{r \geq s^A\}\cap B(0,1)$, the surface $2^iM$ is the
  graph of $u_i$ over $C\times\mathbf{R}$.
  Since $a_i^{-1}u_i \to 0$ locally uniformly,
  we can choose $i$ sufficiently large so that $|u_i| < s a_i$
  on the set $\{r \geq s^A\}\cap B(0,3/4)$, and so
  \[D_{C\times\mathbf{R}}(2^iM; \{r \geq s^A\}\cap B(0,3/4)) \leq
    C_1s a_i \]
  for $C_1$ depending on the cone $C$. Using \eqref{eq:nc1}, together
  with the definition of the $a_i$, this implies that for sufficiently
  large $i$ (depending on $s$) we have
  \[ a_{i+1} \leq C_2s a_i, \]
  for a constant $C_2$ depending only on the cone $C$. This in turn
  implies that for any $s > 0$ there is a constant $C_s$ depending on
  $s$ such that $a_i\leq C_s s^i$ for all $i > 0$, i.e. 
  \[ D_{C\times\mathbf{R}}(2^iM; B(0,1)) \leq C_s s^i. \]
  Letting $\rho = 2^{-i}$ and $k = -\log_2 s$, we can write this as
  \[ D_{C\times\mathbf{R}}(\rho^{-1}M; B(0,1)) \leq C_k \rho^k. \]
  This implies inequalities of the form \eqref{eq:DCvanish} (note that
  $k\to \infty$ as $s\to 0$), so in a
  neighborhood of the origin $M=C\times\mathbf{R}$, contrary to our
  assumption. 
\end{proof}

\section{Minimal hypersurfaces with large
  symmetry}\label{sec:symmetry}
In this section we restrict ourselves to minimizing cones $C = C(S^p\times S^q)$
for $p+q > 6$, and to codimension one stationary integral varifolds
$M$ in a neighborhood of the origin $0 \in
\mathbf{R}^n\times\mathbf{R}$  that are
invariant under the action of the group $O(p+1)\times O(q+1)$ on
$\mathbf{R}^n = \mathbf{R}^{p+1}\times\mathbf{R}^{q+1}$. The main
result is the following.

\begin{thm}\label{thm:symmetric}
  Let $M$ be a stationary integral varifold in a neighborhood of the origin in
  $\mathbf{R}^n\times\mathbf{R}$, with tangent cone
  $C\times\mathbf{R}$ at the origin (with multiplicity one), where $C
  = C(S^p\times S^q)$, with $p+q > 6$. Suppose that $M$ is invariant under the action
  of $G = O(p+1)\times O(q+1)$ on $\mathbf{R}^n$. Then either $M =
  C\times\mathbf{R}$ in a neighborhood of the origin, or $M$
  is a graph over one of the surfaces constructed in
  Section~\ref{sec:gluing} near the origin and so it has an
  isolated singularity at the origin. 
\end{thm}

The proof of Theorem~\ref{thm:symmetric} will occupy the rest of this
section. First note that under the assumptions we know that $C\times\mathbf{R}$ is
actually the unique tangent cone of $M$ at the origin. This follows
from Simon~\cite{Simon94}, although under the $O(p+1)\times
O(q+1)$-invariance it is more immediate, since $C\times \mathbf{R}$
has no
invariant degree one Jacobi fields when $p+q>6$. Assuming that $M$ is
not equal to $C\times\mathbf{R}$ in a neighborhood of the origin, we can then apply
Corollary~\ref{cor:nonzeroU} to extract non-trivial Jacobi fields  by
blowing up $M$. We define the degree of $M$ as follows.

\begin{definition}\label{defn:Mdegree}
Given any Jacobi field $U$ on the unit ball in $C\times\mathbf{R}$
satisfying $|r^\gamma U| \in L^\infty$, we let the degree of $U$ be
the smallest degree appearing in the homogeneous decomposition of
$U$. I.e. the degree of $U$ is $d$, if
\[ U = U_d + O(\rho^{c+\gamma}r^{-\gamma})\]
for some $c > d$ as $\rho\to 0$, and $U_d$ is a non-zero degree $d$
homogeneous Jacobi field. If $U$ is identically zero we define its
degree to be infinite. 

Let $M$ be as in the statement of Theorem~\ref{thm:symmetric}, and
suppose that $M$ is not equal to $C\times\mathbf{R}$ in a neighborhood
of the origin. We then define the degree of $M$ to be the
smallest possible degree of all the Jacobi fields that are 
obtained by the rescaling process in Corollary~\ref{cor:nonzeroU},
\end{definition}

Note that since we are considering $M$ that are $O(p+1)\times
O(q+1)$-invariant, by Lemma~\ref{lem:CinvJ} the Jacobi fields $U$
arising from Corollary~\ref{cor:nonzeroU} are of the form
\[ U = \sum_{k, \l\geq 0} a_{k,\l} r^{2k-\gamma} y^\l. \]
This can be written as
\[ U = \lambda u_\l + O(\rho^{c+\gamma}r^{-\gamma}) \]
for some $\lambda\not=0$ and $c > \l-\gamma$, where $u_\l$ is
the function in
\eqref{eq:ul}. Note that in the present setting $\phi_1=1$, so in
particular the formula for $u_\l$ makes sense on all of
$\mathbf{R}^{n+1}\setminus \{0\}\times\mathbf{R}$. 
We will assume from now on that the degree of $M$ is
$\l-\gamma$.

Let us denote by $T_{\pm 1}$ the minimal surfaces constructed in
Proposition~\ref{prop:Texist}, modeled on the Jacobi fields
$\pm u_\l$, and define 
\[ \label{eq:Tlambda}T_\lambda = \lambda^{(1-(\l-\gamma))^{-1}}T_1, \quad T_{-\lambda} =
  \lambda^{(1-(\l-\gamma))^{-1}}T_{-1}\, \text{ for }\lambda > 0.\]
We let $T_0 = C\times\mathbf{R}$. 
For sufficiently small $|\lambda|$, the surface $T_\lambda$ is defined in
$B_2(0)$ and to leading order we can think of  $T_\lambda$
as the graph of $\lambda u_\l$ over $C\times\mathbf{R}$, at
least away from the singular ray. 

To prove Theorem~\ref{thm:symmetric} the strategy is to show that under the
assumptions in the theorem $M$ will decay towards $T_\lambda$ for a
suitable $\lambda\not=0$ at a
rate faster than the degree $\l-\gamma$. In a sufficiently small
neighborhood of the origin this will imply that $M$ is actually a
graph over $T_\lambda$, and in particular it has an isolated
singularity at the origin. The proof has similarities  with the
argument 
in \cite{Sz20} proving the uniqueness of certain cylindrical tangent cones,
based on a suitable three annulus lemma in terms of an $L^\infty$-type
distance. The use of the three annulus
lemma itself for proving decay estimates goes back to
Allard-Almgren's~\cite{AA81} and Simon's~\cite{Simon83} works.

\subsection{Perturbations of $T_\lambda$}
For technical reasons we will need to consider further minimal perturbations
of the surfaces $T_\lambda$, because we do not have sufficiently
precise results comparing $T_\lambda$ to $T_{\lambda'}$ when $\lambda,
\lambda'$ are very close to each other. In order to do this we first construct
the corresponding Jacobi fields. For simplicity we work with $T=T_1$,
but the same results hold for $T_{-1}$ as well. 
\begin{prop}\label{prop:TJacfield}
  There is a Jacobi field $\phi_\l$ on $T$ in a sufficiently small
  neighborhood of 0, satisfying the following properties:
  \begin{enumerate}
  \item $\phi_\l \in C^{2,\alpha}_{\l-\gamma,-\gamma}$.
  \item On the region $r > |y|^b$ for some $b > 1$ we have
    $\phi_\l - u_\l \in C^{2,\alpha}_{\delta, \tau}$ where the
    weights $\delta, \tau$ are chosen as in
  Proposition~\ref{prop:Texist} and Remark~\ref{rem:deltatau}. 
  \end{enumerate}
\end{prop}

\begin{proof}
  We begin by constructing a function $U_\l$ on $X$, which we will
  then perturb to the Jacobi field $\phi_\l$ on $T$. This can be
  thought of as a linear version of the construction of $X$ itself in
  Section~\ref{sec:approxsoln1}. We define $U_\l$ as follows, using
  the notation from Section~\ref{sec:approxsoln1} where we defined
  $X$. 
  \begin{itemize}
  \item On the region where $r \geq |y|^\beta$ we let $U_\l = u_\l$. 
  \item On the region where $r \leq \frac{1}{2}|y|^\beta$ recall that
    $X$ is defined to be $H(y^\l)$ in the slice
    $\mathbf{R}^n\times\{y\}$. For simplicity assume that $y > 0$ and
    let $H= H_+$, so that $H(y^\l) = y^a H$ (recall that $a = \l /
    (\gamma+1)$). Let us also denote by $\Phi$ the Jacobi field on $H$
    satisfying $\Phi = r^{-\gamma} + O(r^{-\gamma-c})$.  For a point
    $(y^a x, y) \in X$ in our region, where $x\in H$, we set
    \[ U_\l (y^a x, y) = y^a \Phi(x). \]
    If we extend $\Phi$ to scalings $\lambda H$ by $\Phi(\lambda x) =
    \Phi(x)$, as in Lemma~\ref{lem:Phipullback} below,
    then we can also write this as $U_\l(x,y) =
    y^a\Phi(x)$. 
  \item On the intermediate region $\frac{1}{2} |y|^\beta < r <
    |y|^\beta$ we interpolate between the definitions above:
    \[ U_\l(y^a x, y) = u_\l(y^a x, y) + \chi\left( \frac{ 2|y^a
          x|}{|y|^\beta}\right) \Big[ y^a\Phi(x) - u_\l(y^a x,
      y)\Big] \]
    in terms of the same cutoff function $\chi$ as before. 
  \end{itemize}

  From the construction it follows that $U_\l \in
  C^{2,\alpha}_{\l-\gamma, -\gamma}$. In particular $L_X U_\l \in
  C^{0,\alpha}_{\l-\gamma-2, -\gamma-2}$, however we actually have the
  following slightly better estimate.
  
  \bigskip
  \noindent {\bf Claim}: The function $U_\l$ satisfies $L_X U_\l \in
   C^{0,\alpha}_{\delta-2, -\gamma-2}$, where $\delta > \l - \gamma$
   is sufficiently close to $\l-\gamma$ as in
   Proposition~\ref{prop:mXest}.
   \bigskip

   To prove the claim we study different regions as in the proof of
   Proposition~\ref{prop:mXest}.
   \begin{itemize}
     \item {\bf Region 1}, where $r > |y|^\beta$. This corresponds to
       Regions I,II,III in Proposition~\ref{prop:mXest}. Suppose that $r\in
       (R,2R)$ and introduce the rescaled surface $\tilde{X} =
       R^{-1}X$ as before. Using the calculations in the proof of
       Proposition~\ref{prop:mXest} we know that in this region
       $\tilde{X}$ can be viewed as the graph of a function $f$ over
       $C\times\mathbf{R}$, where
       \[ |\nabla^i f| \leq C_i \rho^\l R^{-\gamma-1}. \]
       In terms of the rescaled variables $\tilde{x}, \tilde{y}$ we have
       $u_\l(x,y) = R^{\l-\gamma} u_\l(\tilde{x}, \tilde{y})$ since
       $u_\l$ has degree $\l-\gamma$. In addition $u_\l$ is a Jacobi
       field on $C\times\mathbf{R}$. It follows that
       \[ |\nabla^i L_{\tilde{X}} u_\l(\tilde{x}, \tilde{y})| \leq C_i
         \rho^\l R^{-\gamma-1}, \]
       and by rescaling we find
       \[ |\nabla^i L_X u_\l(x,y)| \leq C_i R^{-2-i} R^{\l-\gamma}
         \rho^{\delta+\gamma} R^{-\gamma-1}. \]
       Using that $\l-\gamma > 1$ this implies
       \[ |\nabla^i L_X u_\l(x,y)| \leq C_i\rho^{\delta+\gamma}
         R^{-\gamma-2-i}, \]
       for $\delta > \l-\gamma$ sufficiently close to $\l-\gamma$. 
       This implies the claim.
       \item {\bf Region 2}, where $r < |y|^\beta/2$. Here we use the
         calculation from Region IV in the proof of
         Proposition~\ref{prop:mXest}. We use the same notation as in
         that proof, working on the rescaled surface $\tilde{X} =
         R^{-1}X$ in terms of the rescaled coordinates $\tilde{x},
         \tilde{y}$. The surface $\tilde{X}$ is given by
         $E(\tilde{y})H$ in the slice
         $\mathbf{R}^n\times\{\tilde{y}\}$, where $E(\tilde{y}) =
         E(0)(1+G(Ry_0^{-1}\tilde{y}))$. The function $U_\l$ is given
         by
         \[\label{eq:Ul10} U_\l( E(\tilde{y}) z, \tilde{y}) = y^a \Phi(z) =
           (R\tilde{y} + y_0)^a \Phi(z), \]
         where $z\in H$.  To estimate the effect of the Jacobi
         operator of $\tilde{X}$ on this function we view $\tilde{X}$
         as a graph over $E(0)H\times\mathbf{R}$ and pull the function
         back to this product using Lemma~\ref{lem:Phipullback}
         below. Let us define
         \[ \tilde{F}_{\tilde{y}} : E(0)H \to
           E(\tilde{y}) H = E(0) (1 + G(Ry_0^{-1}\tilde{y})) H\]
         to be the map obtained by viewing $E(\tilde{y}) H$ as a graph
         over $E(0)H$. In terms of the maps $F_\epsilon$ in
         Lemma~\ref{lem:Phipullback} below, we have
         \[ \tilde{F}_{\tilde{y}} ( E(0) z) = E(0) F_{ 1+
             G(Ry_0^{-1}\tilde{y}) }(z), \]
         for $z\in H$. Because of this we can write
         \[ \tilde{F}_{\tilde{y}}^*U_\l( E(0)z, \tilde{y}) =
           (R\tilde{y}+y_0)^a (F^*_{1+G(Ry_0^{-1}\tilde{y})}
           \Phi)(z), \]
         where $z\in H$. It is more convenient to write this in terms
         of $\tilde{z} = E(0)z \in E(0)H$:
         \[ \label{eq:FUl11} \tilde{F}_{\tilde{y}}^*U_\l( \tilde{z}, \tilde{y}) =
           (R\tilde{y}+y_0)^a (F^*_{1+G(Ry_0^{-1}\tilde{y})}
           \Phi)(E(0)^{-1} \tilde{z}), \]
         and use that in our scaled up region $E(0)H$ has uniformly bounded
         geometry. In particular we have the estimates
         \[ |\nabla^i \Phi( E(0)^{-1} \tilde{z}) | \lesssim
           E(0)^\gamma, \]
         where we are taking derivatives on $E(0) H$.
         
         Using these formulas we find that at $y=y_0$ we have
         \[ L_{E(0)H \times \mathbf{R}} \tilde{F}_{\tilde{y}}^* U_\l
           \lesssim (R y_0^{-1})^2 y_0^a E(0)^\gamma. \]
         For this note that when we differentiate with respect to
         $\tilde{y}$ we obtain a factor of $Ry_0^{-1}$ (for
         derivatives falling on $F_{1+G(Ry_0^{-1}\tilde{y})}$ we use
         Lemma~\ref{lem:Phipullback} below), while along
         the slice $y=y_0$ the function $\tilde{F}_{\tilde{y}}^*
         U_\l$ is a Jacobi field on $E(0) H$.

         From the analysis in
         Region IV in Proposition~\ref{prop:mXest} we know that in our region
         $\tilde{X}$ is the graph of a function $A$ over
         $E(0)H\times\mathbf{R}$, where $|\nabla^i A| \lesssim
         Ry_0^{-1} E(0)^{\gamma+1}$. In addition we have
         \[ |\nabla^i \tilde{F}_{\tilde{y}}^* U_\l| \lesssim y_0^a
           E(0)^\gamma, \]
         noting that derivatives in the $\tilde{y}$-direction will
         actually satisfy better estimates. It follows that
         \[ |\nabla^i (L_{\tilde{X}} - L_{E(0)H \times \mathbf{R}})
           \tilde{F}_{\tilde{y}}^* U_\l | \lesssim Ry_0^{-1} E(0)^{\gamma+1}y_0^a E(0)^\gamma, \]
         where we are viewing $L_{\tilde{X}}$ as an operator on
         $E(0)H\times\mathbf{R}$ through the map
         $\tilde{F}_{\tilde{y}}$.  In sum we have 
         \[ |\nabla^i L_{\tilde{X}} \tilde{F}_{\tilde{y}}^* U_\l
           |\lesssim Ry_0^{-1} y_0^a E(0)^\gamma ( Ry_0^{-1} +
           E(0)^{\gamma+1}). \]
         Rescaling we then have
         \[ |\nabla^i L_X U_\l| \lesssim R^{-2-i} Ry_0^{-1} y_0^a
           E(0)^\gamma (Ry_0^{-1} + E(0)^{\gamma+1}). \]
         Since in this region $y_0\sim \rho$, and also $E(0) = R^{-1}
         y_0^a$, for our claim we need to
         show that
         \[ R^{-\gamma-1} \rho^{\l-1} ( R\rho^{-1} +
           R^{-\gamma-1}\rho^\l) \lesssim \rho^{\delta + \gamma}
           R^{-\gamma-2}. \]
         Using that $\rho^a \lesssim R \lesssim \rho^\beta$ on this
         region, a short calculation shows that this indeed holds if
         $\delta > \l - \gamma$ is sufficiently close to $\l-\gamma$.
       \item {\bf Region 3}, where $|y|^\beta / 2 \leq r \leq
         |y|^\beta$. We again consider $\tilde{X} = R^{-1}X$, where
         $r\in (R, 2R)$. In the notation of the proof of
         Proposition~\ref{prop:mXest} this is still Region IV, and we
         view $\tilde{X}$ as a graph over $E(0)H\times
         \mathbf{R}$. We need to estimate the difference
         \[ B(\tilde{x},\tilde{y}) = y^a \Phi(\tilde{x} ) - u_\l(x, y), \]
         where $x = R\tilde{x}$ and $y = R\tilde{y} + y_0$ as before,
         and we are using the convention in
         Lemma~\ref{lem:Phipullback} that we extend define $\Phi$ on
         scalings $\lambda H$ so that $\Phi(\lambda x) = \Phi (x)$. 
         We pull back all the quantities to $E(0)H\times\mathbf{R}$ as
         in Region 2 above, using $\tilde{F}_{\tilde{y}}$.

         We first use \eqref{eq:FUl11} and Lemma~\ref{lem:Phipullback}
         to see that on $E(0)H \times\mathbf{R}$ we have
         \[  |\tilde{F}_{\tilde{y}}^*(y^a \Phi(x)) - y^a \Phi(x)
           |_{C^k} \lesssim Ry_0^{-1} y_0^a E(0)^\gamma. \]
         Since on $H$ we have $\Phi(z) = |z|^{-\gamma} + O(|z|^{-\gamma-c})$
         for some $c > 0$ we can compare $\Phi(\tilde{x})$ to
         $|\tilde{x}|^{-\gamma}$ on $E(0)H\times \mathbf{R}$:
         \[ | y^a\Phi(\tilde{x}) - y^a E(0)^\gamma
           |\tilde{x}|^{-\gamma}|_{C^k} \lesssim y_0^a E(0)^{\gamma+c}. \]
         At the same time, on $\tilde{X}$, from the definition of $u_\l$ we have
         \[  u_\l(x, y)  &= y^\l |x|^{-\gamma} + O(\rho^{\l-2}
           R^{2-\gamma}) \\
           &= y^\l R^{-\gamma} |\tilde{x}|^{-\gamma}
           + O(\rho^{\l-2} R^{2-\gamma}),\]
         where the error estimate holds for derivatives on $\tilde{X}$
         as well. We also have, as in Region 2 above, that $\tilde{X}$
         is the graph of the function $A$ over $E(0)H
         \times\mathbf{R}$ where $|\nabla^i A|\lesssim  Ry_0^{-1}
         E(0)^{\gamma+1}$. This leads to an error of order $Ry_0^{-1}
         E(0)^{\gamma+1} y_0^\l R^{-\gamma}$ when comparing $u_\l$ on
         $\tilde{X}$ and on $E(0)H\times\mathbf{R}$. 

         Combining all of this information, and using $E(0) =
         R^{-1}y_0^a$, we obtain the estimate
         \[ |B|_{C^k} \lesssim Ry_0^{a-1} E(0)^\gamma + y_0^a
           E(0)^{\gamma+c} + Ry_0^{-1} E(0)^{\gamma+1} y_0^\l
           R^{-\gamma}. \]
         Scaling back to $X$ introduces an additional factor of
         $R^{-2}$, and so to verify our claim we need to ensure that
         \[ (R^{1-\gamma} \rho^{\l-1} + R^{-\gamma-c}\rho^{\l + ac} +
           R^{-2\gamma}\rho^{2\l-1}) R^{-2} \lesssim
           \rho^{\delta+\gamma} R^{-\gamma-2}, \]
         where we also used that in our region $\rho\sim y_0$. A
         calculation shows that this inequality holds as long as
         $\delta$ is sufficiently close to $\l-\gamma$, since $1 <
         \beta < a$ and $R\sim \rho^\beta$. 
       \end{itemize}

       Having shown the claim we can finish the proof of the
       Proposition. By Proposition~\ref{prop:Texist}, in a
       sufficiently small neighborhood of the origin, the minimal
       hypersurface $T$ is the graph of a function $f$ over $X$ where
       $f\in C^{2,\alpha}_{\delta, \tau}$. 
       Using
       Lemma~\ref{lem:normcomp1} we have $\Vert
       f\Vert_{C^{2,\alpha}_{1,1}} \lesssim K^{-\kappa}$ on the region
       $\rho < K^{-1}$. It then follows from the surjectivity of
       $L_X$ in Proposition~\ref{prop:Linvert} that if $K$ is
       sufficiently large, then
       \[ \label{eq:LTs10} L_T : C^{2,\alpha}_{\delta,\tau}(T\cap \rho^{-1}(0, K^{-1}])
         \to C^{0,\alpha}_{\delta-2,\tau-2}(T\cap
         \rho^{-1}(0,K^{-1}]) \]
       is also surjective.
       
       We will identify $T$ with $X$ using that $T$ is graphical over
       $X$.  We can then view
       the function $U_\l$ constructed on $X$ above as a function on
       $T$. In order to estimate $L_T U_\l$ we use that $T$ is the
       graph of $f$ over $X$ with $f\in
       C^{2,\alpha}_{\delta,\tau}$. For given $R, S$ with $S < K^{-1}$
       sufficiently small, consider the region $\Omega_{R, S} \subset
       X$, where $r \in(R, 2R)$ and $\rho\in (S, 2S)$. Let  
       $\tilde{X} = R^{-1}X$, $\tilde{T} = R^{-1}T$ and
       $\tilde{\Omega}_{R,S} = R^{-1} \Omega_{R,S}$. By definition
       of the weighted norms $\tilde{T}$ is the graph of a function
       $\tilde{f}$ over $\tilde{X}$ on the region
       $\tilde{\Omega}_{R,S}$, where $|\tilde{f}|_{C^{2,\alpha}}
       \lesssim  S^{\delta-\tau} R^{\tau-1}$. Note that $\delta > \l -
       \gamma$, we take $\tau$ sufficiently close to $-\gamma$,
       and in addition $R \gtrsim S^a$ on $X$. It follows that
       $S^{\delta-\tau} R^{\tau-1} \lesssim S^{\kappa}$ for some
       $\kappa > 0$. In particular if $S$ is sufficiently small then
       on $\tilde{\Omega}_{R,S}$ we have
       \[ \Vert L_{\tilde{X}} - L_{\tilde{T}}\Vert_{C^{2,\alpha} \to
           C^{0,\alpha}} 
         \lesssim S^{\delta-\tau} R^{\tau-1}. \]
       Since $U_\l \in C^{2,\alpha}_{\l-\gamma, -\gamma}$, on
       $\tilde{\Omega}_{S,R}$ we have $\Vert U_\l
       \Vert_{C^{2,\alpha}} \lesssim S^{\l} R^{-\gamma}$. In addition
       our Claim above implies that on $\tilde{\Omega}_{R,S}$ we have
       \[  \Vert L_{\tilde X} U_\l\Vert_{C^{0,\alpha}} \lesssim S^{\delta+\gamma}
         R^{-\gamma}.   \]
       Combining these estimates we have
       \[   \Vert L_{\tilde{T}} U_\l  \Vert_{C^{0,\alpha}} \lesssim
         S^{\delta+\gamma} R^{-\gamma} + S^{\delta-\tau} R^{\tau-1}
         S^\l R^{-\gamma}.  \]
       Using that $R\lesssim S$ and $S^a \lesssim R$ on $X$, as well
       as $\tau < -\gamma$, we have
       \[ \Vert L_{\tilde{T}} U_\l \Vert_{C^{0,\alpha}} \lesssim
         S^{\delta - \tau} R^\tau. \]
       Rescaled, and allowing $S, R$ to vary this implies that $L_T
       U_\l \in C^{0,\alpha}_{\delta, \tau-2}$ on the region $\rho <
       K^{-1}$ for sufficiently large $K$. The surjectivity of the
       operator in \eqref{eq:LTs10} then implies that in a
       neighborhood of the origin we can find a function $\phi_\l$ on
       $T$ with $L_T\phi_\l = 0$, satisfying $\phi_\l - U_\l  \in
       C^{2,\alpha}_{\delta, \tau}$. The properties (1) and (2) in the
       statement of the proposition now follow from the properties of
       $U_\l$, setting $b=\beta$.  Note in particular that $C^{2,\alpha}_{\delta, \tau}
       \subset C^{2,\alpha}_{\l-\gamma, -\gamma}$ for our choice of
       $\delta, \tau$. 
\end{proof}

We used the following result in the proof above.
\begin{lemma}\label{lem:Phipullback}
  Let us denote by $\Phi$ the Jacobi field on $H$ which satisfies
  $\Phi = r^{-\gamma} + O(r^{-\gamma-c})$ for a small $c > 0$.
  For sufficiently small $\epsilon$
  view $(1+\epsilon)H$ as a graph over $H$, and denote by 
  $F_\epsilon: H \to (1 +\epsilon) H$ the corresponding
  identification. Extend $\Phi$ from  $H$ to multiples $\lambda H$
  with $\lambda > 0$ to be homogeneous of degree zero (i.e. $\Phi(\lambda x)
  =\Phi(x)$).  The pullbacks $F_\epsilon^*\Phi$ on $H$ then satisfy the
  estimates
  \[ | \partial_\epsilon^j \nabla^i F_\epsilon^* \Phi| \leq C_{i,j}
    r^{-\gamma-i}, \]
  for suitable constants $C_{i,j}$. 
\end{lemma}
\begin{proof}
  To see this, note that $F_\epsilon^*\Phi$ satisfies the equation
  $L_\epsilon (F_\epsilon^*\Phi) = 0$, where
  \[ L_\epsilon = \Delta_{g_\epsilon} + |A_\epsilon|^2, \]
  in terms of the metric $g_\epsilon$ and norm of second fundamental
  form $|A_\epsilon|$ pulled back from $(1+\epsilon)H$.

  Using Proposition~\ref{prop:Hestimates} we find that the operators
  \[ L_\epsilon : C^{2,\alpha}_\tau \to C^{0,\alpha}_{\tau-2} \]
  vary smoothly with $\epsilon$, for sufficiently small $\epsilon$,
  and in particular as discussed after Proposition~\ref{prop:LHR10},
  these operators are all invertible for $\tau\in (3-n+\gamma,
  -\gamma)$. Note that using Proposition~\ref{prop:Hestimates} again
  we have $L_\epsilon \Phi \in C^{0,\alpha}_{-\gamma-2-c}$ for some small $c >
  0$ on $H$, and so by the invertibility of $L_\epsilon$
  we can find a family of functions $u_\epsilon$
  varying smoothly in $C^{2,\alpha}_{-\gamma-c}$ satisfying
  \[ L_\epsilon (\Phi + u_\epsilon) = 0. \]
  Identifying $H$ with $(1+\epsilon)H$ using $F_\epsilon$, we then
  have that
  \[ (F_\epsilon^{-1})^*(\Phi+u_\epsilon) = r^{-\gamma} +
    O(r^{-\gamma-c})\]
  is a Jacobi field on $(1+\epsilon)H$. This Jacobi field with leading
  order term $r^{-\gamma}$ is unique, and so it must be a multiple of
  $\Phi$. Our assumption that $\Phi$ is extended from $H$ to
  $(1+\epsilon)H$ by $\Phi((1+\epsilon)x) = \Phi(x)$ for $x\in H$ then implies that
  \[  (F_\epsilon^{-1})^*(\Phi+u_\epsilon) = (1+\epsilon)^{-\gamma}
    \Phi, \]
  comparing leading terms. Pulling back again to $H$ we then have
  \[ F_{\epsilon}^* \Phi = (1+\epsilon)^\gamma(\Phi + u_\epsilon), \]
  from which the required result follows. 
\end{proof}

Using this Jacobi field we can construct minimal perturbations of $T$
as follows.
\begin{prop}\label{prop:Tdeform}
  There are $\epsilon_0, C_1> 0$ such that for any
  $\epsilon$ with $|\epsilon| < \epsilon_0$ there exists a minimal
  surface $T_{1,\epsilon}$ in a neighborhood of the origin (the
  neighborhood is independent of $\epsilon$) with the following
  property. $T_{1,\epsilon}$ is the graph of the function $\epsilon
  \phi_\l + v_\epsilon$ over $T$, where $v_0=0$ and
  \[ \label{eq:vepsbound} \Vert v_\epsilon - v_{\epsilon'} \Vert_{C^{2,\alpha}_{\delta,\tau}} \leq
    C_1 (|\epsilon| + |\epsilon'|) |\epsilon - \epsilon'|, \]
  for $|\epsilon|, |\epsilon'| < \epsilon_0$. The weights
  $\delta,\tau$ here are as in Proposition~\ref{prop:Texist} and
  Remark~\ref{rem:deltatau}. 
\end{prop}
\begin{proof}
  This is an application of the implicit function theorem.
  We need to solve the equation
  \[ m_T(\epsilon \phi_\l + v_\epsilon) = 0. \]
  for $v_\epsilon \in C^{2,\alpha}_{\delta, \tau}$, where $m_T(f)$
  denotes the mean curvature of the graph of $f$ over $T$ as
  before. We seek $v_\epsilon$ of the form $v_\epsilon = L_T^{-1}
  h_\epsilon$, where $L_T^{-1}$ is a right inverse for the map in
  \eqref{eq:LTs10}. We apply the implicit function theorem to the
  map
  \[ \mathcal{F} : (-\epsilon_0, \epsilon_0) \times
    C^{0,\alpha}_{\delta-2,\tau-2}(T\cap \rho^{-1}(0, K^{-1}]) &\to
    C^{0,\alpha}_{\delta-2, \tau-2}(T\cap \rho^{-1}(0, K^{-1}])  \\
  (\epsilon, h) &\mapsto m_T(\epsilon\phi_\l + L_T^{-1}h) \]
  for sufficiently small $\epsilon_0 > 0$ and large $K$. 
  The properties of the weighted spaces, in particular the fact that 
  $C^{2,\alpha}_{\delta,\tau} \subset C^{2,\alpha}_{\l - \gamma, -\gamma}\subset
  C^{2,\alpha}_{1,1}$, imply that near the origin the map
  $\mathcal{F}$ is smooth. In addition the
  linearization of $h\mapsto \mathcal{F}(0,h)$ is $L_T L_T^{-1} =
  I$. The implicit function theorem then implies that we can find a
  smooth family $h_\epsilon \in C^{0,\alpha}_{\delta-2,\tau-2}$
  satisfying $\mathcal{F}(\epsilon, h_\epsilon) = 0$. Differentiating
  with respect to $\epsilon$ at $\epsilon =0$ we have
   \[ L_T (\phi_\l + L_T^{-1}  \dot{h}_0) = 0, \]
   and since $L_T \phi_\l = 0$ we get $\dot{h}_0 = 0$. It follows that 
\[ \Vert h_\epsilon - h_{\epsilon'}
  \Vert_{C^{0,\alpha}_{\delta-2,\tau-2}} \lesssim (|\epsilon| +
  |\epsilon'|) |\epsilon - \epsilon'|.\]
  This implies the required claim about $v_\epsilon =
  L_T^{-1}h_\epsilon$. 
\end{proof}

We will write 
\[ T_{\lambda, \epsilon \lambda} = \lambda^{(1-(\l-\gamma))^{-1}}
    T_{1,\epsilon}\]
for $\lambda > 0$ and $|\epsilon| < \epsilon_0$. For $\lambda < 0$ we define
\[ T_{\lambda, \epsilon \lambda} = |\lambda|^{(1-(\l-\gamma))^{-1}}
  T_{-1,-\epsilon}. \]
Roughly speaking we can think of $T_{\lambda, \epsilon\lambda}$ as the graph
of $(1+\epsilon)\lambda u_\l$ over $C\times \mathbf{R}$ to leading
order. We will need the following. 

\begin{lemma}\label{lem:fineq11}
  For sufficiently small $\lambda, \epsilon$ the hypersurface
  $T_{\lambda, \epsilon\lambda}$ is the graph of a function $f$ over
  $H(\lambda y^\l)$ in the slice $(\mathbf{R}^n\times \{y\})\cap B_2(0)$
  satisfying
  \[ \label{eq:fineq11} |f| \leq C_1 |\lambda| (\rho^{\l -\kappa}
    r^{\kappa - \gamma} + |\epsilon| \rho^\l r^{-\gamma}), \]
for some $C_1, \kappa > 0$ independent of $\lambda, \epsilon$. 
\end{lemma} 
\begin{proof}
  From Proposition~\ref{prop:Tgraphbound}, by rescaling, we find that
  $T_\lambda$ is the graph of a function $f_1$ over $H(\lambda y^\l)$
  in the $\mathbf{R}^n\times\{y\}$ slice, where
\[ |f_1| \leq C_1 |\lambda| \rho^{\l -\kappa} r^{\kappa - \gamma}. \]
  In addition by rescaling Proposition~\ref{prop:Tdeform}, the hypersurface
  $T_{\lambda, \epsilon\lambda}$ is the graph of a function $f_2$ over
  $T_{\lambda}$, where
  \[ |f_2| \leq C_1 |\epsilon| |\lambda| \rho^\l r^{-\gamma}, \]
using that $C^{2,\alpha}_{\delta, \tau} \subset
C^{2,\alpha}_{\l-\gamma, -\gamma}$. The required result follows by
combining the estimates for $f_1$ and $f_2$. 
\end{proof}

In the next section we will define the distance $D_{T_{\lambda,
    \epsilon\lambda} }(M; U)$
of $M$ from $T_{\lambda, \epsilon\lambda}$ over a set $U$, prove a non-concentration
result, and a three-annulus lemma. This follows \cite{Sz20}
fairly closely. In Section~\ref{sec:thm3proof} we will use these ingredients
to prove Theorem~\ref{thm:symmetric}. 

\subsection{Non-concentration and three-annulus lemma}
We will initially only define the distance from $T_{\lambda, \epsilon\lambda}$ to $M$ on
subsets of the annulus $B_1\setminus B_{1/2}$. On more general
subsets we will define the distance by rescaling. In addition, as in
\cite{Sz20} the distance initially depends on an additional small
parameter $\beta$, which we will choose in
Proposition~\ref{prop:nonconc2} below. We also need to choose a
number $\gamma_1$ so that $\gamma < \gamma_1 < \frac{n-3}{2}$. 

      \begin{definition}\label{defn:Ddefn}
Let $\beta > 0$ be  a small parameter to be chosen below. For small $d > 0$
we define the $d$-neighborhood $N_d(T_{\lambda, \epsilon\lambda})$ of $T_{\lambda, \epsilon\lambda}$ in the annulus
$B_1\setminus B_{1/2}$ as follows:
\begin{itemize}
  \item[(a)] If $d \geq \beta |\lambda|$, then $N_d(T_{\lambda,
      \epsilon\lambda})$ is the region in 
    $B_1\setminus B_{1/2}$ that is bounded
    between $H(\pm \beta^{-2} d) \times\mathbf{R}$.
  \item[(b)] If $d < \beta|\lambda|$, then $N_d(T_{\lambda,
      \epsilon\lambda})$ is the region in 
    $B_1\setminus B_{1/2}$ that is bounded
    between the graphs of
    \[\label{eq:U2defn} \pm \min\{ (\beta|\lambda| + d)r^{-\gamma}, d
      r^{-\gamma_1}\} \]
    over $T_{\lambda, \epsilon\lambda}$, on $B_2\setminus B_{1/4}$.
  \end{itemize}
  Given this, we define the distance $D_{T_{\lambda, \epsilon\lambda}}(M; U)$ of $M$ from
$T_{\lambda, \epsilon\lambda}$ on the set $U\subset B_1\setminus B_{1/2}$ to be the
infimum of those $d > 0$ for which $M\cap U \subset N_d(T_{\lambda, \epsilon\lambda})$. 
\end{definition}

\begin{remark}\label{rem:DT}
Note that if the subset $U$ is far from the singular set,
i.e. $U\subset \{ r > r_0\}$ for some $r_0 > 0$, then
$D_{T_{\lambda, \epsilon\lambda}}(M; U) < d$ implies that $M\cap U$ is contained in the (usual)
$C(r_0) d$-neighborhood of $T_{\lambda, \epsilon\lambda}$, for some $C(r_0)$ depending on
$r_0, \beta$, and vice versa. As $r_0\to 0$, we have $C(r_0)\to\infty$. 
\end{remark}

We have the following, analogous to \cite[Lemma 5.3]{Sz20}.
\begin{lemma}
  Suppose that $\beta > 0$ is sufficiently small. Then for $0< d_1 <
  d_2$ we have $N_{d_1}(T_{\lambda, \epsilon\lambda}) \subset
  N_{d_2}(T_{\lambda, \epsilon\lambda})$ as 
  long as $|\lambda|, |\epsilon|$ are sufficiently small (depending on $\beta$). In
  addition $\cap_{d > 0} N_d(T_{\lambda, \epsilon\lambda}) =
  T_{\lambda, \epsilon\lambda} \cap (B_1\setminus 
  B_{1/2})$. 
\end{lemma}
\begin{proof}
  The proof is quite similar to the proof of  \cite[Lemma 5.3]{Sz20},
  except for the additional $\epsilon$ here. It is enough to show that
  on the annulus $B_2\setminus B_{1/4}$
  the region between the graphs of $\pm \min\{ 2\beta |\lambda|
  r^{-\gamma},  \beta|\lambda| r^{-\gamma_1}\}$ over $T_{\lambda,
    \epsilon\lambda}$ is contained between $H(\pm \beta^{-2}
  \beta|\lambda|) \times\mathbf{R}$. 
  
  By Lemma~\ref{lem:L51}, working in the $\mathbf{R}^n\times \{y\}$
  slice, the region between $H(\pm
  \beta^{-1}|\lambda|)$ contains the region between the graphs of 
  \[ \pm c_0\min\{ (\beta^{-1}|\lambda| -| \lambda y^\l|) r^{-\gamma},
    r\} \]
  over $H(\lambda y^\l)$. 

  Using Lemma~\ref{lem:fineq11}, 
  we find that the region between the graphs of
  \[\pm \min\{ 2\beta |\lambda|
  r^{-\gamma},  \beta|\lambda| r^{-\gamma_1}\}\] over $T_{\lambda,
    \epsilon\lambda}$ is contained between the graphs of
\[ \pm (2\beta |\lambda| r^{-\gamma} + C_1|\lambda| r^{\kappa-\gamma}
  + C_1 |\epsilon \lambda| r^{-\gamma}) \]
  over $H(\lambda y^\l)$. 

  We therefore have to show that 
  \[ \label{eq:2b10} 2\beta |\lambda| r^{-\gamma} + C_1|\lambda| r^{\kappa-\gamma}
  + C_1 |\epsilon \lambda| r^{-\gamma} \leq c_0 \min\{ (\beta^{-1}|\lambda| - |\lambda y^\l|) r^{-\gamma},
    r\} \]
  if $\beta, \lambda, \epsilon$ are sufficiently small. 
  We can verify
  this by looking at two separate cases. In the region that we are
  considering, if $|y|\leq 1/8$, then we have $r \in( 1/8,2)$, and 
  \eqref{eq:2b10} follows for small $\beta, \lambda$. 
  If $|y|\in (1/8, 2)$, then we also use that
  $r \gtrsim |\lambda y^\l|^{1/(\gamma+1)}$ on
  $H(\lambda y^\l)$. Using this the inequality \eqref{eq:2b10} follows
  for small $\beta, \lambda, \epsilon$. 
\end{proof}

Note that $T_0 = C\times\mathbf{R}$, and in
Definition~\ref{defn:DCRdist} we already defined a distance
$D_{C\times\mathbf{R}}(M; U)$. We will distinguish the two notions of
distance by writing $D_{T_0}$ and $D_{C\times\mathbf{R}}$ respectively. Note
however that for
$U\subset (B_1\setminus B_{1/2})$  the two distances are uniformly equivalent:
\[ \label{eq:DCRDTequiv}
  C_1^{-1} D_{C\times\mathbf{R}}(M; U) \leq D_{T_0}(M ; U) \leq C_1
  D_{C\times\mathbf{R}}(M; U). \]
Here the constant $C_1$ depends on the cone $C$ and the constant
$\beta$ (which will be fixed later). The proof follows easily from the
definitions, noting that case (b) in Definition~\ref{defn:Ddefn} does
not arise when $\lambda=0$. 

We will use the following result to obtain graphical estimates,
given bounds for the distance $D_{T_\lambda}$.
\begin{lemma} \label{lem:DTgraph}
  There exists a $C_1 > 0$ depending on the cone $C$ and on
  $\beta$, such that for any $r_1 > 0$ there is a $\delta > 0$
  (depending on $r_1, C, \beta$) with
  the following property. Suppose that $M$ is a stationary integral
  varifold in $B_2(0)$ with the area bound \eqref{eq:Mareabound},
  satisfying $D_{T_{\lambda, \epsilon\lambda}}(M; B_1\setminus B_{1/2}) < \delta$ for
  some $|\lambda|, |\epsilon| < \delta$. Then on the region $(B_1\setminus
  B_{1/2})\cap \{ r > r_1\}$ we can write $M$ as the graph of a
  function $u$ over $T_{\lambda, \epsilon\lambda}$ satisfying
  \[ |u| \leq C_1 D_{T_{{\lambda, \epsilon\lambda}}}(M; B_1\setminus B_{1/2})
    r^{-\gamma_1}. \]
  Note that in later applications $\beta$ will be fixed and recall
  that $\gamma_1 > \gamma$ is close to $\gamma$, used in
  Definition~\ref{defn:Ddefn}. 
\end{lemma}
\begin{proof}
  Let us write $d = D_{T_{\lambda, \epsilon\lambda}}(M; B_1\setminus B_{1/2})$. If $d
  =0$, then the support of $M$ equals $T_{\lambda, \epsilon\lambda}$
  on the annulus $B_1\setminus 
  B_{1/2}$ by the previous lemma. The area bound for $M$ ensures that
  the multiplicity of $M$ is 1 on the annulus, and so the result
  follows.

  We therefore assume that $d > 0$, and so by definition $M
  \cap (B_1\setminus B_{1/2}) \subset N_{2d}(T_{\lambda,
    \epsilon\lambda})$. There are two cases to
  consider. If $2d < \beta |\lambda|$, then $M$ is contained between
  the graphs of \eqref{eq:U2defn} and so in particular $M$ is
  contained between the graphs of $\pm 2d r^{-\gamma_1}$ over
  $T_{\lambda, \epsilon\lambda}$. 

  If $2d \geq \beta |\lambda|$, then 
  $M$ is contained between the
  hypersurfaces $H(\pm \beta^{-2} 2d) \times\mathbf{R}$. By the
  discussion in Definition~\ref{defn:Ht}, to leading order the surfaces $H(\pm
  \beta^{-2}2d)$ are the graphs of $\pm \beta^{-2}2d r^{-\gamma}$ over
  the cone $C$ on
  the set where $r \gg |\beta^{-2} 2d|^{1/(\gamma+1)}$. In particular
  for any given $r_1 > 0$ this is the case on the set $\{r > r_1\}$
  once $d$ is sufficiently small (depending on $\beta$ and the cone
  $C$). At the same time using Lemma~\ref{lem:fineq11} we also know
  that for sufficiently small $\lambda$ the surface $T_{\lambda, \epsilon\lambda}$ is the
  graph of $f$ over $H(\lambda y^\ell)$ with $f$ satisfying
  \eqref{eq:fineq11}. As long as $|\lambda|\leq \beta^{-2}2d$ (which
  is implied by our assumption $2d\geq \beta|\lambda|$ for small $\beta$), it then
  follows that on the region where $r > r_1$, $M$ is contained between
  the graphs of $\pm C_2 d r^{-\gamma}$ over $T_{\lambda, \epsilon\lambda}$, where $C_2$
  depends on $\beta$.  

  To see that $M$ is actually graphical over $T_{\lambda, \epsilon\lambda}$ on the region
  where $r > r_1$ when $d$ is sufficiently small, we can argue in the
  same way as in Step 1 of the proof of
  Proposition~\ref{prop:3anndecay}. 
\end{proof}

Similarly to \cite[Proposition 5.6]{Sz20}, we obtain the following
non-concentration estimate for minimal hypersurfaces $M$ near to
$T_{\lambda, \epsilon\lambda}$ using barrier arguments. As long as $M$ is not too close
to $T_{\lambda, \epsilon\lambda}$, we use the barrier surfaces constructed in
Proposition~\ref{prop:barrier10}, while when $M$ is very close to
$T_{\lambda, \epsilon\lambda}$ we use graphical barriers constructed using
Lemma~\ref{lem:Fa2}.

\begin{prop}\label{prop:nonconc2}
  Suppose that $\beta > 0$ is sufficiently small. There is a constant
  $C_\beta$ (depending on $\beta$), such that given any $s >  0$,
  there are $r_0=r_0(\beta, s) > 0$ and $\delta = \delta(\beta, s,
  r_0) > 0$ with
  the following property. Suppose that $M$ is a stationary integral
  varifold in $B_1\setminus B_{1/2}$ such that for some $|\lambda| , |\epsilon|<
  \delta$ we have
  \begin{itemize} 
  \item  $D_{T_{\lambda, \epsilon\lambda}}(M; B_1\setminus B_{1/2}) < \delta$,
  \item  $M$ is contained
  between the graphs of $\pm D(r_0) r^{-\gamma}$ over $T_{\lambda,
    \epsilon\lambda} $ on the region where
  $r \geq r_0$ in the annulus $B_1\setminus B_{1/2}$, for some
  $D(r_0)\leq \delta$.
\end{itemize}
  Then we have the estimate
  \[ \label{eq:nc20} D_{T_{\lambda, \epsilon\lambda}}(M; B_{4/5}\setminus B_{3/5}) \leq C_\beta \Big(D(r_0) +
    s D_{T_{\lambda, \epsilon\lambda}}(M; B_1\setminus B_{1/2})\Big). \]
  Note that $\beta$ will be fixed after this point, and crucially $C_\beta$ is
  independent of $s$.
\end{prop}
Note that by rescaling, the same estimate holds for pairs of annuli
other than $(B_{4/5}\setminus B_{3/5}) \subset (B_1\setminus
B_{1/2})$ with an appropriate change in the constants. 
\begin{proof}
  The proof is similar to that in \cite{Sz20}, except for the
  presence of $\epsilon$. 
  For the reader's convenience we provide the argument here. 

  Set $d = D_{T_{\lambda, \epsilon\lambda}}(M; B_1\setminus B_{1/2})$, 
  fix $s > 0$, and let
  \[  \overline{d} &= \max\{ \beta |\lambda|, d + s^{-1}D(r_0)\}. \]
  Our first goal will be to show that on a slightly smaller annulus we
  have
  \[ \label{eq:goal1} D_{T_{\lambda, \epsilon\lambda}}(M;
    B_{19/20}\setminus B_{11/20}) \leq C_\beta 
    s\overline{d}, \]
  if $r_0, \delta$ are sufficiently small (i.e. $D(r_0), d,
  |\lambda|, |\epsilon|$
  are sufficiently small as well). Throughout the proof we will denote
  by $C_\beta$ constants that may change from line to line, and depend
  on $\beta$. Constants $C_1, C_2, \ldots$ will depend only on the
  cone $C$. 

  Note first that if $\beta, \delta$ are sufficiently small, then $D(r_0)$
  controls the distance of $M$ from $T_{\lambda, \epsilon\lambda}$ on the region $\{r >
  r_0\}$:
  \[ D_{T_{\lambda, \epsilon\lambda}}(M ; (B_1\setminus B_{1/2})\cap \{ r > r_0\}) \leq
    D(r_0). \]
  To see this we need to show that $M\cap \{ r > r_0\} \subset
  N_{D(r_0)}(T_{\lambda, \epsilon\lambda})$. According to 
  Definition~\ref{defn:Ddefn} we look at two cases separately:
  \begin{itemize}
    \item If $D(r_0)
  \geq\beta|\lambda|$, then we need to verify that $M$ is contained
  between the surfaces $H(\pm \beta^{-2}D(r_0))\times
  \mathbf{R}$ on the region $\{r > r_0\}$, using that $M$ is contained
  between the graphs of $\pm D(r_0) r^{-\gamma}$ over
  $T_{\lambda, \epsilon\lambda}$. On the region $\{r > r_0\}$, once
  $D(r_0), \lambda, \epsilon$ are
  sufficiently small, it follows that $M$ is contained between the
  graphs of $\pm C_2(D(r_0) + |\lambda|) r^{-\gamma}$ over
  $C\times\mathbf{R}$, while $H(\pm\beta^{-2}D(r_0))$ are on the
  positive (resp. negative)
  side of the graphs of $\pm C_2^{-1} \beta^{-2} D(r_0) r^{-\gamma}$ over
  $C\times\mathbf{R}$. It remains to check that
  \[ D(r_0) + |\lambda| \leq C_2^{-2} \beta^{-2} D(r_0). \]
  This follows if $\beta$ is sufficiently small, using that $D(r_0)
  \geq \beta |\lambda|$.
  \item If $D(r_0) < \beta|\lambda|$, then we need to check that, on
    the region $\{r > r_0\}$, $M$
    is contained between the graphs of
    \[ \pm \min \{ (\beta|\lambda| + D(r_0)) r^{-\gamma},
      D(r_0) r^{-\gamma_1}\}, \]
    over $T_{\lambda, \epsilon\lambda}$
    using that it is contained between the graphs of $\pm D(r_0)
    r^{-\gamma}$ over $T_{\lambda, \epsilon\lambda}$. Since $r < 1$ this is clear.
  \end{itemize}

  From now on we will only be concerned with estimates on the region
  $\{r < r_0\}$ for sufficiently small $r_0$. In addition 
  instead of annuli it will be convenient to work on rectangular
  regions of the form  $\{r < r_0\} \cap  \{|y| \in (a,b)\}$. Once
  $r_0, s$ are sufficiently small, the region $(B_{19/20}\setminus B_{11/20})
  \cap \{ r < r_0\}$ is contained in the region 
\[ \Omega_1 = \{r < r_0\} \cap \{|y| \in (0.51+s,0.99-s). \]
 To control $M$ on this region, we will
  apply the barrier construction from Proposition~\ref{prop:barrier10}
  to the function
  \[ f(y) = \sigma ( t^{-1}\lambda y^\ell + h(y)), \]
  where $h(y) = (0.99 - |y|)^{-1} + (|y| - 0.51)^{-1}$, and $\sigma$
  is a smooth function
  satisfying $|\sigma(z)^p - z| < \frac{1}{10}$ for all $z$,  as in
  the proof of Proposition~\ref{prop:Tminimizing}. We will 
  use the barrier construction on the intervals defined by $|y|\in
  (0.51 + s, 0.99-s)$ and we will only consider $t$
  satisfying $t > s |\lambda|$. Under these conditions
  $|f|_{C^3}$ has a bound depending on $s$, and so from 
  Proposition~\ref{prop:barrier10} we obtain barrier surfaces
  $X_t$ on the region $\{r < r_0\}$ for $s|\lambda| < t <
  t_0$, where $r_0, t_0$ depend on $s$. Note that
  $r_0$ may still be chosen smaller below. The
  $y$-slices of the  barrier surfaces $X_t$ lie between
  $H(t f^p \pm t)$, and we have
  \[ t \left(t^{-1}\lambda y^\ell + h(y) - \frac{1}{10}\right) \leq
    t f(y)^p \leq t\left( t^{-1}\lambda y^\ell + h(y)
    + \frac{1}{10}\right). \]
  Using that $h(y) > 4$ it follows that the $y$-slices of $X_t$ lie between
  \[ H(\lambda y^\ell + t h(y)/2 ) \text{ and }
    H(\lambda y^\ell + 2 t h(y)). \]

  We claim that if  $\lambda, d, D(r_0)$ are sufficiently small, then $M$
  lies on the negative side of $X_{t_0}$ on our region
  $\Omega_1$. Since $h(y) \geq 4$, we have that $X_{t_0}$ lies
  on the positive side of the surface with slices $H(\lambda y^\ell +
  2t_0)$. By definition we have $\overline{d} > d$ and
  $\overline{d} > \beta|\lambda|$. It follows that on our region $M$
  is contained between $H(\pm \beta^{-2} \overline{d})$. Our claim
  follows since $\beta^{-2} \overline{d} < \lambda y^\ell +
  2t_0$ if $\lambda, d, D(r_0)$ are sufficiently small
  (depending on $s, \beta$).

  Let $t_1 = \beta^{-3} s\overline{d}$, which in particular satisfies
  $t_1 > s|\lambda|$ is $\beta < 1$. In addition if
  $\overline{d}$ is sufficiently small (depending on $\beta$), then
  $t_1 < t_0$. We claim that on the boundary of the
  region $\Omega_1$ the surface $M$ lies on the negative side of
  $X_t$ for all $t \in [t_1, t_0]$. For
  this we examine the two kinds of boundary pieces of $\Omega_1$:
  \begin{itemize}
  \item On the boundary pieces where $|y|=0.51+s$ or $|y| = 0.99-s$ we
    have $h(y) > s^{-1}$, and so the slices of $X_t$ lie on the positive
    side of $H(\lambda y^\ell + t_1 s^{-1}/2)$. Since $M$ lies
    on the negative side of $H(\beta^{-2} \overline{d})$ it suffices
    if
    \[ \lambda y^\ell + t_1 s^{-1}/2 \geq
      \beta^{-2}\overline{d}. \]
    We have $\overline{d} \geq \beta|\lambda|$ and $t_1 =
    \beta^{-3}s\overline{d}$, so it suffices if
    \[ \beta^{-3} \overline{d} / 2 - \beta^{-1}\overline{d} \geq
      \beta^{-2}\overline{d}, \]
    and this holds for small $\beta$. 
  \item On the boundary piece $\{r = r_0\}$ we know that $M$ is on the
    negative side of the graph of $D(r_0) r^{-\gamma}$ over
    $T_{\lambda, \epsilon\lambda}$, and so using the graphical
    property, Lemma~\ref{lem:fineq11} of $T_{\lambda, \epsilon\lambda}$
    over $H(\lambda y^\ell)$,  $M$ is on the negative side of the graph of
    \[ 2(C_1 |\lambda| r_0^{\kappa - \gamma} + C_1 |\lambda\epsilon|
      r_0^{-\gamma} + D(r_0) r_0^{-\gamma}) \]
    over $H(\lambda y^\ell)$ if $\lambda, D(r_0)$ are sufficiently
    small (depending on $r_0$). At the same time $X_{t}$ is on
    the positive side of $H(\lambda y^\ell + 2t_1)$, i.e. on
    the positive side of the graph of
    \[ c_0 \min \{ 2t_1 r^{-\gamma}, r\} \]
    over $H(\lambda y^\ell )$ using Lemma~\ref{lem:L51}. We therefore need to check that
    \[ 2(C_1 |\lambda| r_0^{\kappa -\gamma} + C_1 |\lambda\epsilon|
      r_0^{-\gamma} + D(r_0) r_0^{-\gamma})
      \leq c_0 \min\{ 2t_1 r_0^{-\gamma}, r_0\}. \]
    Note that if $\overline{d}$ is sufficiently small (depending on
    $\beta, r_0$), then on the right hand side the minimum is achieved by
    $2t_1 r_0^{-\gamma}$. The inequality then follows if
    $\beta, r_0$ and $\epsilon$ are sufficiently small ($r_0, \epsilon$ can depend on $s,
    \beta$),
    since by definition
    $t_1 \geq \beta^{-2} s|\lambda|$, and also $t_1 \geq
    \beta^{-3} D(r_0)$. 
  \end{itemize}

  As $t$ varies from $t_0$ to $t_1$, the surface
  $X_t$ cannot have any interior contact point with $M$ in
  $\Omega_1$ by Proposition~\ref{prop:barrier10}. It follows then that
  $M$ lies on the negative side of $X_{t_1}$. In particular
  since on the interval $|y|\in [0.52, 0.98]$ we have an upper bound
  $h(y) < C_3/2$, it follows that for these $y$ the $y$-slices of the
  surface $M$ lie on the negative side of $H(\lambda y^\ell +
  C_3\beta^{-3}s\overline{d})$. We can repeat the same argument from
  the negative side of $M$ by reversing orientations, and we find that
  on the smaller region
  \[ \Omega_2 = \{|y|\in [0.52, 0.98]\} \cap \{r < r_0\} \]
  $M$ lies between the surfaces with $y$-slices $H(\lambda y^\ell \pm
  A)$, where $A = C_3\beta^{-3}s\overline{d}$.

  Next we claim that $D_{T_{\lambda, \epsilon\lambda}}(M ; \Omega_2) \leq C_4 A$, for
  a sufficiently large $C_4$ depending on the cone $C$. To see this we look at two
  cases:
  \begin{itemize}
  \item If $C_4 A \geq \beta|\lambda|$, then we need to check that $M$
    lies between $H(\pm \beta^{-2} C_4A)$. This follows, since 
    \[  |\lambda| + A < \beta^{-2} C_4 A, \]
    as long as $\beta < 1$ and $C_4 > 1$.
  \item If $C_4 A <  \beta|\lambda|$, then we need to show that over
    $\Omega_2$ the surface $M$ is contained between the graphs of
    \[ \pm \min\{ (\beta|\lambda| + C_4A) r^{-\gamma} , C_4 A
      r^{-\gamma_1}\} \]
    over $T_{\lambda, \epsilon\lambda}$. Since $M$ is between $H(\lambda y^\ell \pm A)$
    we know that $M$ is contained between the graphs of
    \[ \pm C_1' (|\lambda| r^{\kappa - \gamma} + |\lambda\epsilon|
      r^{-\gamma} + Ar^{-\gamma}) \]
    over $T_{\lambda, \epsilon\lambda}$, for some $C_1'$. Therefore the claim follows if
    \[ \label{eq:C1'1} C_1' (|\lambda| r^{\kappa - \gamma} +
      |\lambda\epsilon| r^{-\gamma} + Ar^{-\gamma}) < \min \{ (\beta|\lambda| + C_4A) r^{-\gamma} , C_4 A
      r^{-\gamma_1}\}. \]
    Since $A \geq s\overline{d} \geq s\beta|\lambda|$, we have
    \[ \min \{ (\beta|\lambda| + C_4A) r^{-\gamma} , C_4 A
      r^{-\gamma_1}\} &\geq \frac{1}{2} C_4 A r^{-\gamma} +
      \frac{1}{2}C_4 s\beta|\lambda| r^{-\gamma} \\
      &\geq C_1' (A r^{-\gamma} +  |\lambda| r^\kappa
      r^{-\gamma} + |\lambda\epsilon| r^{-\gamma}), \]
    as long as $C_4$ is large and $r, \epsilon$ are sufficiently small
    (depending on $\beta, s$). 
  \end{itemize}

  We have now achieved our first goal, \eqref{eq:goal1}. In the
  definition of $\overline{d}$ if we have $\overline{d} = d + s^{-1}
  D(r_0)$, then we obtain \eqref{eq:nc20}. For the remainder of the
  proof we can therefore assume that $\overline{d} = \beta|\lambda|$,
  and $d + s^{-1}D(r_0) \leq \beta |\lambda|$. In particular we have
  $A = C_3 \beta^{-2} s |\lambda|$, and so on $\Omega_2$ the surface
  $M$ lies between $H(\lambda y^\ell \pm C_3\beta^{-2} s|\lambda|)$. 

We claim that $M$ is bounded between the graphs of $\pm \frac{1}{2}
\beta |\lambda| r^{-\gamma}$ over $T_{\lambda, \epsilon\lambda}$ on $\Omega_2$,
decreasing $r_0, \epsilon$ if necessary. To see this, note that arguing as
before we have that $M$ is between the graphs of
 \[ \pm C_1'(|\lambda| r^{\kappa -\gamma} + |\lambda\epsilon|
   r^{-\gamma} + C_3\beta^{-2} s|\lambda|
   r^{-\gamma})\]
over $T_{\lambda, \epsilon\lambda}$. The claim follows if $r_0,
\epsilon$ and $s$ are sufficiently 
small (depending on $\beta$). 
  
We now use graphical barrier surfaces, based on the function
$F_{-\gamma_1}$ defined in 
Lemma~\ref{lem:Fa2}. Note first that if $\epsilon$ is sufficiently
small, then $F_{-\gamma_1}$ satisfies the same estimates on
$T_{\lambda, \epsilon\lambda}$ as on $T$. This follows by rescaling
and that $T_{\lambda, \epsilon\lambda}$ is a small perturbation of
$T_\lambda$. 

We set
\[ g(y) = (|y| - 0.52)^{-1} + (0.98 - |y|)^{-1}, \]
  and we work on the region
  \[ \Omega_3 = \{|y| \in (0.52+s, 0.98-s)\}\cap \{ r < r_0\}.\]
  On this region
  we have bounds for the derivatives of $g$ depending on $s$, and also
  $g(y) > 4$. Let $G =
  g(y) F_{-\gamma_1}$, and define the function
  \[ \widetilde{G}_t = \min \{ \beta |\lambda| r^{-\gamma},
    \epsilon G\}. \]

  We claim that on $\Omega_3$ the graph of $\widetilde{G}_t$
  over $T_{\lambda, \epsilon\lambda}$ has negative curvature at all points where
  $\widetilde{G}_t < \beta |\lambda| r^{-\gamma}$ (in particular
  at any intersection point of $M$ with the graph). To see this, note
  first that
  \[ L_{T_{\lambda, \epsilon\lambda}} G &= g(y) L_{T_{\lambda, \epsilon\lambda}} F_{-\gamma_1} + 2\nabla
    g\cdot \nabla F_{-\gamma_1} + F_{-\gamma_1} \Delta g. \]
  Using that $|\nabla y| \leq 1$ and $\Delta y =0$ on $T_{\lambda, \epsilon\lambda}$, we
  have $|\nabla g|, |\Delta g| \leq C_s$ for a constant $C_s$
  depending on the $C^2$-norm of $g$ (i.e. on $s$). It follows that
  \[ L_{T_{\lambda, \epsilon\lambda}} G &\leq - g(y) C_{\gamma_1}^{-1} r^{-\gamma_1-2} +
    C_s r^{-\gamma_1-1} \\
    &\leq - \frac{1}{2} g(y) C_{\gamma_1}^{-1} r^{-\gamma_1-2}, \]
  as long as $r < r_0$ for $r_0$ depending on $s$. 
  To estimate the mean curvature $m_{T_{\lambda, \epsilon\lambda}}(t G)$, note
  that we have
  \[ r^{-1}|t G| + |\nabla t G| + r |\nabla^2 t
    G| \leq C_s t g(y) r^{-\gamma_1 - 1}, \]
  for $C_s$ depending on $s$. Since $G > 4g(y)C_{\gamma_1}^{-1}
  r^{-\gamma_1}$, on the set where $t G < \beta
  |\lambda| r^{-\gamma}$ we have 
  \[ t g C_{\gamma_1}^{-1} r^{-\gamma_1} < \beta |\lambda|
    r^{-\gamma}, \]
  and so
  \[ r^{-1}|t G| + |\nabla t G| + r|\nabla^2 t G|
    \leq C_s C_{\gamma_1} \beta |\lambda| r^{-\gamma-1}. \]
  At the same time on $T_{\lambda, \epsilon\lambda}$, on the region where $|y|\in (1/2,
  1)$, we have a uniform bound $r > c_2|\lambda|^{1/(\gamma+1)}$
  (since the $y$-slices of $T_{\lambda, \epsilon\lambda}$ are close to $H(\lambda
  y^\ell)$). In particular we have
  \[ r^{-1}|t G| + |\nabla t G| + r|\nabla^2 t G|
    \leq C_s C_{\gamma_1}c_2^{-\gamma-1} \beta. \]
  As long as $\beta$ is sufficiently small (depending on $s$), the
  non-linear terms in the expansion of $m_{T_\lambda}(\epsilon G)$
  will then be bounded by 
  \[ C_s \beta t g(y) r^{-\gamma_1-2}, \]
  for $C_s$ depending on $s$. 
  This can be seen by rescaling the regions where $r\in (R,2R)$ by a
  factor $R^{-1}$ as we have done before, using also that the
  regularity scale of $T_{\lambda, \epsilon\lambda}$ is comparable to $r$ at each point. 
  Therefore on the set where $t G < \beta|\lambda| r^{-\gamma}$
  we have
  \[ m_{T_{\lambda, \epsilon\lambda}}(t G) &\leq -\frac{1}{2} t g(y)
    C_{\gamma_1}^{-1} r^{-\gamma_1-2} + C_s \beta t
    g(y) r^{-\gamma_1-2}  \\ &\leq
     -\frac{1}{4} t g(y)
     C_{\gamma_1}^{-1} r^{-\gamma_1-2} < 0\]
   as long as $\beta$ is chosen to be sufficiently small.

   The conclusion is that for all $t$, the graph of
   $\widetilde{G}_t$ over $T_{\lambda, \epsilon\lambda}$ has negative mean
   curvature on the region where $\widetilde{G}_t < \beta
   |\lambda| r^{-\gamma}$. If we choose $t$ to be sufficiently
   large, then $\widetilde{G}_t = \beta|\lambda| r^{-\gamma}$,
   since $t G > t C_{\gamma_1}^{-1} r^{-\gamma_1}$. In
   particular for such $t$ the surface $M$ lies on the negative
   side of the graph of $\widetilde{G}_t$ on the region
   $\Omega_3$ (recall that here $M$ lies between the graphs of
   $\pm \frac{1}{2} \beta |\lambda| r^{-\gamma}$ over $T_{\lambda, \epsilon\lambda}$).

   Let us set $t_2 = C_6(D(r_0) + sd)$. We claim that for large
   enough $C_6$ the surface $M$ lies on the negative side of the graph
   of $\widetilde{G}_t$ over $T_{\lambda, \epsilon\lambda}$
   on the boundary  $\partial \Omega_3$
   for all $t \geq t_2$. We look at the two types of
   boundary components:
   \begin{itemize}
     \item On the boundary piece $\{r = r_0\}$ we know that $M$ is
       on the negative side of the graph of $D(r_0) r^{-\gamma}$ over
       $T_{\lambda, \epsilon\lambda}$, and
       so we just need to check that $D(r_0) r^{-\gamma} \leq
       \widetilde{G}_t$. This is implied by
       \[ D(r_0) r^{-\gamma} < C_{\gamma_1}^{-1} t_2
         r^{-\gamma_1} = C_{\gamma_1}^{-1} C_6(D(r_0) +
         sd) r^{-\gamma_1}, \]
       as soon as $C_6 > C_{\gamma_1}$.
     \item On the boundary pieces of $\Omega_3$ where $|y| = 0.52+s$
       or $|y| = 0.98-s$, we use that $D_{T_{\lambda, \epsilon\lambda}}(M; B_1\setminus
       B_{1/2})=d$, and that at this point we know that $d <
       \beta|\lambda|$. By definition this means that on $\Omega_3$
       the surface $M$ is bounded between the graphs of
       \[ \pm \min\{ (\beta|\lambda| + d) r^{-\gamma}, d
         r^{-\gamma_1}\} \]
       over $T_{\lambda, \epsilon\lambda}$, and in particular $M$ is on the negative side
       of the graph of $d r^{-\gamma_1}$. At the same time on this
       boundary piece $g(y) > s^{-1}$. We therefore need to check
       that
       \[ d r^{-\gamma_1} < C_{\gamma_1}^{-1} C_6(D(r_0) + sd)
         s^{-1} r^{-\gamma_1}, \]
       which follows if $C_6 > C_{\gamma_1}$. 
     \end{itemize}

     Since the graphs of $\widetilde{G}_t$ have negative mean
     curvature on the set where $\widetilde{G}_t < \beta
     |\lambda| r^{-\gamma}$, as we decrease $t$ to
     $t_2$, the surface $M$ cannot have any interior contact
     point with these graphs (note that $M$ lies on the negative side
     of the graph of $\frac{1}{2}\beta|\lambda| r^{-\gamma}$ so any
     contact must happen where the graph of $\widetilde{G}_t$
     has negative mean curvature). It follows that $M$ lies on the
     negative side of the graph of $\widetilde{G}_{t_2}$, and
     by a similar argument $M$ lies on the positive side of the graph
     of $-\widetilde{G}_{t_2}$. On the region $(B_{4/5}
     \setminus B_{3/5} ) \cap \{r < r_0\}$ we have a uniform upper
     bound $g(y) < C_7$ (independent of $s$), and so we find that on
     this region $M$ lies between the graphs of $\pm C_7C_{\gamma_1}
     t_2 r^{-\gamma_1}$ over $T_{\lambda, \epsilon\lambda}$, i.e. between $\pm C_8
     (D(r_0) + sd) r^{-\gamma_1}$ for some $C_8$. Since we already
     know that $M$ lies between the graphs of $\pm \frac{1}{2}\beta
     |\lambda| r^{-\gamma}$ over $T_{\lambda, \epsilon\lambda}$,
     it follows that $M$ lies between the
     graphs of
     \[ \pm \min\{ (\beta |\lambda| + C_8(D(r_0)+sd) )r^{-\gamma},
       C_8(D(r_0) + sd) r^{-\gamma_1}\}. \]
     By definition this means that
     \[ D_{T_{\lambda, \epsilon\lambda}} (M; (B_{4/5}\setminus B_{3/5})\cap \{ r < r_0\})
       \leq C_8( D(r_0) + sd). \]
     This completes the proof of \eqref{eq:nc20}. 
\end{proof}

Above we defined the distance $D_{T_{\lambda, \epsilon\lambda}}(M; U)$ over sets
$U\subset B_1\setminus B_{1/2}$. For more general sets $U\subset
\mathbf{R}^{n+1}\setminus \{0\}$ we define the distance as follows:
\[  D_{T_{\lambda, \epsilon\lambda}}(M ; U) = \sup_{\Lambda > 0} D_{\Lambda
    T_{\lambda, \epsilon\lambda}}(\Lambda M; \Lambda U \cap (B_1\setminus B_{1/2})). \]
Note that the rescalings $\Lambda T_{\lambda, \epsilon\lambda}$ are all of the form
$T_{\lambda', \epsilon\lambda'}$. 

Using this non-concentration estimate we obtain a three-annulus lemma analogous to
\cite[Proposition 5.12]{Sz20}. We fix a value $d\in\mathbf{R}$ and
we will apply the $L^2$ 3-annulus lemma,
Lemma~\ref{lem:L23annulus}. From this lemma we obtain $\alpha_0,\alpha_0',
\rho_0 > 0$, depending on $d$ and the cone $C$. To simplify the
notation below from now on we will
write
\[ D_{T_{\lambda, \epsilon\lambda}}(M) := D_{T_{\lambda, \epsilon\lambda}}(M; B_1\setminus B_{\rho_0}). \]
We then have the following result, whose proof is essentially
identical to that of \cite[Proposition 5.2]{Sz20} and so we will omit
it.
\begin{prop}\label{prop:DT3annulus2}
  Let $d\in \mathbf{R}$, and $\alpha_0, \alpha_0', \rho_0 > 0$ be the corresponding
  values obtained in Lemma~\ref{lem:L23annulus}. There is a 
  large integer $k_0 > 0$ and a small $\delta > 0$ with the following
  property. Let $L = \rho_0^{-k_0}$, and suppose that $M$ is a
  stationary integral varifold in $B_{2L^2}$ satisfying
  \[ \label{eq:Marea2} \Vert M\Vert(B_{2L^2}) \leq \frac{11}{10} \Vert C\times
    \mathbf{R}\Vert (B_{2L^2}). \]
  Suppose that $|\lambda|, |\epsilon| < \delta$ and $D_{T_{\lambda, \epsilon\lambda}}(M) <
  \delta$. Then for any $\alpha\in(\alpha_0, \alpha_0')$ we have
  \begin{enumerate}
  \item[(i)] If $D_{LT_{\lambda, \epsilon\lambda}}(LM) \geq L^{1-d+\alpha}
    D_{T_{\lambda, \epsilon\lambda}}(M)$, then \\
    $D_{L^2T_{\lambda, \epsilon\lambda}}(L^2M) \geq L^{1-d+\alpha}
    D_{LT_{\lambda, \epsilon\lambda}}(LM)$.
  \item[(ii)] If $D_{L^{-1}T_{\lambda, \epsilon\lambda}}(L^{-1}M) \geq L^{d-1+\alpha}
    D_{T_{\lambda, \epsilon\lambda}}(M)$, then \\
    $D_{L^{-2}T_{\lambda, \epsilon\lambda}}(L^{-2}M) \geq L^{d-1+\alpha}
    D_{L^{-1}T_{\lambda, \epsilon\lambda}}(L^{-1}M)$.
  \end{enumerate}
\end{prop}

We will also need to use the following estimate, analogous to the
triangle inequality. The proof is similar to that of
Lemma 6.1 in \cite{Sz20}, except it is slightly simpler since we do
not need to consider rotations.
\begin{lemma} \label{lem:triangle}
  There is a constant $C_1 > 0$ (depending on $\beta$) with the
  following property. Suppose that $|\lambda|, |\epsilon|, |\epsilon'| < C_1^{-1}$,
  and $M$ is a subset of the annulus $B_1\setminus B_{1/2}$ such that
  $D_{T_{{\lambda, \epsilon\lambda}}}(M) < C_1^{-1}$. Then we have
  \[ \label{eq:tri1} D_{T_{{\lambda, \epsilon'\lambda}}}(M) \leq C_1 ( D_{T_{\lambda,
        \epsilon\lambda}}(M) + |\lambda (\epsilon - \epsilon')|), \]
  and 
  \[ \label{eq:tri2} D_{T_{\lambda, \epsilon\lambda}}(M) \leq C_1( D_{T_0}(M) +
    |\lambda|). \]
\end{lemma}
\begin{proof}
  Let us consider \eqref{eq:tri1}. Let $d = D_{T_{\lambda,
      \epsilon\lambda}}(M)$, and write $K = d + |\epsilon -
  \epsilon'||\lambda|$.  Using the definition of the distance we need
  to show that $M \subset 
  N_{C_1K}(T_{\lambda, \epsilon'\lambda})$ for some $C_1 > 0$.

 We consider two cases. 
  \begin{itemize}
    \item Suppose that $d \geq \beta|\lambda|$.
      By definition of the distance this means that $M$ is
      bounded between $H(\pm \beta^{-2}d)\times\mathbf{R}$,
      and so we also have $M \subset N_d(T_{\lambda,
          \epsilon'\lambda})$ as required (note that $C_1K \geq d$ if
        $C_1\geq 1$).
     \item Suppose that $d < \beta |\lambda|$. Then by definition $M$
       lies between the graphs of
       \[ \pm \min\{ (\beta|\lambda| + d) r^{-\gamma},
         dr^{-\gamma_1}\} \]
       over $T_{\lambda, \epsilon\lambda}$. However, using
       Proposition~\ref{prop:Tdeform}, $T_{\lambda,
         \epsilon\lambda}$ is the graph of a function $f$ over
       $T_{\lambda, \epsilon'\lambda}$ with $|f| \lesssim
       |\epsilon-\epsilon'| |\lambda| \rho^\l r^{-\gamma}$, and on our
       annulus we have $\rho\in(1/4, 2)$. It follows from this that $M$
       lies between the graphs of
       \[ \pm \min\{ (\beta|\lambda| + d + C_1 |\epsilon-\epsilon'|
         |\lambda|) r^{-\gamma}, (d + C_1 |\epsilon-\epsilon'|
         |\lambda|) r^{-\gamma_1}\} \]
       over $T_{\lambda, \epsilon'\lambda}$
       for some $C_1 > 0$. This implies that $M \subset
       N_{C_1K}(T_{\lambda, \epsilon'\lambda})$ as required.
     \end{itemize}
     This completes the proof of \eqref{eq:tri1}. The proof of
     \eqref{eq:tri2} is very similar so we will omit it. 
\end{proof}

\subsection{Proof of Theorem~\ref{thm:symmetric}}\label{sec:thm3proof}
In this section we will turn to the proof of
Theorem~\ref{thm:symmetric}. Let $\delta_0 > 0$, and suppose that
$M$ is as in Theorem~\ref{thm:symmetric} with the additional property
that all rescalings $\Lambda M$ for $\Lambda \geq 1$ satisfy
$D_{C\times\mathbf{R}}(\Lambda M; B_1) < \delta_0$ as well as the
area bounds \eqref{eq:Marea2}. Note that we can always arrange this by
first scaling $M$ up. If $\delta_0$ is chosen sufficiently
small,  the triangle
inequality, Lemma~\ref{lem:triangle}, also implies that
$D_{T_{\lambda, \epsilon\lambda}}(\Lambda M) < 2C_1\delta_0$ as long
as $|\lambda|, |\epsilon| <
\delta_0$.  Then the 3-annulus lemma, Proposition~\ref{prop:DT3annulus2},
with $d=\l-\gamma$ will apply to all blowups $\Lambda M$ for $\Lambda
\geq 1$. From now on $\alpha_0, \alpha_0', \rho_0$ will denote the constants from
Proposition~\ref{prop:DT3annulus2}, for the value $d=\l-\gamma$.
Note also that by \eqref{eq:Tlambda} we have that if
$|\lambda| < \delta_0$ and $T_{\lambda', \epsilon\lambda'} = \Lambda
T_{\lambda, \epsilon\lambda}$ for
$\Lambda \geq 1$, then $|\lambda'| \leq |\lambda| < \delta_0$. Using
this we can iterate the statement (i) in
Proposition~\ref{prop:DT3annulus2}.

We have the following consequence of the fact that $M$ has degree
$\l-\gamma$. For simplicity we will write $d_M = \l-\gamma$. 
\begin{lemma}\label{lem:DTlgrowth} There is a $\delta_0 > 0$ with
  the following property. 
  Suppose that $M$ is as in Theorem~\ref{thm:symmetric}, has degree
  $d_M = \l-\gamma$, and $D_{C\times\mathbf{R}}(\Lambda M; B_1) <
  \delta_0$ for all
  $\Lambda \geq 1$ as above. Then for $|\lambda|, |\epsilon| < \delta_0$
  and any $\alpha\in (\alpha_0, \alpha_0')$ we have
  \[ D_{L T_{\lambda, \epsilon\lambda}}(LM) \leq L^{1-d_M+\alpha}
    D_{T_{\lambda, \epsilon\lambda}}(M), \]
  where $\alpha_0, \alpha_0'$ are the constants in
  Proposition~\ref{prop:DT3annulus2}. 
\end{lemma}
\begin{proof}
  Suppose that the conclusion does not hold. If $\delta_0$ is
  sufficiently small, then using
  Proposition~\ref{prop:DT3annulus2} we have 
  \[ \label{eq:DL10} D_{L^{k} T_{\lambda, \epsilon\lambda}}(L^{k}M) >
    L^{k(1-d_M+\alpha)} D_{T_{\lambda, \epsilon\lambda}}(M) \] 
  for all $k \geq 0$. Using Lemma~\ref{lem:triangle} and the
  formulas \eqref{eq:Tlambda} it follows that 
  \[ D_{L^{k} T_{\lambda, \epsilon\lambda}} (L^{k}M) \leq C_1 \Big(D_{T_0}(L^{k} M) + L^{k(1 -d_M)}
    |\lambda| \Big). \]
  Combining this with \eqref{eq:DL10} we have
  \[ \label{eq:DLkj} D_{T_0}(L^{k}M) > C_1^{-1}
    L^{k(1-d_M+\alpha)} D_{T_{\lambda, \epsilon\lambda}}(M) -
    L^{k(1-d_M)}|\lambda|. \]
  We will see that this rate of growth
  contradicts that $M$ has degree $d_M$ once $k$ is sufficiently
  large. 

  Indeed, by the assumption that $M$ has degree $d_M$,
  we can choose a sequence $k_i \to \infty$ such
  that in the setting of Corollary~\ref{cor:nonzeroU} the limiting
  Jacobi field obtained from the sequence $L^{k_i}M$ is a Jacobi field
  $U$ with degree $d_M$. More precisely,  set $a_i =
  D_{C\times\mathbf{R}}(L^{k_i}M; B_1) \to 0$. Then there is a sequence
  $r_i\to 0$ such that $L^{k_i}M$ is the graph of a function $u_i$
  over $C\times\mathbf{R}$ on the region $B_1\cap \{ r > r_i\}$, and
  $a_i^{-1}u_i \to U$ locally uniformly on compact subsets of
  $B_1\setminus\{r=0\}$. 

  For any $i, j > 0$ we have that $L^{k_i + j}M$ is the graph of the
  function $u_{i,j} = L^j u_i(\cdot / L^j)$ over $C\times\mathbf{R}$ on the
  region $\{r > L^j r_i\}\cap B_1$. For fixed $j$, we have $a_i^{-1}
  u_{i,j} \to U_j$ locally uniformly on $B_1\setminus \{r = 0\}$ as
  $i\to\infty$, where $U_j = L^j U(\cdot / L^j)$. Since $U$ is a sum
  of homogeneous Jacobi fields with degrees at least $d_M$, it follows
  that
  \[ \Vert U_j\Vert_{L^2(B_2)} \leq L^{j(1-d_M)} \Vert
    U\Vert_{L^2(B_2)}, \]
  and so using Lemma~\ref{lem:L2Linfty} we have
  \[ \sup_{B_1} |r^\gamma U_j| \leq C_2 L^{j(1-d_M)} \]
  for all $j > 0$, for a constant $C_2$. By
  the scaling property of the distance $D_{C\times\mathbf{R}}$ (see
  \eqref{eq:DCRscale}) we also have
  \[ D_{C\times\mathbf{R}}(L^{k_i + j}M; B_2) \leq L^{(1+\gamma)j}
    D_{C\times\mathbf{R}}(L^{k_i}M ; B_1) =
    L^{(1+\gamma)j }a_i. \]
  The non-concentration estimate, Proposition~\ref{prop:nonconc}, then
  implies that given a fixed $j$, for any $s > 0$ we can choose $i$
  sufficiently large (depending on $j, s$) such that
  \[ D_{C\times\mathbf{R}}(L^{k_i+j}M;B_1) \leq C_3 L^{j(1-d_M)} a_i + s
    L^{(1+\gamma)j} a_i.\]
  Given $j$, we can choose $s$ such that $s L^{(1+\gamma)j} =
  L^{j(1-d_M)}$, and so we get
  \[ D_{C\times\mathbf{R}}(L^{k_i+j}M;B_1) \leq 2C_3 L^{j(1-d_M)} D_{C\times\mathbf{R}}(L^{k_i}M;B_1), \]
  for sufficiently large $i$ (depending on $j$). The constant $C_3$ here is independent of $i, j$.
   At the same time, using the hypersurfaces $H(\pm t)\times\mathbf{R}$
   as barriers, we see that
   \[ D_{C\times\mathbf{R}} (M ; B_1)\leq D_{C\times\mathbf{R}}(M ;
     B_1\setminus B_{\rho_0}) \]
   for any stationary integral varifold $M$ in $B_1$.
   It follows, using also the equivalence of the two distances given
   in  \eqref{eq:DCRDTequiv},  that for any given $j$, once $i$ is large enough
  we have
  \[ \label{eq:DT0g1} D_{T_0}(L^{k_i+j}M) \leq C_4 L^{j(1-d_M)} D_{T_0}(L^{k_i}M), \]
  for a constant $C_4$ independent of $i,j$.
  
  Choose an $\alpha' \in (\alpha_0, \alpha)$. Using the three annulus lemma
  we have that for sufficiently large $m$
  \[ D_{T_0}(L^{m+1}M) \geq L^{1-d_M + \alpha'} D_{T_0}(L^m M) \]
  implies
  \[ D_{T_0}(L^{m+j}M) \geq L^{j(1-d_M +\alpha')} D_{T_0}(L^m M), \]
  for any $j > 0$. If we choose $j$ sufficiently large, and a
  corresponding $k_i$ for large enough $i$ then this
  contradicts \eqref{eq:DT0g1}. We therefore have
  \[ D_{T_0}(L^{m+1}M) \leq L^{1-d_M + \alpha'} D_{T_0}(L^m M) \]
  for all sufficiently large $m$. In particular there is a constant $C_5$
  such that  for all sufficiently large $m$ we have
  \[ D_{T_0}(L^m M) \leq C_5 L^{m(1-d_M+\alpha')}. \]
  Together with \eqref{eq:DLkj} this means that for sufficiently large
  $m$ we have
  \[ C_1^{-1} L^{m(1-d_M +\alpha)} D_{T_{\lambda, \epsilon\lambda}}(M) - L^{m(1-d_M)}
    |\lambda| < C_5 L^{m(1-d_M+\alpha')}. \]
  This is a contradiction as $m\to\infty$ using that $\alpha >
  \alpha'$. 
\end{proof}

The main ingredient in the proof of Theorem~\ref{thm:symmetric} will
be the following decay estimate. 

\begin{prop}\label{prop:decay2}
  There are constants $C_1, \delta_0, B > 0$
  depending on the cone $C$ and on the degree $d_M = \l-\gamma$ of $M$
  (see Definition~\ref{defn:Mdegree}) with the following property. Suppose
  that $M$ is as in Theorem~\ref{thm:symmetric} and in addition
  $D_{C\times\mathbf{R}}(\Lambda M;B_1) < \delta_0$ for all $\Lambda \geq
  1$, and  $D_{T_{\lambda, \epsilon\lambda}}(M) < \delta_0 |\lambda|$ for some
  $|\lambda|, |\epsilon| < 
  \delta_0$. Then there exists an $\epsilon'$ such that $|\epsilon' -
  \epsilon||\lambda| \leq C_1D_{T_\lambda}(M)$, and
  \[ D_{L^BT_{\lambda, \epsilon' \lambda}}(L^BM) \leq \frac{1}{2} L^{B(1-d_M)}
    D_{T_{\lambda, \epsilon\lambda}}(M). \]
\end{prop}
\begin{proof}
  The crucial ingredient is that the family of comparison
  surfaces $T_{\lambda, \epsilon\lambda}$ with varying $\epsilon$
  accounts for the whole space of $O(p+1)\times
  O(q+1)$-invariant degree $d_M$ Jacobi fields on $C\times\mathbf{R}$,
  which is spanned by $u_\l$. Therefore by choosing the 
  correct reference surface at each scale, we can achieve the required
  decay. A similar idea appears in the proof of the uniqueness of integrable
  tangent cones in~\cite{AA81} as well as in more recent works such as
  \cite{ES20, Sz20} for instance. 
  
  For the sake of contradiction let
  us suppose that no matter how small $\delta_0$ is
  we cannot find suitable constants $C_1, B$. In particular we then
  have a sequence of $O(p+1)\times O(q+1)$-invariant stationary
  integral varifolds $M_i$ of degree $d_M$ such that
  $|\lambda_i|^{-1} D_{T_{\lambda_i, \epsilon_i\lambda_i}}(M_i) \to 0$ for sequences
  $\lambda_i, \epsilon_i \to 0$. Let us also 
  fix a large integer $B > 0$. We will show that if $B$ is
  sufficiently large, then there is a constant $C_1$ such that for sufficiently large
  $i$ we can find $\epsilon_i'$ satisfying the conclusion.

  Let $a_i = D_{T_{\lambda_i, \epsilon_i\lambda_i}}(M_i)$, so that
  $|\lambda_i|^{-1}a_i\to 0$. Using Lemma~\ref{lem:DTgraph} we
  have a sequence $r_i\to 0$
  such that $M_i$ is the graph of a function $u_i$ over
  $T_{\lambda_i, \epsilon_i\lambda_i} $ on the region $(B_1\setminus B_{\rho_0})\cap \{r >
  r_i\}$, where $|r^{\gamma_1} u_i| \leq C_1 a_i$. Up to choosing a
  subsequence we can assume that $a_i^{-1} u_i$ converges locally
  uniformly away from the singular line to a Jacobi field $U$ on
  $C\times\mathbf{R}$ over the annulus $B_1\setminus B_{\rho_0}$
  satisfying $|r^{\gamma_1}U|\leq C_1$. Let us write $U = U_{d_M} +
  U'$, where $U'$ has no degree $d_M$ component. Since we are in the
  $O(p+1)\times O(q+1)$-invariant setting, we have $U_{d_M} = c u_\l $
  in terms of the function $u_\l$ in \eqref{eq:ul}, for a constant
  $c$. 
  We choose $\epsilon_i'$ so that 
  $\lambda_i\epsilon_i' = \lambda_i\epsilon_i + c a_i$. It follows that $|\epsilon_i' -
  \epsilon_i| |\lambda_i|\leq C_1 a_i$, and also by our assumption
  that $|\epsilon_i' - \epsilon_i| \to 0$.  Using Lemma~\ref{lem:triangle} we have
  $D_{T_{\lambda_i, \epsilon_i'\lambda_i}}(M_i) \leq C_2a_i$ for large
  $i$.

  In addition we claim that we can write $M_i$ as the graph of $u_i'$ over
  $T_{\lambda_i, \epsilon_i'\lambda_i}$ over the region $(B_1\setminus B_{\rho_0})\cap \{r
  > r_i'\}$ for some $r_i' \to 0$, such that $a_i^{-1} u_i' \to U'$
  locally uniformly. To see this, first note that by
  Proposition~\ref{prop:Tdeform} we know that $T_{1,
    \epsilon_i}$ is the graph of $\epsilon_i\phi_\l + v_{\epsilon_i}$
  over $T_1$, and similarly $T_{1,\epsilon_i'}$ is the graph of $
  \epsilon_i' \phi_\l + v_{\epsilon_i'}$ over $T_1$. Let us write
  \[ \phi_{\l, \lambda} &= \lambda^{(1-d_M)^{-1}} \phi_\l
    (\lambda^{-(1-d_M)^{-1}}\cdot), \\
    v_{\epsilon, \lambda} &=  \lambda^{(1-d_M)^{-1}}
    v_{\epsilon}(\lambda^{-(1-d_M)^{-1}}\cdot) \]
  on $T_{\lambda}$ for corresponding rescaled functions. In terms of
  these we have that $T_{\lambda_i, \epsilon_i\lambda_i}$ is the graph
  of \[ f = \epsilon_i \phi_{\l,\lambda_i} + v_{\epsilon_i,
      \lambda_i}\]  over
  $T_{\lambda_i}$ and $T_{\lambda_i, \epsilon_i' \lambda_i}$ is the
  graph of \[g = \epsilon_i' \phi_{\l,\lambda_i} + v_{\epsilon_i',
    \lambda_i}\]  over $T_{\lambda_i}$.

  For any fixed $R > 0$, let us focus on the region
  $\Omega_R = (B_1\setminus B_{\rho_0}) \cap \{r > R\}$. Using that $\phi_\l\in
  C^{2,\alpha}_{\l-\gamma, -\gamma}$ as well as the bounds \eqref{eq:vepsbound} for
  $v_\epsilon$, we have $\Vert
  f\Vert_{C^k} \leq C_R \epsilon_i |\lambda_i|$ and $\Vert
  g\Vert_{C^k} \leq C_R \epsilon_i' |\lambda_i|$ on $\Omega_R$, where
  we write $C_R$ for $R$-dependent constants, that may change from line
  to line. Note that we
  automatically obtain derivative bounds for $f, g$ from $C^0$ bounds,
  since $T_{\lambda_i}, T_{\lambda_i, \epsilon_i \lambda_i},
  T_{\lambda_i, \epsilon_i' \lambda_i}$ are all minimal surfaces. In
  addition we have $\Vert f - g\Vert_{C^k} \leq C_R |\epsilon_i -
  \epsilon_i'| |\lambda_i|$. 
  It follows, using a result similar to Lemma~\ref{lem:graphfg}, that
  $T_{\lambda_i, \epsilon_i\lambda_i}$ is the graph of a function $h$
  over $T_{\lambda_i, \epsilon_i' \lambda_i}$ satisfying
  \[ \label{eq:hfg10} \Vert h - (f-g) \Vert_{C^k} \leq C_R ( |\epsilon_i| +
    |\epsilon_i'| ) |\epsilon_i - \epsilon_i'| |\lambda_i|^2.  \]
  Here and below we are comparing functions defined on different
  hypersurfaces, such as $h$ and $f-g$. In order to do this we
  identify the two hypersurfaces using that each is a graph over the
  other. 
  
  From property (2) in Proposition~\ref{prop:TJacfield} we have
\[ |\phi_{\l, \lambda} - \lambda u_\l | &\lesssim
  \lambda^{(d_M-1)^{-1}(\delta-1)} \rho^{\delta-\tau} r^{\tau} \\
  &\leq C_R \lambda^{1+\kappa} \]
  on $\Omega_R$  for some $\kappa > 0$ (recall that $\delta >
  \l-\gamma = d_M$). At the same time 
  \eqref{eq:vepsbound} implies that
  \[ \Vert v_{\epsilon_i', \lambda_i} - v_{\epsilon_i, \lambda_i}
    \Vert_{C^k} \leq
    C_R (|\epsilon_i| + |\epsilon_i'|) |\epsilon_i - \epsilon_i'|
    |\lambda_i|. \]
  It follows that on $\Omega_R$ we have
  \[ \Big\Vert f - g - (\epsilon_i - \epsilon_i') \lambda_i u_\l \Big\Vert_{C^k}
    \leq C_R |\epsilon_i - \epsilon_i'|
    |\lambda_i| \Big( |\lambda_i|^{\kappa} + |\epsilon_i| +|
    \epsilon_i'|\Big).   \]

  Using also \eqref{eq:hfg10}, the conclusion is that we have a
  sequence of numbers $t_i \to 0$ such that on 
  $\Omega_R$ we have
  \[ \label{eq:hest30} \Vert h - (\epsilon_i - \epsilon_i') \lambda_i u_\l \Vert_{C^k}
    \leq C_R t_i |\epsilon_i - \epsilon_i'| |\lambda_i|. \]
  In addition $M_i$ is the graph of $u_i$ over $T_{\lambda_i,
    \epsilon_i\lambda_i}$ which in turn is the graph of $h$ over
  $T_{\lambda_i, \epsilon_i' \lambda_i}$. It follows that $M_i$ is the
  graph of $u_i'$ over $T_{\lambda_i, \epsilon_i' \lambda_i}$, where
  $u_i'$ satisfies
  \[ \Vert u_i' - (u_i + h) \Vert_{C^k} \leq C_R t_i |\epsilon_i -
    \epsilon_i' | |\lambda_i| \]
  for a new sequence $t_i\to 0$,  using also that $|u_i| \to 0$
  on $\Omega_R$. Combining this with \eqref{eq:hest30} we have
  \[ \Vert u_i' - [u_i + (\epsilon_i - \epsilon_i') \lambda_i u_\l]
    \Vert_{C^k} \leq C_R t_i |\epsilon_i - \epsilon_i'|
    |\lambda_i|, \]
  for some $t_i \to 0$. Recall that
  $a_i^{-1} \lambda_i (\epsilon_i - \epsilon_i') = - c$ by the choice
  of $\epsilon_i'$. It follows that
  \[ a_i^{-1} u_i' = a_i^{-1}u_i - cu_\l + O( C_R t_i). \]
  By assumption $a_i^{-1} u_i \to cu_\l + U'$ locally
  uniformly on $C\times\mathbf{R}$ away from the singular line, and
  so we can choose a sequence $r_i' \to 0$ such that $M_i$ is the graph of
  $u_i'$ on $\{r > r_i'\}$ and in addition $a_i^{-1} u_i' \to U'$ locally
  uniformly on $C\times\mathbf{R}$ as required. 

  Lemma~\ref{lem:DTlgrowth} implies that if $i$ is
  sufficiently large, we have
  \[ \label{eq:ai'} a_i' = D_{L^BT_{\lambda_i, \epsilon_i'\lambda_i}}(L^BM_i) \leq C_B a_i, \]
  where $C_B$ depends on $B$ (as well as on the cone $C$ and the
  degree $d_M$). It remains to show that if $B$ is chosen sufficiently
  large, then $a_i' \leq \frac{1}{2} L^{B(1 - d_M)}a_i$ once $i$ is
  large too. We can assume that up to choosing a subsequence we have
  $a_i' / a_i \to c_1$ for a constant $0 < c_1 \leq C_B$, and so
  $a_i'^{-1} u_i' \to c_1^{-1} U'$. We treat two cases separately.

  \bigskip
  \noindent {\em Case 1.}
  First suppose that $U'=0$. In particular for any $s > 0$ we have
  $|u_i'| < sa_i$ on the set $\{ r > s\}$, once $i$ is sufficiently
  large. Note also that similarly to \eqref{eq:ai'} we also have
  \[ D_{\Lambda  T_{\lambda_i, \epsilon_i'\lambda_i}} (\Lambda M_i) \leq C_B a_i, \]
  for a possibly larger constant $C_B$, for all $\Lambda \in [1,
  L^{2B}]$. It then follows from the
  non-concentration estimate, Proposition~\ref{prop:nonconc2}, that for
  any $s > 0$, once $i$ is sufficiently large, we have
  \[ \label{eq:ai'10} a_i' = D_{L^BT_{\lambda_i,
        \epsilon_i'\lambda_i}}(L^BM) \leq C_1( s a_i + s C_B a_i ). \] 
  Choosing $s$ sufficiently small (depending on $B$) we have the required inequality
  $a_i' \leq \frac{1}{2} L^{B(1-d_M)} a_i$ once $i$ is large.

  \bigskip
  \noindent {\em Case 2.}
  Suppose now that $U'\not= 0$. Let us fix $\alpha_1 \in (\alpha_0,
  \alpha_0')$ for the $\alpha_0, \alpha_0'$ 
  in the 3-annulus lemma. We claim first that either $a_i' \leq
  \frac{1}{2} L^{B(1-d_M)}a_i$, or at least one of the 
  following inequalities holds for sufficiently large $i$, if $B$ is
  chosen large enough:
  \begin{itemize}
    \item[(i)] $D_{T_{\lambda_i, \epsilon_i'\lambda_i}}(M_i) \geq L^{B(d_M-1 + \alpha_1)}
      a_i'$,
    \item[(ii)] $D_{L^{2B}T_{\lambda_i, \epsilon_i'\lambda_i}}(L^{2B}M_i) \geq
      L^{B(1-d_M+\alpha_1)} a_i'$.
    \end{itemize}
  To see this, suppose that both (i), (ii) fail for all sufficiently
  large $i$. Then using Lemma~\ref{lem:DTgraph} we find that
  the limit $U'' = c_1^{-1}U'$ of the sequence
  $a_i'^{-1} u_i'$ satisfies the bounds
  \[ L^{2B} |U''| &\leq C_3 L^{B(1 - d_M + \alpha_1)} r^{-\gamma_1} \text{ on }
    B_{L^{-2B}} \setminus B_{L^{-2B}\rho_0}, \\
      |U''| &\leq C_3 L^{B(d_M - 1 + \alpha_1)} r^{-\gamma_1} \text{
        on } B_1\setminus B_{\rho_0}. \]
    In terms of the $L^2$ norms in \eqref{eq:L2ann} this means that
   replacing $C_3$ by a larger constant (depending on the cone $C$) we have
    \[ \label{eq:U''b2} \Vert U''\Vert_{\rho_0, 2Bk_0} &\leq C_3 L^{-B} L^{B(-d_M +
        \alpha_1)}, \\
      \Vert U''\Vert_{\rho_0, 0} &\leq C_3 L^{-B} L^{B(d_M +
        \alpha_1)}, \]
    where we use that $L = \rho_0^{-k_0}$. Let $\alpha_2 > \alpha_1$
    with $\alpha_2\in (\alpha_0, \alpha_0')$. Since $U''$ has no degree
    $d_M$ component, it follows from Lemma~\ref{lem:L23annulus} that
    one of the following must hold:
    \begin{itemize}
      \item[(a)] $\Vert U''\Vert_{\rho_0, (B+1)k_0+1} \geq
        \rho_0^{d_M-\alpha_2} \Vert U''\Vert_{\rho_0, (B+1)k_0}$,
        \item[(b)] $\Vert U''\Vert_{\rho_0, (B+1)k_0 - 1} \geq
          \rho_0^{-d_M-\alpha_2} \Vert U''\Vert_{\rho_0, (B+1)k_0}$.
        \end{itemize}
      Let us fix $\eta > 0$. Applying Lemma~\ref{lem:L23annulus} repeatedly, we find that in
      case (a) we have
      \[ \Vert U''\Vert_{\rho_0, (B+1)k_0 + (B-1)k_0} \geq
        \rho_0^{(B-1)k_0 (d_M-\alpha_2)} \Vert U''\Vert_{\rho_0,
          (B+1)k_0}, \]
      and so from \eqref{eq:U''b2} we have
      \[ \Vert U''\Vert_{\rho_0,
          (B+1)k_0} &\leq C_3 L^{-B} \rho_0^{B(d_M-\alpha_1)k_0 -
          (B-1)k_0(d_M-\alpha_2)} \\
          &= C_3 L^{-B} \rho_0^{k_0(d_M - \alpha_2)}
          \rho_0^{Bk_0(\alpha_2 - \alpha_1)}. \]
        Since $\alpha_2 > \alpha_1$, by choosing $B$ sufficiently
        large we can ensure that
        \[ \Vert U'' \Vert_{\rho_0, (B+1)k_0} \leq \eta L^{-B}. \]
        Similarly, in case (b) by applying Lemma~\ref{lem:L23annulus}
        repeatedly we have
      \[ \Vert U''\Vert_{\rho_0, (B+1)k_0 - (B+1)k_0} \geq
        \rho_0^{-(B+1)k_0(d_M+\alpha_2)} \Vert U''\Vert_{\rho_0,
          (B+1)k_0}, \]
      and so using \eqref{eq:U''b2} we have
      \[ \Vert U''\Vert_{\rho_0,
          (B+1)k_0} &\leq C_3L^{-B}
        \rho_0^{k_0(d_M+\alpha_2)}\rho_0^{Bk_0(\alpha_2-\alpha_1)}.  \]
      One again if $B$ is sufficiently large then we have
      \[ \Vert U''\Vert_{\rho_0,
          (B+1)k_0} \leq \eta L^{-B}. \]
      We can argue similarly to show that if $B$ is sufficiently
      large, then
      \[ \Vert U''\Vert_{\rho_0, Bk_0 + j} \leq \eta L^{-B} \]
      as long as $|j| \leq k_0$. Using the $L^\infty$ estimate
      Lemma~\ref{lem:L2Linfty} we find that
      \[ |r^\gamma U''| \leq C_4 \eta L^{-B} \]
      on the annulus $B_{L^{-B+1/2}}\setminus B_{L^{-B-1/2}\rho_0}$, for a
      constant $C_4$. Similarly to \eqref{eq:ai'10}, using the
      non-concentration estimate we obtain the following: for any $s >
      0$, once $i$ is sufficiently large, we have
      \[ a_i' = D_{L^BT_{\lambda_i, \epsilon_i'\lambda_i}}(L^BM_i) \leq C_1(C_4\eta a_i'
        + s C_B a_i). \]
      Recall that the choice of a small $\eta > 0$ determines how
      large we must take $B$, which in turn determines the constant
      $C_B$. We can therefore first choose $\eta$ sufficiently small,
      then $s$ small and finally $i$ sufficiently large so that we
      have $a_i' \leq \frac{1}{2} L^{B(1-d_M)} a_i$ as required. 
     
      It remains to deal with the case when either (i) or (ii) above
      holds for all sufficiently large $i$. Note that
      Lemma~\ref{lem:DTlgrowth} rules out case (ii) for large $i$, so
      we can assume that (i) holds. As before, we have
      $D_{T_{\lambda_i, \epsilon_i'\lambda_i}}(M_i) \leq C_2 a_i$, so (i) implies
      \[ a_i' = D_{L^B T_{\lambda_i, \epsilon_i'\lambda_i}}(L^B M_i) \leq C_2
        L^{B(1-d_M)}L^{-B\alpha_1} a_i. \]
      Since $\alpha_1 > 0$, for sufficiently large $B$ we have the
      required inequality $a_i' \leq \frac{1}{2} L^{B(1-d_M)} a_i$. 
    \end{proof}
    
It will be convenient to also state a rescaled version of this
proposition:
\begin{cor} \label{cor:decay2}
  For the same $C_1, \delta_0, B > 0$ as in
  Proposition~\ref{prop:decay2} we have the following. Suppose that
  $M$ is as in Theorem~\ref{thm:symmetric} and
  $D_{C\times\mathbf{R}}(\Lambda M;B_1) < \delta_0$ for all
  $\Lambda\geq 1$. In addition suppose that for some $\Lambda \geq 1$
  and $|\lambda|, |\epsilon| < \delta_0$ 
  we have $D_{\Lambda T_{\lambda, \epsilon\lambda}}(M) <
  \Lambda^{1-d_M} \delta_0|\lambda|$. Then we can find an $\epsilon'$
  satisfying $|\epsilon' - \epsilon||\lambda| \leq C_1
  \Lambda^{(1-d_M)^{-1}} D_{\Lambda T_{\lambda, \epsilon\lambda}}(M)$
  such that
  \[ D_{L^B\Lambda T_{\lambda, \epsilon'\lambda}}(L^BM) \leq
    \frac{1}{2} L^{B(1-d_M)} D_{\Lambda T_{\lambda,
        \epsilon\lambda}}(M). \]
\end{cor}

To deduce the graphicality of $M$ we will use the following. 
\begin{lemma}\label{lem:graphT}
  Given $B > 0$, there exists $\delta > 0$ with the following
  property. Suppose that $|\lambda|, |\epsilon| < \delta$, and in addition $M$ is
  a stationary integral varifold in $B_2(0)$ with the area bound
  \eqref{eq:Mareabound} satisfying the bounds
  \[ D_{T_\lambda}(M), D_{L^BT_{\lambda, \epsilon\lambda}}(L^BM) < \delta |\lambda|. \]
  Then $M$ is a graph over $T_{\lambda, \epsilon\lambda}$ on the annulus $B_1\setminus
  B_{L^{-B}\rho_0}$. 
\end{lemma}
\begin{proof}
  We claim that there is a constant $C_B > 0$ depending on $B$ such
  that if $\delta$ is sufficiently small, then the hypotheses in the
  statement of the lemma imply
  that for all $\Lambda\in [0, L^B]$ we have
  \[ D_{\Lambda T_{\lambda, \epsilon\lambda}}(\Lambda M) < C_B \delta |\lambda|. \]
  Indeed, suppose that no such $C_B$ exists. Then we can find
  sequences $M_i$ and $\lambda_i, \epsilon_i$ with $|\lambda_i|,
  |\epsilon_i|\to 0$ such that on 
  the one hand
  \[ \label{eq:DTl20} D_{T_{\lambda_i, \epsilon_i\lambda_i}}(M_i),
    D_{L^B T_{\lambda_i, \epsilon_i\lambda_i}}(L^B M_i) < 
    \frac{1}{i} |\lambda_i|, \]
  but if we set
  \[ \label{eq:aimax20} a_i = \max_{\Lambda} D_{\Lambda T_{\lambda_i,
        \epsilon_i\lambda_i}}(\Lambda M_i), \] 
  then $a_i i |\lambda_i|^{-1} \to \infty$. Note that by the
  hypotheses we have $M_i \to C\times\mathbf{R}$ on $B_1$, and so
  $a_i\to 0$. Using Lemma~\ref{lem:DTgraph} we have a sequence $r_i\to
  0$ such that $M_i$ is the graph of a function $u_i$ over
  $T_{\lambda_i, \epsilon_i\lambda_i}$ on the region $\{ r > r_i\}$ in the annulus
  $B_1\setminus B_{L^{-B}\rho_0}$, and $|u_i| \leq C_1 a_i
  r^{-\gamma_1}$, for a constant $C_1$ depending on $B$. In addition
  because of \eqref{eq:DTl20} we have the better estimates
  \[ |u_i| \leq C_1 \frac{1}{i} |\lambda_i| r^{-\gamma_1} \]
  on the annuli $B_1\setminus B_{\rho_0}$ and $B_{L^{-B}}\setminus
  B_{L^{-B}\rho_0}$. Up to choosing a subsequence we have $a_i^{-1}
  u_i \to U$ locally uniformly away from the singular set for a Jacobi
  field $U$ on $C\times\mathbf{R}$, but since $a_i i |\lambda_i|^{-1}
  \to \infty$, we have that $U$ vanishes on the annuli $B_1\setminus
  B_{\rho_0}$ and $B_{L^{-B}}\setminus  B_{L^{-B}\rho_0}$. It follows
  that $U$ vanishes on $B_1\setminus B_{L^{-B}\rho_0}$. We can now
  apply the non-concentration estimate,
  Proposition~\ref{prop:nonconc2}, to control $M_i$ on the slightly
  smaller annulus $B_{1/2} \setminus B_{2L^{-B}\rho_0}$. More
  precisely for any $s > 0$ we find that $|u_i| < a_i s$ on the set $\{ r
  > s\}$ for sufficiently large $i$, and so if $\Lambda\in [2,
  L^B/2]$, then 
  \[ D_{\Lambda T_{\lambda_i, \epsilon_i\lambda_i}}( \Lambda M_i) \leq C_2 a_i s \]
  for a constant $C_2$ depending on $B$, for large $i$. Together with the bounds
  \eqref{eq:DTl20} this contradicts \eqref{eq:aimax20} if $s$ is
  sufficiently small.

  We therefore have that $D_{\Lambda T_{\lambda, \epsilon\lambda}}(\Lambda M) < C_B
  \delta |\lambda|$ for all $\Lambda\in [1, L^B]$ once $\delta$ is
  sufficiently small. Using \eqref{eq:Tlambda} for such $\Lambda$ we have $\Lambda
  T_{\lambda, \epsilon\lambda} = T_{\lambda', \epsilon\lambda'}$ where
  $C_2^{-1} < |\lambda / \lambda'| < C_2$ for 
  a $C_2$ depending on $B$. Definition~\ref{defn:Ddefn} then implies
  that if $\delta$ is sufficiently small (depending on $B$ as well as
  the constant $\beta$ that we chose in defining the distance), then
  $\Lambda M$ is a graph over $\Lambda T_{\lambda, \epsilon\lambda}$ for all
  $\Lambda\in [1, L^B]$ (since we are in case (b) in the
  definition). This means that $M$ is a graph over $T_{\lambda, \epsilon\lambda}$ on the
  annulus $B_1\setminus B_{L^{-B}\rho_0}$. 
\end{proof}

Finally we  prove Theorem~\ref{thm:symmetric}.
\begin{proof}[Proof of Theorem~\ref{thm:symmetric}]
  Suppose that $M$ has degree $d_M = \l-\gamma$ as above.

  \bigskip
  \noindent {\em Step 1.} Let $\delta_1 > 0$ be
  small. We first claim that replacing $M$ by a rescaling $\Lambda M$
  for sufficiently large $\Lambda$, we can assume that there is a
  $\lambda_0$ such that $|\lambda_0| < \delta_1$ and
  \[ \label{eq:DTl3} D_{T_{\lambda_0}}(M) < \delta_1|\lambda_0|. \]

  To see this set $a_i = D_{C\times\mathbf{R}}(2^{i}M;
  B_1)$, and write $2^iM$ as the graph of $u_i$ over
  $C\times\mathbf{R}$ on the set
  $B_1\cap \{r > r_i\}$, with $r_i \to 0$. Along a
  subsequence we have $a_i^{-1}u_i \to U$ locally uniformly on $B_1$, away from
  the singular set of $C\times \mathbf{R}$ for a Jacobi field $U$ on
  $C\times\mathbf{R}$.  In particular we have a
  sequence $\tau_i\to 0$ such that on the sets $B_{1/2}\cap \{r > r_i\}$
  (increasing $r_i$ if necessary), we have
  \[ \label{eq:ui'} u_i = a_i U + u_i', \quad \text{ where } |u_i'| < \tau_i a_i.\]
  At the same time, by the definition of the
  degree of $M$, for a suitable subsequence we have
  \[ U = \lambda u_{\l} + U', \]
  where $U'$ is a sum of Jacobi fields with degrees at least $d_M
  + \kappa$ for some $\kappa > 0$. In addition the terms in $U'$ all
  satisfy estimates of the form $|U'|\leq Cr^{-\gamma}$. 

  Note that if $2^iM$ is the graph of $u_i$ over $C\times\mathbf{R}$
  on the set $\{r > r_i\}$, then for any $K > 2$ we have that
  $2^{i+K}M$ is the graph of $2^Ku_i(2^{-K}\cdot)$ on the set $\{r >
  2^Kr_i\}$. We have
  \[ 2^Ku_i(2^{-K}\cdot) &= 2^{K(1 - d_M)} a_i \lambda u_\l +
    a_i 2^K U'(2^{-K}\cdot) + 2^K u_i'(2^{-K}\cdot) \\
    &= 2^{K(1 - d_M)} a_i \lambda u_\l + \widetilde{U}_i. \]
Using \eqref{eq:ui'} and that $U'$ is comprised of terms of degree at least
$d_M+\kappa$, we have
\[ |\widetilde{U}_i| \leq C_1  2^{K(1 - (d_M+\kappa))} a_i
  r^{-\gamma} + 2^K \tau_i a_i, \]
on $B_2\cap \{r > 2^Kr_i\}$, for a suitable $C_1$. We first choose $K$ so large that
\[ C_1 2^{K(1-(d_M+\kappa))} < \delta_1 2^{K(1-d_M)}|\lambda|, \]
and then we let $i$ be sufficiently large so that $2^K\tau_i <
\delta_1  2^{K(1-d_M)} |\lambda|$. Let us also define
\[ \lambda_i = 2^{K(1-d_M)} a_i\lambda. \]
With these choices $2^{i+K}M$ is the graph of $\lambda_i u_\l
+ \widetilde{U}_i$ over $C\times\mathbf{R}$ on the region $B_2\cap\{ r
> 2^Kr_i\}$, where
\[ |\widetilde{U}_i| \leq  2\delta_1|\lambda_i|  r^{-\gamma}.\]
Replacing $2^Kr_i$ by a larger sequence $r_i' \to 0$ if necessary, it
follows that $2^{i+K}M$ is contained between the graphs of $\pm
D(r_i') r^{-\gamma}$ over $T_{\lambda_i}$ on the region $(B_2\setminus
B_{\rho_0/2})\cap \{r > r_i'\}$, where
\[ \label{eq:DTl1}  D(r_i') \leq C_2
  \delta_1 |\lambda_i|, \]
for a constant $C_2$ depending on the cone $C$.

In order to use \eqref{eq:DTl1} in the non-concentration estimate, we
also need to bound the distance from $T_{\lambda_i}$ to $2^{i+K}M$ on
all of the annulus $B_2\setminus B_{\rho_0/2}$. For this note first that
using the scaling property of the distance, we have
\[ D_{C\times\mathbf{R}}(2^{i+K}M; B_2) \leq C_K a_i, \]
for a constant $C_K$ depending on $K$ and $\gamma$, and so by
Lemma~\ref{lem:triangle} we have
\[ \label{eq:DTl2} D_{T_{\lambda_i}}(2^{i+K}M; B_2\setminus B_{\rho_0/2}) &\leq C_3(
  C_Ka_i + |\lambda_i|) \leq C_K a_i, \]
for a larger constant $C_K$. Let us apply the non-concentration
estimate, Proposition~\ref{prop:nonconc2} with
\[ s = \delta_1 C_K^{-1} 2^{K(1-d_M)} |\lambda|. \]
Using \eqref{eq:DTl1} and \eqref{eq:DTl2} we find that for
sufficiently large $i$ we have
\[ D_{T_{\lambda_i}}(2^{i+K}M; B_1\setminus B_{\rho_0}) \leq C_4
  \delta_1 |\lambda_i|. \]
Replacing $\delta_1$ by $C_4^{-1}\delta_1$ and $M$ by $2^{i+K}M$,
we obtain the required estimate \eqref{eq:DTl3}. At the same time
$\lambda_i \to 0$ as $i\to\infty$. 

\bigskip
\noindent {\em Step 2.}
We will next apply Corollary~\ref{cor:decay2} iteratively. If
$\delta_1 < \delta_0$ for the $\delta_0$ from  
  Corollary~\ref{cor:decay2}, then 
  there is an $\epsilon_1$ with
  \[ \label{eq:l1l0}|\epsilon_1| \leq C_1\delta_1 \]
  such that
  \[ D_{L^B T_{\lambda_0, \epsilon_1\lambda_0}} (L^BM) \leq \frac{1}{2}
    (L^B)^{1-d_M}\delta_1|\lambda_0|. \]
  As long as $\epsilon_1$ is sufficiently small, we can  apply
  Corollary~\ref{cor:decay2} again to find $\epsilon_2$ with
  \[ |\epsilon_2 - \epsilon_1| \leq C_1 \frac{1}{2} \delta_1, \]
  satisfying
  \[ D_{L^{2B} T_{\lambda_0, \epsilon_2\lambda_0}}(L^{2B}M) \leq
    \frac{1}{4} L^{2B(1-d_M)} \delta_1 |\lambda_0|. \]
   As long as each $\epsilon_i$ is sufficiently small, we can continue
   this procedure to obtain $\epsilon_3, \epsilon_4,\ldots$ satisfying
   \[ |\epsilon_{i+1} - \epsilon_i| \leq C_1 \frac{1}{2^i} \delta_1 \]
   and
   \[ D_{L^{iB} T_{\lambda_0, \epsilon_i\lambda_0}}(L^{iB}M) \leq
     \frac{1}{2^i} L^{iB (1-d_M)} \delta_1 |\lambda_0|. \]
   It follows that we always have $|\epsilon_i| \leq
   2C_1\delta_1$, so as long as $\delta_1$ is
   chosen sufficiently small, we can keep applying Corollary~\ref{cor:decay2}. From
   Lemma~\ref{lem:triangle} we have
   \[ D_{L^{iB} T_{\lambda_0}}(L^{iB}M) &\leq C_1\left( \frac{1}{2^i}
       L^{iB(1-d_M)} \delta_1|\lambda_0| + |\epsilon_i|
       L^{iB(1-d_M)} |\lambda_0|\right) \\
     &\leq C_1' L^{iB(1-d_M)} \delta_1 |\lambda_0|,
   \]
   for all $i\geq 0$. 
   If $\delta_1$ is chosen sufficiently small, then using
   Lemma~\ref{lem:graphT} it follows from this that $M$ is
   graphical over $T_{\lambda_0}$ near the origin. 
\end{proof}

\end{document}